\documentclass[english]{article}
\usepackage{amsmath,amssymb,latexsym,amsthm,mathrsfs,appendix}
\usepackage{color}
\usepackage{scalerel}
\usepackage{graphicx}
\usepackage{caption}
\usepackage{subcaption}
\usepackage{mathtools}

\date{}

\def\theenumi{\arabic{enumi}}

\def\theenumii{\alph{enumii}}
\def\p@enumii{\theenumi.}

\def\theenumiii{\arabic{enumiii}}
\def\p@enumiii{(\theenumi)(\theenumii)}

\def\p@enumiv{\p@enumiii.\theenumiii}
\newtheorem{theorem}{Theorem}[section]

\newtheorem{corollary}[theorem]{Corollary}

\newtheorem{proposition}[theorem]{Proposition}
\newtheorem{lemma}[theorem]{Lemma}

\theoremstyle{definition}

\newtheorem{definition}[theorem]{Definition}
\newtheorem{remark}[theorem]{Remark}

\newtheorem{assumption}[theorem]{Assumption}

\newcommand{\ra}{\rightarrow}
\DeclareMathOperator*{\esssup}{ess\,sup}
\DeclareMathOperator*{\essinf}{ess\,inf}
\newcommand{\tens}[1]{%
	\mathbin{\mathop{\otimes}\limits_{#1}}%
}

 \makeatother

\usepackage{babel}
\begin{document}

\title{Time regularity for local weak solutions of the heat equation on local Dirichlet spaces}

\author{Qi Hou\thanks{%
Partially supported by NSF grant DMS  1404435 and DMS 1707589} \,
and Laurent Saloff-Coste\thanks{ Partially supported by NSF grant DMS  1404435 and DMS 1707589} \\
{\small Department of Mathematics}\\
{\small Cornell University}  }
\maketitle

\begin{abstract}
We study the time regularity of local weak solutions of the heat equation in the context of local regular symmetric Dirichlet spaces.  Under two basic and rather minimal assumptions, namely, the existence of certain cut-off functions and a very weak $L^2$ Gaussian type upper-bound for the heat semigroup, we prove that the time derivatives of a local weak solution of the heat equation are themselves local weak solutions. This applies, for instance, to local weak solutions of parabolic equations with uniformly elliptic symmetric divergence form second order operators with measurable coefficients. We describe some applications 
to the structure of ancient local weak solutions of such equations which generalize recent results of \cite{app1} and \cite{app2analyticity}.
\end{abstract}

\section{Introduction}   \setcounter{equation}{0} 
 When $-P$ is the infinitesimal generator of a self-adjoint strongly continuous semigroup of operators $H_t=e^{-tP}$ acting on a Hilbert space $\mathbf H$, spectral theory implies the time regularity of any (global) solution $u(t)= H_tu_0$  of the equation $(\partial_t+P)u=0$ with initial data $u_0\in \mathbf H$.  When $\mathbf H=L^2(X,m)$ and $-P$ is associated with a bilinear form $\mathcal E$ so that $\mathcal E(f,g)=\int fPg\, dm$ for enough functions $f,g$, it is often very useful to consider the concept of local weak solution of the equation $(\partial_t+P) u=0$ in $I\times \Omega \subset \mathbb R\times X$, in some appropriate sense. Such definition goes roughly as follows. A local weak solution $u$ is a function defined on $I\times \Omega$ which MUST belong (locally) to a certain  function space $\mathcal F$ (in the most classical case, $\mathcal F$ is related to the Sobolev space) and satisfies 
\begin{equation} \label{ws}
-\int_{I\times \Omega} u \partial_t \phi \,dtdm +\int_I \mathcal E(u,\phi) dt=0
\end{equation}
for all ``test functions'' $\phi$ compactly supported in $I\times \Omega$. The precise nature of the space $\mathcal F$ and of the space of test functions to be used here are an important part of such definition.  When dealing with such a definition, the time regularity of a local weak solution is not automatic. Formally, one expects 
the time derivative of a local weak solution to be a local weak solution of (\ref{ws}), but the problem lies with the a priori requirement that $v=\partial _tu$ 
belongs locally to the space $\mathcal F$.   

Consider the classical case when $P$ is a symmetric locally uniformly elliptic second order operator with measurable coefficients  $(a_{ij}(x))_{i,j=1}^n$
so that $$\mathcal E(f,g)= \int \sum_{i,j} a_{ij}(x) \partial_if(x)\partial_jg(x)\,dx.$$ 
The basic assumption, local uniform ellipticity, means that for any compact subset $K$ there are $\epsilon_K>0$ and $C_K<\infty$ such that
$$\max_{i,j}\sup_K\{|a_{ij}|\}\le C_K \mbox{ and } \sum_{i,j}a_{ij}\xi_i\xi_j\ge \epsilon_K \|\xi\|_2^2,\,\forall\xi=(\xi_i)_1^n .$$ 
A local weak solution of $(\partial_t+P)u=0$ in $(a,b)\times \Omega$ is an element $u \in L_{\mbox{\tiny loc}}^2((a,b)\ra W_{\mbox{\tiny loc}}^{1,2}(\Omega)) $ such that
$$-\int _a^b\int_\Omega u(t,x) \partial_t\phi (t,x) dxdt+\int_a^b\int_\Omega \sum_{i,j} a_{ij}(x) \partial_i u(t,x) \partial_j \phi (t,x) \,dx dt=0$$
for all functions $\phi \in C^{\infty}((a,b)\times \Omega)$ with compact support in $(a,b)\times \Omega$.  

One consequence of the general results proved in this paper is that the iterated time derivatives $v_{k}(t,x)=\partial^k_tu(t,x)$ of any local weak solution $u$ of the equation above are themselves in $L_{\mbox{\tiny loc}}^2((a,b)\ra W_{\mbox{\tiny loc}}^{1,2}(\Omega))$ and are local weak solutions of the same equation in $(a,b)\times \Omega$. This follows for instance from the following more general  theorem. In this statement we assume that $(X,m)$ is a locally compact separable Hausdorff space and $m$ is a positive radon measure with full support. 

\begin{theorem} Assume $(\mathcal E,\mathcal F)$ is a symmetric strictly local regular Dirichlet form on $L^2(X,m)$ whose intrinsic pseudo-metric is a continuous metric which induces the topology of $X$. For any local weak solution $u$ of the associated heat equation  in $(a,b)\times \Omega$, the iterated time derivatives $v_{k}=\partial_t^k u$ are themselves local weak solutions of the same heat equation in $(a,b)\times \Omega$.
\end{theorem}

Although this theorem excludes fractal sets such as the Sierpinski Gasket and the Sierpinski Carpet (on such examples, the intrinsic pseudo-distance is 
identically equal to $0$) as well as some infinite dimensional examples (e.g., on the  infinite dimensional torus in cases when the intrinsic pseudo-distance  is infinite almost surely), these cases are in fact also covered by our more general results.  Indeed, only two related types of assumptions play a key part in our results:
\begin{itemize}
\item  the existence of good cut-off functions (in a sense that is somewhat weaker than most conditions of this type that exist in the literature);  
\item  a very weak $L^2$-Gaussian bound, namely, the fact that for any $a>0$\ and any integer $k=0,1,2,\dots$, for any disjoint compact sets $V_1,V_2$,
$$  t^{-a}\sup_{\phi_1,\phi_2}\int _X \phi_2 \partial_t^k H_t\phi_1  dm  \ra 0\ \ (\mbox{as }t\rightarrow 0)$$
where the sup is taken over all  functions $\phi_1$, $\phi_2$ supported respectively in $V_1,V_2$ and with $L^2$-norm at most $1$.
\end{itemize}

As an application of our results, we extend two recent structure theorems regarding ancient weak solutions, \cite{app1,App2LZ,app2analyticity}. The first result of this type describes very general conditions under which any ancient (local) weak solutions with ``polynomial growth'' must be of the form $u(t,x)=\sum_{k=1}^d t^k u_k(x)$ where all $u_k$\ are of polynomial growth, $u_d$ is a harmonic function, and other $u_k$'s satisfy $\Delta u_k=(k+1)u_{k+1}$ in a weak sense. The integer $d$ is related to the given growth degree of $u$.  The second result  describes very general conditions under which any ancient weak solution of ``exponential growth''  is real analytic in time.  

The general approach we take is to utilize the \textit{heat semigroup} to study the time regularity properties of local weak solutions of the heat equation. The basic idea of deriving hypoelliticity type results from properties of the heat semigroup goes back to Kusuoka and Stroock's paper \cite{Kusuoka} which is written in the context of the heat equation associated with H\"ormander sums of squares of vector fields on Euclidean spaces. It was also implemented in \cite{hypoellipticity} to study distributional solutions of the Laplace equation on the infinite dimensional torus and other infinite dimensional compact groups.

This approach differs from the classical hypoellipticity viewpoint in the primary role it gives to the fundamental solution of the heat equation (here, in the very minimal form of the heat semigroup itself), while traditional studies of hypoellipticity treat all solutions equally and are then used to deduce the basic regularity of the fundamental solution. In this paper we generalize this heat semigroup approach on hypoelliticity to the general setting of Dirichlet spaces on metric measure spaces. One natural goal is to cover rougher structures that make smoothness more elusive. Here, we treat a purely $L^2$-theory.
In a sequel of this paper, we will further utilize this method to study the local boundedness and continuity properties of local weak solutions of the heat equation (the $L^\infty$-type properties) under additional assumptions. 

This work is organized as follows. In Section 2, we introduce our general Dirichlet space setup and define the relevant notion of local weak solutions. In Section 3, we introduce and discuss our two main hypotheses, the existence of certain cut-off functions and the notion of a very weak $L^2$-Gaussian bound. We state in Section 4 the main theorems proved in this paper and give a sketch of the proof of the main result that conveys the main ideas while avoiding many long necessary computations and technical details.  In Section 5 we give a complete proof of the main theorems stated in Section 4. Section 6 is devoted to the results concerning the structure of ancient (local weak) solutions. Section 7 discusses briefly several typical examples that illustrate the results of this paper in a variety of different contexts. Lastly Section 8 provides tools to verify that the very weak $L^2$-Gaussian bound is satisfied under rather weak assumptions involving the
existence of cut-off functions, as well as the proofs for some lemmas regarding cut-off functions.

We remark that, in this paper, the Dirichlet forms we treat are symmetric, and are not time dependent. The independence on time is a crucial assumption for us, as we take advantage of the smoothness of the heat semigroup in time. The symmetry assumption can probably be replaced by some form of the sector condition but we leave this to a further study. For related but different results (under stronger assumptions) for nonsymmetric or time dependent Dirichlet spaces, we refer to \cite{Sturm2,Sturm3} and \cite{Lierl1,Lierl2,Lierl3}.

\section{Dirichlet spaces and local weak solutions} \setcounter{equation}{0} 
\subsection{Dirichlet spaces} We briefly review some concepts and properties related to Dirichlet forms. A classical reference for (symmetric) Dirichlet forms is \cite{Fukushima}. Let $(X,d,m)$\ be a metric measure space where $X$\ is locally compact separable, $m$\ is a Radon measure on $X$\ with full support, and $d$\ is some metric on $X$\ that we omit writing in the rest of the paper since we do not use it explicitly. Let $(\mathcal{E},\mathcal{F})$\ be a symmetric regular local Dirichlet form on $L^2(X,m)$, $\mathcal{F}$\ denotes the domain of $\mathcal{E}$. By definition, a (symmetric) Dirichlet form is a closed symmetric form that further satisfies the Markov property. Here the term symmetric form refers to any symmetric, nonnegative definite, densely defined bilinear form. 
The domain $\mathcal{F}$\ equipped with the $\mathcal{E}_1$\ norm
\begin{eqnarray*}
	||f||_{\mathcal{E}_1}:=\left(\mathcal{E}(f,f)+\int_Xf^2\,dm\right)^{1/2}
\end{eqnarray*}
is a Hilbert space. A Dirichlet form $(\mathcal{E},\mathcal{F})$\ is called regular, if $C_c(X)\cap \mathcal{F}$\ is dense in $C_c(X)$\ in the sup norm and dense in $\mathcal{F}$\ in the $\mathcal{E}_1$\ norm. Any subset $\mathcal{C}\subset C_c(X)\cap \mathcal{F}$\ that is dense in these two senses is called a core of $\mathcal{E}$. A Dirichlet form $(\mathcal{E},\mathcal{F})$\ is called local, if $\mathcal{E}(u,v)=0$\ for $u,v\in \mathcal{F}$\ whenever $\mbox{supp}\{u\}$\ and $\mbox{supp}\{v\}$\ are disjoint and compact.

Regular Dirichlet forms satisfy the Beurling-Deny decomposition formula; as a corollary, any regular local Dirichlet form $(\mathcal{E},\mathcal{F})$\ admits the decomposition formula
\begin{eqnarray*}
	\mathcal{E}(u,v)=\int_Xd\Gamma(u,v)+\int_Xuv\,dk.
\end{eqnarray*}
Here $dk$\ is a positive Radon measure, called the killing measure. $\Gamma$\ stands for the \textit{energy measure}, which is a (Radon) measure-valued bilinear form first defined for any $u$\ in $\mathcal{F}\cap L^\infty(X)$\ by
\begin{eqnarray*}
	\int_X \phi\,d\Gamma(u,u):=\mathcal{E}(\phi u,u)-\frac{1}{2}\mathcal{E}(u^2,\phi)
\end{eqnarray*} 
for any $\phi\in \mathcal{F}\cap C_c(X)$, then extended by polarization for arbitrary pairs of $u,v\in \mathcal{F}\cap L^\infty(X)$. For $u\in \mathcal{F}$, the energy measure of $u$ is the limit of the energy measures associated with the truncation functions $\left(\left(u\wedge n\right)\vee -n\right)$\ as $n\rightarrow \infty$.

As a generalization of the classical energy integral $\int_{\mathbb{R}^n} \nabla u\cdot \nabla v\,dx$\ in $\mathbb{R}^n$, that is, intuitively as a measure given by gradients, the energy measure satisfies the following properties. We do not specify quasi-continuous modifications of functions.
\begin{itemize}
\item (Leibniz rule) For any $u,v,w\in \mathcal{F}$\ with $uv\in \mathcal{F}$\ (e.g. $u,v\in \mathcal{F}\cap L^\infty$),
\begin{eqnarray*}
	d\Gamma(uv,w)=u\,d\Gamma(v,w)+v\,d\Gamma(u,w).
\end{eqnarray*}
\item (Chain rule) For any $u,v\in \mathcal{F}$, any $\Phi\in C^1(\mathbb{R})$\ with bounded derivative and satisfies $\Phi(0)=0$,
\begin{eqnarray*}
	d\Gamma(\Phi(u),v)=\Phi'(u)\,d\Gamma(u,v).
\end{eqnarray*}
\item (Cauchy-Schwartz inequality) For any $f,g,u,v\in \mathcal{F}\cap L^\infty$\ (more generally, when $u,v\in \mathcal{F}\cap L^\infty$\ and $f\in L^2(X,\Gamma(u,u))$, $g\in L^2(X,\Gamma(v,v))$)
\begin{eqnarray}
	\int fg\,d\Gamma(u,v)&\leq& \left(\int f^2\,d\Gamma(u,u)\right)^{1/2}\left(\int g^2d\Gamma(v,v)\right)^{1/2}\notag\\
	&\leq& \frac{C}{2}\int f^2\,d\Gamma(u,u)+\frac{1}{2C}\int g^2d\Gamma(v,v).\label{CSineq}
\end{eqnarray}
This last inequality holds for any $C>0$. The corresponding measure version is
\begin{eqnarray*}
	|fg|\,d|\Gamma(u,v)|\leq \frac{C}{2}f^2\,d\Gamma(u,u)+\frac{1}{2C}g^2\,d\Gamma(v,v).
\end{eqnarray*}
\item (Strong locality) For any $u,v\in \mathcal{F}$, if on some precompact open set $U\Subset X$, $v\equiv C$\ for some constant $C$, then
\begin{eqnarray*}
	1_U\,d\Gamma(u,v)=0.
\end{eqnarray*}
\end{itemize}
Any Dirichlet form $(\mathcal{E},\mathcal{F})$\ is associated with a corresponding Markov semigroup $(H_t)_{t>0}$, an (infinitesimal) generator $-P$ with dense domain $\mathcal{D}(P)$, and a Markov resolvent $(G_\alpha)_{\alpha>0}$ (in the sense of \cite[page 15]{Fukushima}). The semigroup $H_t$ and resolvent $G_\alpha$\ have domain $L^2(X,m)$; the domain $\mathcal{D}(P)$ of $-P$ is dense in $\mathcal{F}$\ w.r.t. the $\mathcal{E}_1$\ norm. These are self-adjoint operators. By spectral theory, $P$ has a spectral resolution $(E_\lambda)_{\lambda\geq 0}$ such that, for any $t>0$,
\begin{eqnarray*}
	PH_t=\int_{0}^{\infty}\lambda e^{-\lambda t}\,dE_\lambda.
\end{eqnarray*}
As a consequence, for any $k\in \mathbb{N}$,
\begin{eqnarray*}
	\left|\left|\partial_t^kH_t\right|\right|_{L^2(X)\rightarrow L^2(X)}=\left|\left|P^kH_t\right|\right|_{L^2(X)\rightarrow L^2(X)}\leq \left(k/et\right)^k.
\end{eqnarray*}
For any function $u_0\in L^2(X,m)$, $u(t,x):=H_tu_0(x)$\ is smooth in $t>0$, and solves $$\partial_t u=-Pu$$ in the strong sense (i.e., $\lim\limits_{h\rightarrow 0}\frac{u(t+h,\cdot)-u(t,\cdot)}{h}=-Pu(t,\cdot)$\ in $L^2(X,m)$).
 
Given the notations above, our main goal in this section is to define local weak solutions of the heat equation (with appropriate right-hand side $f$)
\begin{eqnarray*}
	(\partial_t+P)u=f.
\end{eqnarray*}
\subsection{Function spaces associated with $(\mathcal{E},\mathcal{F})$}
To properly discuss candidate functions for local weak solutions, and later their properties, we first introduce some function spaces associated with $(\mathcal{E},\mathcal{F})$. In choosing notations for these function spaces, we mostly follow \cite{Sturm2} with a few exceptions that we will remark on later. Among these function spaces there are two prevalent types, one type consists of functions that have compact support (all with subscript ``$c$''); the other type of functions that locally satisfy the required properties (all with subscript ``$\mbox{loc}$'').

Recall that the inclusion $\mathcal{F}\subset L^2(X)$\ is dense. Equating $L^2(X)$\ with its dual w.r.t. the $L^2$\ inner product, we get the Hilbert triple
\begin{eqnarray*}
\mathcal{F}\subset L^2(X)\subset \mathcal{F}'
\end{eqnarray*}
in which the inclusions are dense and continuous. Intuitively, the ``$\sim_c$'' spaces are on the ``$\mathcal{F}$'' end, and the ``$\sim_{\scaleto{\mbox{loc}}{5pt}}$'' spaces are on the ``$\mathcal{F}'$'' (dual space) end. We consider the dual spaces of ``$\sim_c$'' spaces too.

We now give precise definitions of these spaces, organized in pairs, starting with the following two pairs:
\begin{itemize}
\item $\mathcal{F}_c(X):=\left\{f\in \mathcal{F}\,|\, f\ \mbox{has compact (essential) support}\right\}$;
\item $\mathcal{F}_{\scaleto{\mbox{loc}}{5pt}}(X):=\left\{f\in L^2_{\scaleto{\mbox{loc}}{5pt}}(X)\,|\,\right.\\[0.05in]
\left. \forall \mbox{compact }K\subset X\ \exists f^\sharp\in \mathcal{F}\ \mbox{s.t.}\ f^\sharp=f\ \mbox{a.e. on } K\right\}$.
\end{itemize}
For any open subset $U\subset X$, define
\begin{itemize}
\item $\mathcal{F}_c(U):=\left\{f\in \mathcal{F}\,|\,f\ \mbox{has compact (essential) support in }U\right\};$
\item $\mathcal{F}_{\scaleto{\mbox{loc}}{5pt}}(U):=\left\{f\in L^2_{\scaleto{\mbox{loc}}{5pt}}(U)\,|\,\right.\\[0.05in]
\left.\forall \mbox{compact }K\subset U\ \exists f^\sharp\in \mathcal{F}\  \mbox{s.t. }f^\sharp=f\ \mbox{a.e. on } K\right\}$.
\end{itemize}
\begin{remark}
	\label{FS2}
	When $U\neq X$, by definition, there is an injection $i: \mathcal{F}_c(U)\hookrightarrow\mathcal{F}_c(X)$, and clearly $\mathcal{F}_{\scaleto{\mbox{loc}}{5pt}}(X)\hookrightarrow\mathcal{F}_{\scaleto{\mbox{loc}}{5pt}}(U)$\ by restriction to $U$. Note, however, that $\mathcal{F}_{\scaleto{\mbox{loc}}{5pt}}(U)$\ is not a subspace of $\mathcal{F}_{\scaleto{\mbox{loc}}{5pt}}(X)$.
\end{remark}

Fix any open set $U\subset X$\ and open interval $I=(a,b)\subset \mathbb{R}$. $-\infty\leq a<b\leq \infty$. In the sequel, when there is no ambiguity, we use the notation $u^t(\cdot)$\ as an abbreviation for $u(t,\cdot)$. That is, for any fixed $t$, consider $u(t,y)$\ as a function of $y$, denoted by $u^t$. Note that this is not any power of $u$\ or time derivative of $u$; the time derivative is denoted by $\partial_t u$.
Consider the following function spaces involving time and space associated to $\mathcal{E}$. In defining these spaces, we switch freely between two viewpoints where elements in these spaces are viewed (1) as functions of time and space; (2) as functions on the time interval $I$\ with values in some (spatial) function space. The rigorous setup for the latter viewpoint is the theory of Bochner integrals, for which we refer to \cite{Wloka}.

First, we fix the notation for the ``base space''
\begin{itemize}
    \item $\mathcal{F}(I\times X):=L^2(I\rightarrow \mathcal{F})$.
\end{itemize}

\begin{remark}
	\label{FS3}
	$L^2(I\rightarrow \mathcal{F})$\ is the completion of the space of bounded continuous functions from $I$\ to $\mathcal{F}$, $C_b(I\rightarrow \mathcal{F})$, under the $\left|\left|\cdot\right|\right|_{L^2(I\rightarrow \mathcal{F})}$\ norm $$\left|\left|u\right|\right|_{L^2(I\rightarrow \mathcal{F})}=\left(\int_I\left|\left|u^t\right|\right|_{\mathcal{E}_1}^2\,dt\right)^{1/2}.$$  
	The space $C_c^\infty(I\rightarrow \mathcal{F})$\ of smooth compactly supported functions from $I$\ to $\mathcal{F}$\ is also dense in $L^2(I\rightarrow \mathcal{F})$\ w.r.t. the $\left|\left|\cdot\right|\right|_{L^2(I\rightarrow \mathcal{F})}$\ norm. We use the notation $\mathcal{F}(I\times X)$ to clarify the use of notations $\mathcal{F}_c(I\times U)$\ and $\mathcal{F}_{\scaleto{\mbox{loc}}{5pt}}(I\times U)$\ for function spaces defined below. See 
	also Remark \ref{remarkonequivofdefs}.
\end{remark}
Based on the ``base space'' $\mathcal{F}(I\times X)$, for any open subset $U\subset X$, define
\begin{itemize}
    \item $\mathcal{F}_c(I\times U):=\left\{u\in \mathcal{F}(I\times X)\, |\, u\ \mbox{is compactly supported in }I\times U\right\}$;
    \item 
	$\mathcal{F}_{\scaleto{\mbox{loc}}{5pt}}\left(I\times U\right):=\left\{u\in L^2_{\scaleto{\mbox{loc}}{5pt}}\left(I\times U\right)\, |\, \right.\\[0.05in]
	\left.\forall I'\Subset I,\ \forall U'\Subset U,\ \exists u^\sharp\in \mathcal{F}\left(I\times X\right)\ \mbox{s.t. }u^\sharp=u\ \mbox{on }I'\times U'\ \mbox{a.e.}\right\}$.
\end{itemize}

The first two spaces $\mathcal{F}(I\times X)$\ and $\mathcal{F}_c(I\times U)$\ are subspaces of $L^2(I\times X)$\ and $L^2(I\times U)$, respectively. We identify the $L^2$\ spaces with their own duals (under the $L^2$\ inner product), and denote the dual spaces of $\mathcal{F}(I\times X)$, $\mathcal{F}_c(I\times U)$\ under the $L^2$-inner-product by $\left(\mathcal{F}(I\times X)\right)'$, $\left(\mathcal{F}_c(I\times U)\right)'$.
\begin{remark}
	$\left(\mathcal{F}(I\times X)\right)'=\left(L^2(I\rightarrow \mathcal{F})\right)'=L^2(I\rightarrow \mathcal{F}')$.
\end{remark}
\begin{remark}
	\label{remarkonequivofdefs}
	Here our notations are slightly different from the ones used in other papers (e.g. \cite{Sturm2,Pavel}). In the definition of $\mathcal{F}(I\times X)$, we do not require the functions to further be in $W^{1,2}(I\rightarrow \mathcal{F}')$\ (functions with time derivatives in the distribution sense that belong to $L^2(I\rightarrow \mathcal{F}')$). The reason we consider the function spaces defined above instead of the ones obtained by taking the intersection with $W^{1,2}(I\rightarrow \mathcal{F}')$, is to put minimum assumptions on the definition of local weak solutions. Under our definition and hypotheses, such local weak solutions automatically satisfy better properties. In particular, we explain at the end of this section that under a very natural assumption on existence of cut-off functions, and when we require the right-hand side $f$\ to be locally in $L^2(I\rightarrow \mathcal{F}')$, our choice of definition of local weak solutions agrees with the definition used in other papers. This is proved by adapting the proof of Lemma 1 in \cite{Nate}.
\end{remark}
To include more time derivatives we introduce the following notations for function spaces
\begin{itemize}
    \item $\mathcal{F}^k(I\times X):=W^{k,2}(I\rightarrow \mathcal{F})$;
    \item $\mathcal{F}^k_c(I\times U):=\left\{u\in \mathcal{F}^k(I\times X)\, |\, u\ \mbox{is compactly supported in }I\times U\right\}$;
    \item $\mathcal{F}^k_{\scaleto{\mbox{loc}}{5pt}}(I\times U):=\left\{u\in L^2_{\scaleto{\mbox{loc}}{5pt}}(I\times U)\, |\,\right.\\[0.05in]
	\left.\forall I'\Subset I,\ \forall U'\Subset U,\ \exists u^\sharp\in \mathcal{F}^k\left(I\times X\right)\, \mbox{s.t.}\ u^\sharp=u\ \mbox{on }I'\times U'\ a.e.\right\}$.
\end{itemize}

\begin{remark}
	\label{defoflocallyinspace}
	In general, we say a function $u$\ is locally in some function space if for any compact set, there exists a function $w$\ in the said function space such that $w=u$\ $a.e.$\ on the compact set.
\end{remark}
\subsection{Notion of local weak solutions}
For any symmetric local regular Dirichlet form $(\mathcal{E},\mathcal{F})$\ on $L^2(X,m)$, we define the following notion of local weak solutions of the associated heat equation (below $-P$\ and $(H_t)_{t>0}$\ are the corresponding generator and semigroup as before). 	

\begin{definition}[local weak solution]
	\label{soldef1}
	Let $U\subset X$\ be an open subset and $I\subset \mathbb{R}$\ be an open interval. Let $f$\ be a function locally in $L^2(I\rightarrow \mathcal{F}')$. We say $u$\ is a \textit{local weak solution} of the heat equation $(\partial_t+P)u=f$\ on $I\times U$, if $u\in \mathcal{F}_{\scaleto{\mbox{loc}}{5pt}}(I\times U)$, and for any $\varphi\in \mathcal{F}_c(I\times U)\cap C_c^\infty(I\rightarrow \mathcal{F})$,
	\begin{eqnarray}
	\label{localweaksol}
	-\int_I\int_X\ u\cdot \partial_t\varphi\,dm dt+\int_I\mathcal{E}(u,\varphi)\,dt\ =\int_I<f,\,\varphi>_{\mathcal{F}',\mathcal{F}}\,dt.
	\end{eqnarray}
\end{definition}	
Here $u$\ in the integral is understood as $u^\sharp$\ (relative to the support of $\varphi$) as in the definition for $\mathcal{F}_{\scaleto{\mbox{loc}}{5pt}}(I\times U)$. We take this convention throughout this paper. Note that $\mathcal{E}\left(u,\varphi\right)$\ is well-defined (independent of the choice of $u^\sharp$) by the local property of $\mathcal{E}$. $<\cdot,\cdot>_{\mathcal{F}',\mathcal{F}}$\ stands for the pairing between elements in $\mathcal{F}'$\ and $\mathcal{F}$.

We remark that we can define local weak solutions for more general right-hand side $f$, e.g., $f\in \left(\mathcal{F}_c(I\times U)\right)'$. But in the propositions and theorems in this paper we always put more restrictions on $f$\ than $f$\ locally in $L^2(I\rightarrow \mathcal{F}')$, moreover the results are interesting even for the case $f\equiv 0$, so here in the definition we do not aim to consider the most general right-hand side. With this choice, Definition \ref{soldef1} will be shown to be equivalent to the following variant, under a natural assumption on the existence of certain cut-off functions. As mentioned in Remark \ref{remarkonequivofdefs}, the following definition is often adopted in the literature.
\begin{definition}[local weak solution, variant]
	\label{soldef2}
	Let $U,I,f$\ be as in Definition \ref{soldef1}. Let $u$\ be a function locally in $L^2(I\rightarrow \mathcal{F})\cap W^{1,2}(I\rightarrow \mathcal{F}')$. $u$\ is called a local weak solution of the heat equation $(\partial_t+P)u=f$, if for any $\varphi$\ in $L^2(I\rightarrow \mathcal{F})\cap W^{1,2}(I\rightarrow \mathcal{F}')$\ with compact support in $I\times U$, for any $J\Subset I$,
	\begin{eqnarray*}
	\int_J\int_X<\partial_tu,\,\varphi>_{\mathcal{F}',\mathcal{F}}\,dmdt+\int_J\mathcal{E}(u,\varphi)\,dt=\int_J<f,\,\varphi>_{\mathcal{F}',\mathcal{F}}\,dt.\label{equivdef}
	\end{eqnarray*}
\end{definition}
Under a natural assumption on existence of some type of cut-off functions to be introduced below (Assumption \ref{cut-offgeneralassumption}), the two notions of local weak solutions defined above agree.
	
Note that in general,
$$\mathcal{F}_c(I\times U)\cdot \mathcal{F}_{\scaleto{\mbox{loc}}{5pt}}(I\times U)\nsubseteq \mathcal{F}_c(I\times U),$$
roughly because $\mathcal{F}$\ is not an algebra. What we want to assume is that there is a subset of $\mathcal{F}_c(I\times U)\cap C(I\times U)$\ that contains enough functions, each of which brings functions in $\mathcal{F}_{\scaleto{\mbox{loc}}{5pt}}(I\times U)$\ to $\mathcal{F}_{c}(I\times U)$\ by multiplication (these can be thought of as cut-off functions with some nice properties). Here $C(I\times U)$\ is the space of continuous functions on $I\times U$. We denote this subset of cut-off functions by $\mathfrak{C}(I\times U)$. Observe that we just need the existence of an analogous subset $\mathfrak{C}(U)\subset \mathcal{F}_c(U)\cap C(U)$, and then to construct $\mathfrak{C}(I\times U)$, we take products of functions in $\mathfrak{C}(U)$\ with standard cut-off functions in $C^\infty_c(\mathbb{R})$. The following assumption makes precise what we want to require from the set $\mathfrak{C}(U)\subset \mathcal{F}_c(U)\cap C(U)$.
\begin{assumption}
	\label{cut-offgeneralassumption}
	There exists a subset $\mathfrak{C}(U)\subset \mathcal{F}_{c}(U)\cap C(U)$\ such that
	\begin{itemize}
	\item[(i)] for any pair of open sets $V\Subset U\Subset X$, there exists a function $\varphi\in \mathfrak{C}(U)$\ such that $\varphi=1$\ on $V$, $\mbox{supp}\{\varphi\}\subset U$;
	\item[(ii)] for any $\varphi\in \mathfrak{C}(U)$, any $u\in \mathcal{F}_{\scaleto{\mbox{loc}}{5pt}}(U)$, the product $\varphi u\in \mathcal{F}_{c}(U)$.
	\end{itemize}
\end{assumption}
\begin{remark}
The requirement (i) in Assumption \ref{cut-offgeneralassumption} is standard and easily fulfilled when the Dirichlet form is regular. The requirement (ii) is nontrivial. In general, only the products of functions in $\mathcal{F}\cap L^\infty(X)$\ are guaranteed to belong to $\mathcal{F}$.
\end{remark}
We now state the equivalence of the two definitions for local weak solutions.
\begin{lemma}[Equivalence of definitions of local weak solutions]	
	Under Assumption \ref{cut-offgeneralassumption}, Definition \ref{soldef1} is equivalent to Definition \ref{soldef2}.
\end{lemma}
\begin{proof}
The proof follows essentially that of \cite[Lemma 1]{Nate}.    
\end{proof}

\section{Main hypotheses} \setcounter{equation}{0}
\subsection{Assumption on existence of cut-off functions}
For a pair of open sets $V\Subset U\Subset X$, by a cut-off function for the pair $V\subset U$\ we mean a function $\eta\in \mathcal{F}\cap C(X)$\ in between $0$\ and $1$\ such that $\eta=1$\ on $V$\ and $\mbox{supp}\{\eta\}\subset U$. Such cut-off functions always exist for any pair of open sets $V\subset U$ in a regular Dirichlet space, see \cite{Fukushima}. For results in this paper what we need is the existence of cut-off functions that further have controlled energy, we explain what this means in the following assumption.
\begin{assumption}[existence of nice cut-off functions]
	\label{assumptionbothcases}
	There exists some topological basis $\mathcal{T}\mathcal{B}$\ of $X$\ such that for any pair of open sets $V\Subset U$, $U,V\in \mathcal{T}\mathcal{B}$, for any $0<C_1<1$, there exists some constant $C_2(C_1,U,V)>0$\ and some cut-off function $\eta$\ for the pair $V\subset U$, such that for any $v\in \mathcal{F}$,
	\begin{eqnarray}
	\label{cut-offbddenergy}
	\int_X v^2 d\Gamma(\eta,\eta)\leq C_1\int_X\eta^2d\Gamma(v,v)+C_2\int_{\scaleto{\mbox{supp}\{\eta\}}{6pt}}v^2\,dm.
	\end{eqnarray}
	We call such $\eta$\ functions \textit{nice cut-off functions} corresponding to $C_1,C_2$.
\end{assumption}
\begin{remark}
\label{cut-offrmk}
Later we show that in Assumption \ref{assumptionbothcases}, the condition $U,V\in \mathcal{TB}$\ for some topological basis $\mathcal{TB}$\ is ``redundant'', in the sense that Assumption \ref{assumptionbothcases} implies automatically that nice cut-off functions in the sense of (\ref{cut-offbddenergy}) exist for any pair of open sets $V\Subset U$. See Lemma \ref{cut-offannuli}. We also remark Assumption \ref{assumptionbothcases} has a straightforward equivalent form that for any pair of precompact open sets $U,V$\ with disjoint closures, i.e., $\overline{U}\cap\overline{V}=\emptyset$, for any $C_1$\ between $0$\ and $1$, there exists a cut-off function $\eta$\ such that $\eta=1$\ on $U$, $\eta=0$\ on $V$, and  there exists some constant $C_2(C_1,U,V)>0$, such that for any $v\in \mathcal{F}$,
\begin{eqnarray}
\label{cut-offbddenergy1}
\int_X v^2 d\Gamma(\eta,\eta)\leq C_1\int_X\eta^2d\Gamma(v,v)+C_2(C_1,U,V)\int_{\scaleto{\mbox{supp}\{\eta\}}{6pt}}v^2\,dm.
\end{eqnarray}
\end{remark}
Let $\eta(x)$\ be a nice cut-off function and $l(t)$\ be a smooth function on $\mathbb{R}$\ with compact support, then the product $\eta(x)l(t)$\ is a function in $\mathcal{F}_c(I\times X)$. We call such product functions \textit{nice product cut-off functions}, and we denote such functions by adding an overline, i.e., $\overline{\eta}(t,x):=\eta(x)l(t)$.
\begin{remark}
If a cut-off function $\eta$\ for some pair $V\subset U$\ satisfies that its corresponding energy measure is absolutely continuous w.r.t. $m$, and $d\Gamma(\eta,\eta)/dm$\ is bounded, i.e.,
\begin{eqnarray}
\label{cut-offbddgradient}
d\Gamma(\eta,\eta)\leq C\,dm
\end{eqnarray}
for some $C<\infty$, then $\eta$\ is a nice cut-off function and satisfies (\ref{cut-offbddenergy}) with $C_1=0$\ (hence any $0<C_1<1$), $C_2=C$. $C_2$\ is independent of $C_1$. We say in this special case that the cut-off function $\eta$\ has \textit{bounded gradient}.

Conversely, if for some nice cut-off function $\eta$, (\ref{cut-offbddenergy}) can be extended to hold true for $C_1=0$\ and $C_2(0,U,V)<\infty$, then $\eta$\ has bounded gradient.

In particular, when the intrinsic pseudo-distance of the Dirichlet space,
\begin{eqnarray}
\label{intrinsicdist}
\rho_X(x,y)=\sup\left\{\varphi(x)-\varphi(y)\,|\,\varphi\in \mathcal{F}_{\scaleto{\mbox{loc}}{5pt}}(X)\cap C(X),\ d\Gamma(\varphi,\varphi)\leq dm\right\},
\end{eqnarray}
is a continuous metric that induces the same topology of $X$, the Dirichlet space satisfies Assumption \ref{assumptionbothcases} with existence of cut-off functions with bounded gradient, and the cut-off functions can be explicitly constructed using the intrinsic distance. Cf. \cite{Sturm2}.
\end{remark}
\begin{remark}
Typical examples of Dirichlet spaces that satisfy Assumption \ref{assumptionbothcases} but do not possess cut-off functions with bounded gradient are some fractal spaces, including for example the Sierpinski gasket and the Sierpinski carpet. For fractal spaces, usually the existence of nice cut-off functions is guaranteed as consequences of  other properties like sub-Gaussian upper bounds satisfied by the Dirichlet space (heat kernel). In general, in such cases, there are no simple explicit constructions of cut-off functions satisfying (\ref{cut-offbddenergy}). For references we mention \cite{AndresBarlow} and \cite{EHI}.
\end{remark}
Let $(X,\mathcal{E},\mathcal{F})$\ be a symmetric regular local Dirichlet space as before. We first verify that the cut-off functions in Assumption \ref{assumptionbothcases} indeed satisfy the conditions in Assumption \ref{cut-offgeneralassumption}.
\begin{lemma}
	\label{cut-offweaklemma}
	Any nice cut-off function $\varphi$\ in the sense of (\ref{cut-offbddenergy}) satisfies (ii) in Assumption \ref{cut-offgeneralassumption}, namely, let $U\Subset X$\ be some open set such that $\mbox{supp}\{\varphi\}\subset U$, then for any $u\in \mathcal{F}_{\scaleto{\mbox{loc}}{5pt}}(U)$, the product $\varphi\cdot u\in \mathcal{F}_{c}(U)$.	
\end{lemma}
\begin{proof}
	The support of the product function $\varphi u$\ is clearly contained in $U$. To show $\varphi u\in \mathcal{F}$, recall that $u\in \mathcal{F}_{\scaleto{\mbox{loc}}{5pt}}(U)$\ means $u$\ is in $L^2_{\scaleto{\mbox{loc}}{5pt}}(U)$, and satisfies for any $V\Subset U$, there exists some $u^\sharp$\ in $\mathcal{F}$\ such that $u^\sharp=u$\ $m$-a.e. on $V$. Pick some open set $V$\ such that $\mbox{supp}\{\varphi\}\subset V\Subset U$, fix some $u^\sharp\in \mathcal{F}$\ that agrees with $u$\ $m$-a.e. on $V$. Then
	\begin{eqnarray*}
		\lefteqn{\left|\left|\varphi u^\sharp\right|\right|^2_{\mathcal{E}_1}
		= \int_X (\varphi u^\sharp)^2\,dm+\int_X d\Gamma(\varphi u^\sharp,\varphi u^\sharp)+\int_X (\varphi u^\sharp)^2\,dk}\\
		&\leq& \int_X (\varphi u^\sharp)^2\,dm+\int_X (\varphi u^\sharp)^2\,dk+2\left[\int_X \varphi^2\,d\Gamma(u^\sharp,u^\sharp)+\int_X (u^\sharp)^2\,d\Gamma(\varphi,\varphi)\right].
	\end{eqnarray*}
	The first two terms are clearly finite, the third term is bounded above by $\mathcal{E}_1(u^\sharp,u^\sharp)$\ up to some constant, and the last term is finite due to (\ref{cut-offbddenergy}). Hence $\left|\left|\varphi u^\sharp\right|\right|_{\mathcal{E}_1}<+\infty$, and $\varphi u=\varphi u^\sharp\in \mathcal{F}_c(U)$.
\end{proof} 
So far the examples we have described satisfy Assumption \ref{assumptionbothcases} for all pairs of open sets $V\Subset U$. The reason in Assumption \ref{assumptionbothcases} we only require nice cut-off functions to exist for pairs of open sets in some topological basis $\mathcal{TB}$\ is to make the assumption easy to check for some infinite dimensional examples, like the infinite dimensional torus or the infinite product of Sierpinski gaskets.

In the next lemma we state the automatic extension of existence of nice cut-off functions for general pairs of open sets, given Assumption \ref{assumptionbothcases}. We postpone the proof to Section 8.
\begin{lemma}
	\label{cut-offannuli}	
	Suppose Assumption \ref{assumptionbothcases} holds. Then for any open sets $U,V$\ with $V\Subset U$, any constant $0<C_1<1$, there exists some constant $C_2=C_2(C_1,U,V)>0$\ and some nice cut-off function in the sense of (\ref{cut-offbddenergy}) corresponding to $C_1,C_2$. In particular, $U,V$\ are not necessarily in $\mathcal{TB}$.
\end{lemma}
Given any nice cut-off function and any function in the domain $\mathcal{F}$, by Lemma \ref{cut-offweaklemma}, their product belongs to $\mathcal{F}$. The energy of the product function satisfies the following estimate, which we later refer to as the gradient inequality.
\begin{lemma}[gradient inequality]
	\label{gradientineq}
	Let $\eta$\ be a nice cut-off function corresponding to $C_1,C_2$\ in the sense of (\ref{cut-offbddenergy}), let $v\in \mathcal{F}$. Then 
	\begin{eqnarray}
	\int_Xd\Gamma(\eta v,\,\eta v)\leq \frac{1-2C_1}{1-4C_1}\int_Xd\Gamma(\eta^2v,\,v)+\frac{C_2}{1-4C_1}\int_{\scaleto{\mbox{supp}\{\eta\}}{5pt}}v^2\,dm.\label{gradient1}
	\end{eqnarray}
\end{lemma}
The point of the lemma is to bound the energy of the product function $\eta v$ on the left-hand side by $L^2$\ integrals on the right-hand side (when $v\in \mathcal{D}(P)$). Indeed, the first integral is equal to $\int_X \eta^2 v\,Pv\,dm$.

It is easy to check the validity of this lemma in the special case when the cut-off function has bounded gradient. In this case, by expanding $\int_Xd\Gamma(\eta v,\,\eta v)$\ by the product rule and utilizing the upper bound  $d\Gamma(\eta,\eta)/dm\le M$, we get
\begin{eqnarray*}
\int_Xd\Gamma(\eta v,\,\eta v)\leq \int_Xd\Gamma(\eta^2v,\,v)+M\int_{\scaleto{\mbox{supp}\{\eta\}}{5pt}}v^2\,dm,
\end{eqnarray*}
which is exactly (\ref{gradient1}) with $C_1=0$, $C_2=M$. 

In the general case, when the cut-off function does not have bounded gradient (thus $C_1$\ in (\ref{cut-offbddenergy}) must be taken as positive), (\ref{gradient1}) is less obvious, and we give the proof below.

\begin{proof}[Proof for Lemma \ref{gradientineq}]
	\begin{eqnarray*}
		\lefteqn{\int_Xd\Gamma(\eta v,\,\eta v)= \int_X\eta^2\,d\Gamma(v,v)+\int_Xv^2\,d\Gamma(\eta,\eta)+2\int_X\eta v\,d\Gamma(\eta,v)}\notag\\
		&\geq& \int_X\eta^2\,d\Gamma(v,v)+\int_Xv^2\,d\Gamma(\eta,\eta)-\frac{1}{2}\int_X\eta^2\,d\Gamma(v,v)-2\int_Xv^2\,d\Gamma(\eta,\eta)\notag\\
		&=& \frac{1}{2}\int_X\eta^2\,d\Gamma(v,v)-\int_Xv^2\,d\Gamma(\eta,\eta)\notag\\
		&\geq& \frac{1}{2}\int_X\eta^2\,d\Gamma(v,v)-\left[C_1\int_X\eta^2\,d\Gamma(v,v)+C_2\int_{\scaleto{\mbox{supp}\{\eta\}}{6pt}}v^2\,dm\right]\notag\\
		&=&\left(\frac{1}{2}-C_1\right)\int_X\eta^2\,d\Gamma(v,v)-C_2\int_{\scaleto{\mbox{supp}\{\eta\}}{6pt}}v^2\,dm.
	\end{eqnarray*}
	Hence when $C_1<\frac{1}{2}$,
	\begin{eqnarray}
	\int_X\eta^2\,d\Gamma(v,v)\leq \frac{1}{\frac{1}{2}-C_1}\int_Xd\Gamma(\eta v,\,\eta v)+\frac{C_2}{\frac{1}{2}-C_1}\int_{\scaleto{\mbox{supp}\{\eta\}}{6pt}}v^2\,dm.\label{gradient0}
	\end{eqnarray}
	On the other hand,
	\begin{eqnarray*}
		\lefteqn{\int_Xd\Gamma(\eta v,\,\eta v)= \int_Xd\Gamma(\eta^2v,\,v)+\int_Xv^2\,d\Gamma(\eta,\eta)}\\
		&\leq& \int_Xd\Gamma(\eta^2v,\,v)+C_1\int_X\eta^2\,d\Gamma(v,v)+C_2\int_{\scaleto{\mbox{supp}\{\eta\}}{6pt}}v^2\,dm.
	\end{eqnarray*}
	Substituting the upper bound in (\ref{gradient0}) for $\int_X\eta^2\,d\Gamma(v,v)$\ here, we get
	\begin{eqnarray*}
		\lefteqn{\int_Xd\Gamma(\eta v,\,\eta v)
		\leq \int_Xd\Gamma(\eta^2v,\,v)+C_2\int_{\scaleto{\mbox{supp}\{\eta\}}{6pt}}v^2\,dm}\\
		&&+C_1\left[\frac{1}{\frac{1}{2}-C_1}\int_Xd\Gamma(\eta v,\,\eta v)+\frac{C_2}{\frac{1}{2}-C_1}\int_{\scaleto{\mbox{supp}\{\eta\}}{6pt}}v^2\,dm\right].
	\end{eqnarray*}	
	When $C_1<\frac{1}{4}$, this implies
	\begin{eqnarray*}
		\int_Xd\Gamma(\eta v,\,\eta v)\leq \frac{1-2C_1}{1-4C_1}\int_Xd\Gamma(\eta^2v,\,v)+\frac{C_2}{1-4C_1}\int_{\scaleto{\mbox{supp}\{\eta\}}{6pt}}v^2\,dm.
	\end{eqnarray*}
\end{proof}
In applications we do not care about the exact constants, so in the following we consider $C_1<\frac{1}{8}$\ and (\ref{gradient1}) implies
\begin{eqnarray}
\int_Xd\Gamma(\eta v,\,\eta v)\leq 2\int_Xd\Gamma(\eta^2v,\,v)+2C_2\int_{\scaleto{\mbox{supp}\{\eta\}}{6pt}}v^2\,dm.\label{gradient11}
\end{eqnarray}
Since $dk$\ is nonnegative, we also have
\begin{eqnarray}
\label{gradient111}
\mathcal{E}(\eta v,\,\eta v)\leq 2\mathcal{E}(\eta^2v,\,v)+2C_2\int_{\scaleto{\mbox{supp}\{\eta\}}{6pt}}v^2\,dm.
\end{eqnarray}

\subsection{$L^2$\ Gaussian upper bound}
In our treatment of the $L^2$\ time regularity of local weak solutions, we rely much on the heat semigroup, which is smooth in time. Roughly speaking, we use the heat semigroup to construct an approximate sequence to a local weak solution $u$, and show that this approximate sequence (1) converges to $u$\ in some weak sense; (2) forms a Cauchy sequence in the space $\mathcal{F}^n(I\times X)=W^{n,2}(I\rightarrow\mathcal{F})$. These two statements together then imply that $u$\ is (locally) in $\mathcal{F}^n(I\times X)$. To show the approximate sequence is Cauchy, we use the following $L^2$\ version of Gaussian type upper bound for the heat semigroup.
\begin{assumption}
	\label{L2gaussian}
	For any two open sets $V_1,V_2\Subset X$\ with $\overline{V_1}\cap \overline{V_2}=\emptyset$, let $\mathcal{A}(V_1,V_2):=\left\{(g_1,g_2)\,\big|\,\mbox{supp}\{g_i\}\subset V_i,\,||g_i||_{\scaleto{L^2(V_i,m)}{6pt}}\leq 1,\,i=1,2\right\}$. For any $a\geq 0$, any $n\in \mathbb{N}$,
	\begin{eqnarray*}
		\lim_{t\rightarrow 0^+}\left(\sup_{(g_1,g_2)\in \mathcal{A}(V_1,V_2)}\left\{\frac{1}{t^a}\left|<\partial_t^nH_tg_1,g_2>\right|\right\}\right)=0.
	\end{eqnarray*}
	To simplify notation we write
	\begin{eqnarray*}
		G_{\scaleto{V_1,V_2}{6pt}}(a,n,t)
		:=\max_{0\leq k\leq n}\sup\left\{\frac{1}{t^a}\left|<\partial_t^kH_tg_1,g_2>\right|\,\bigg|\,(g_1,g_2)\in \mathcal{A}(V_1,V_2)\right\}.
	\end{eqnarray*}
\end{assumption}
\begin{remark}
The $L^2$\ Gaussian type bound above is a very weak Gaussian upper bound. For example, from this bound itself we cannot tell if the heat semigroup admits a density, and even if we assume there is a density, neither can we say anything about the pointwise estimate of the density function. On the other hand, when there is some pointwise Gaussian or sub-Gaussian upper bound, then the $L^2$\ Gaussian bound is a very weak consequence. So we still name it ``$L^2$\ Gaussian type upper bound'' after the name of the classical pointwise Gaussian or sub-Gaussian upper bound.
\end{remark}
\begin{remark}
We also remark that the $L^2$\ Gaussian type upper bound is often automatically satisfied by the heat semigroup. For example, when there are enough cut-off functions with bounded gradient (see (\ref{cut-offbddgradient})), or when the general assumption (Assumption \ref{assumptionbothcases}) holds with $C_2(C_1,U,V)=C(U,V)C_1^{-\alpha}$\ for some $\alpha>0$, then the $L^2$\ Gaussian bound for the semigroup holds.
\end{remark}
More precisely, under Assumption \ref{assumptionbothcases} with cut-off functions with bounded gradient, one can define the distance between sets (cf. \cite{Hino}): for any two measurable precompact sets $U,V$,
\begin{eqnarray}
\label{defd(U,V)}
d(U,V):=\sup_{\substack{\scaleto{\phi\in \mathcal{F}_{\scaleto{\mbox{loc}}{3pt}}(X)\cap L^\infty}{6pt}\\\scaleto{d\Gamma(\phi,\phi)\leq dm}{6pt}}}\left\{\essinf_{x\in U}\phi(x)-\esssup_{y\in V}\phi(y)\right\}.
\end{eqnarray}
The following more concrete $L^2$\ Gaussian bound is a classical result, often referred to as the Takeda formula (cf. \cite{Takeda}). Let $V_1,V_2$\ be two precompact measurable subsets of $X$\ with $\overline{V_1}\cap \overline{V_2}=\emptyset$. Then $0<d(V_1,V_2)<\infty$. For any pair $(g_1,g_2)\in \mathcal{A}(V_1,V_2)$,
\begin{eqnarray}
\label{L2Gaussiannicecut-off}
|<H_tg_1,g_2>|\leq \exp{\left\{-\frac{d(V_1,V_2)^2}{4t}\right\}}.
\end{eqnarray}
Proofs for various kinds of Gaussian upper bounds usually use the so-called Davies' method, cf. eg. \cite{Davies}. To generalize the upper bound for terms like $|<\partial_t^nH_tg_1,g_2>|$, one can use for example the complex analysis method from \cite{Coulhon}, or the method in \cite{DaviesnonGaussian}.

However, when the existence of nice cut-off functions with bounded gradient is not guaranteed, there could be disjoint closed measurable sets $U,V$\ with distance $d(U,V)=0$\ (because roughly speaking the only functions with bounded gradient are constant functions), then this distance notion is not helpful in getting a Gaussian type upper bound.

Under Assumption \ref{assumptionbothcases} with cut-off functions satisfying the general inequality (\ref{cut-offbddenergy}), or (\ref{cut-offbddenergy1}) as in the equivalent form of Assumption \ref{assumptionbothcases} (see Remark \ref{cut-offrmk}), when furthermore $C_2$\ depends on $C_1$\ in the specific form $C_2(C_1,U,V)=C(U,V)C_1^{-\alpha}$\ for some $\alpha>0$, $C(U,V)>0$, by a modification of Davies' method, we can show the following $L^2$\ Gaussian type bound
\begin{eqnarray}
\label{L2Gaussiancut-off}
|<H_tg_1,g_2>|\leq \exp{\left\{-\left(\frac{1}{4^{\alpha+1}C(V_1,V_2)t}\right)^{\frac{1}{1+2\alpha}}\right\}}.
\end{eqnarray}
Here again $V_1,V_2$\ are two precompact measurable subsets of $X$\ with $\overline{V_1}\cap \overline{V_2}=\emptyset$, $(g_1,g_2)\in \mathcal{A}(V_1,V_2)$.

We call both (\ref{L2Gaussiannicecut-off}) and (\ref{L2Gaussiancut-off}) Gaussian type upper bounds. Note that (formally) if we take $\alpha=0$\ and $C(V_1,V_2)=d(V_1,V_2)^{-2}$\ in (\ref{L2Gaussiancut-off}), then we recover (\ref{L2Gaussiannicecut-off}). In Section 8 we give a proof for (\ref{L2Gaussiancut-off}), as well as how this implies a similar bound for $\left|\left<\partial_t^nH_tg_1,\,g_2\right>\right|$.

\section{Statement of the main results and overview of proof} \setcounter{equation}{0} 
\subsection{Statement of the main results}
In this section we state our results on the time regularity property of local weak solutions of the heat equation $(\partial_t+P)u=f$. Our main result is that the regularity in time of $u$\ is as good as that of the right-hand side $f$. Note that as a local weak solution on some $I\times U\subset I\times X$, $u$\ satisfies the prerequisite $u\in \mathcal{F}_{\scaleto{\mbox{loc}}{5pt}}(I\times U)$, so any of its ``$\mathcal{F}(I\times X)$\ representative'' $u^\sharp$\ automatically has distributional time derivatives of any order. The challenge hence lies in showing that these time derivatives belong to $\mathcal{F}(I\times X)=L^2(I\rightarrow \mathcal{F})$. 
Our main theorem is the following.
\begin{theorem}
	\label{L2thm}
	Let $(X,m)$\ be a metric measure space and $(\mathcal{E}, \mathcal{F})$\ be a symmetric regular local Dirichlet form satisfying Assumption \ref{assumptionbothcases} (existence of nice cut-off functions). Assume the associated heat semigroup $(H_t)_{t>0}$\ satisfies Assumption \ref{L2gaussian} (the $L^2$\ Gaussian type upper bound). 
	Let $U\subset X$\ be an open set, $I=(a,b)$\ be an open interval, and $f$\ be a function locally in $W^{n,2}(I\rightarrow L^2(U))$. Let $u$\ be a local weak solution of $(\partial_t+P)u=f$\ on $I\times U$. Then $u$\ is in $\mathcal{F}^n_{\scaleto{\mbox{loc}}{5pt}}(I\times U)$.
\end{theorem}

In short, Theorem \ref{L2thm} claims that if the right-hand side $f$\ of the heat equation locally has time derivatives up to order $n$, then so does the local weak solution $u$, and its time derivatives up to order $n$\ locally belong to $L^2(I\rightarrow \mathcal{F})$. An important implication of Theorem \ref{L2thm} is that the time derivatives of $u$\ (up to order $n$) are local weak solutions of corresponding heat equations.
\begin{corollary}
	\label{L2corollary}
	Under the hypotheses in Theorem \ref{L2thm}, if $f$\ is locally in the space $W^{n,2}(I\rightarrow L^2(U))$, then for any $1\leq k\leq n$, $\partial_t^ku$\ is a local weak solution of
	\begin{eqnarray}
	\label{L2coreqn1}
	\left(\partial_t+P\right)\partial_t^ku=\partial_t^kf.
	\end{eqnarray}
	In particular, if $u$\ is a local weak solution of 
	$(\partial_t+P)u=0$ on $I\times U$, then all time derivatives of $u$\ are local weak solutions of the same heat equation on $I\times U$.
\end{corollary}
\subsection{Sketch of proof for a special case of Theorem \ref{L2thm}} In the next two sections we prove Theorem \ref{L2thm} and Corollary \ref{L2corollary}. In this section, we give a simplified outline in a special case - when $X$\ is compact, and we consider any local weak solution $u$\ of the heat equation on $I\times X$. The rigorous proof in the next section is built on this outline but takes into consideration the complications brought in by noncompactness of the space $X$\ and the restriction on some open subset $U\subset X$. In this more general context, the existence of nice cut-off functions and the $L^2$\ Gaussian type upper bound become essential. In the special case where $X$\ is compact and $u$\ is a local weak solution on the full time-space cylinder $I\times X$, since  $\mathcal{F}_c(X)=\mathcal{F}=\mathcal{F}_{\scaleto{\mbox{loc}}{5pt}}(X)$, we know that $u$\ itself is in the domain of the Dirichlet form, and in particular, in $L^2(X)$. The statement of the theorem is much shortened as we do not need to assume the existence of nice cut-off functions on $X$, nor that the heat semigroup satisfies the $L^2$\ Gaussian type upper bound. The spaces $\mathcal{F}_c(I\times X)$, $\mathcal{F}(I\times X)$, $\mathcal{F}_{\scaleto{\mbox{loc}}{5pt}}(I\times X)$\ are different due to the inclusion of the open time interval $I=(a,b)$. In the proof we do need to multiply $u$\ with some smooth cut-off function in time, but in the outline below we ignore that technicality and pretend the functions are globally good in time.

Recall that we take the following convention: for any function $g(s,x)$, we write $g^s(x):=g(s,x)$.
\begin{proposition}[special case] 
	Let $(X,m)$\ be a compact metric measure space and $(\mathcal{E}, \mathcal{F})$\ be a symmetric regular local Dirichlet form. Given $I=(a,b)\subset \mathbb{R}$\ and a function $f$\ that is locally in $W^{1,2}(I\rightarrow L^2(X))$, let $u$\ be a local weak solution of $(\partial_t+P)u=f$\ on $I\times X$. Then $u$\ is locally in $\mathcal{F}^1(I\times X)$.
\end{proposition}
\begin{proof}[Outline of Proof]
We consider the function
	\begin{eqnarray*}
		u_\tau(s,x):=\int_I\rho_\tau(s-t)H_{s-t}u^t(x)\,dt.
	\end{eqnarray*}
	The integral makes sense as a Bochner integral. Here $\tau>0$, and $\rho_\tau(r)=\frac{1}{\tau}\rho(\frac{r}{\tau})$, where $\rho$\ is some smooth nonnegative cut-off function on $\mathbb{R}$\ supported in $(1,2)$, with total integral equal to $1$. Note that when there is the notion of convolution and when $H_t$\ admits a density function (heat kernel), the approximate sequence above is exactly the convolution in time and space of $u$\ and the heat kernel (with a cut-off function in time). 
	
	Since $H_t$\ is smooth in time, it is easy to show that $u_\tau$\ is smooth in time. More precisely, for any $\tau>0$, $u_\tau\in C^\infty(I\rightarrow \mathcal{F})$. It is routine to verify that $u_\tau$\ converges to $u$\ in $L^2(I\times X)$\ as $\tau$\ tends to $0$. So to prove the proposition, it suffices to show that $\{u_\tau\}_{\tau>0}$\ is Cauchy in $W^{1,2}(I\rightarrow \mathcal{F})=\mathcal{F}(I\times X)$.
	
	To this end, we show that $||\partial_\tau u_\tau||_{\scaleto{W^{1,2}(I\rightarrow \mathcal{F})}{6pt}}$\ is integrable in $\tau$\ near $0$, then
	\begin{eqnarray*}
		\int_{0}^{\gamma}\left|\left|\partial_\tau u_\tau\right|\right|_{\scaleto{W^{1,2}(I\rightarrow \mathcal{F})}{6pt}}\,d\tau\rightarrow 0\ \ \mbox{as }\gamma\rightarrow 0,
	\end{eqnarray*}
	and thus $\{u_\tau\}_{\tau>0}$\ is Cauchy in $W^{1,2}(I\rightarrow \mathcal{F})$. We first estimate $\left|\left|\partial_\tau \partial_s u_\tau\right|\right|_{\scaleto{L^2(I\times X)}{6pt}}$. Note that $\partial_\tau\rho_\tau(r)=-\partial_r\bar{\rho}_\tau(r)$, where $\bar{\rho}_\tau(r)=\frac{r}{\tau^2}\rho(\frac{r}{\tau})$. By duality,
	\begin{eqnarray*}
		\lefteqn{||\partial_\tau \partial_su_\tau||_{\scaleto{L^2(I\times X)}{6pt}}=\sup_{\substack{\left|\left|\varphi\right|\right|_{L^2(I\times X)}\leq 1\\
					\varphi\in C^\infty_c(I\rightarrow L^2(X))}}\left<\partial_\tau \partial_su_\tau,\, \varphi\right>_{\scaleto{L^2(I\times X)}{6pt}}}\\
		&=& \sup_{\substack{\left|\left|\varphi\right|\right|_{L^2(I\times X)}\leq 1\\
				\varphi\in C^\infty_c(I\rightarrow L^2(X))}}\left\{\left<\int_I\partial_s\left[\partial_t\bar{\rho}_\tau(s-t)H_{s-t}\right]u^t(x)\,dt,\ \varphi\right>\right\}\\
		&=& \sup_{\substack{\left|\left|\varphi\right|\right|_{L^2(I\times X)}\leq 1\\\varphi\in C^\infty_c(I\rightarrow L^2(X))}}\left\{\int_I\int_I\int_Xu^t(x)\cdot \partial_s\partial_t\left[\bar{\rho}_\tau(s-t)H_{s-t}\right]\varphi^s(x)\,dmdsdt\right.\\
		&&\left.-\int_I\int_I\int_Xu^t(x)\cdot \partial_s\left[\bar{\rho}_\tau(s-t)\cdot \partial_tH_{s-t}\right]\varphi^s(x)\,dmdsdt\right\}.
	\end{eqnarray*}
	From the second line to the third line we used the Fubini theorem and the self-adjointness of $H_t$\ to move $\partial_s[\partial_t\bar{\rho}_\tau(s-t)H_{s-t}]$\ from the ``$u$'' side to ``$\varphi$'' side, then used the product rule to redistribute $\partial_t$. Since $\partial_tH_{s-t}=PH_{s-t}$\ and $u$\ is a local weak solution of $(\partial_t+P)u=f$\ on $I\times X$, the above two terms in the brackets together, modulo a cut-off function in time that we omit in this proof (i.e., think of the function $(t,x)\mapsto \partial_s[\bar{\rho}_\tau(s-t)H_{s-t}]\varphi^s(x)$\ above as a test function), equals 
	\begin{eqnarray*}
	-\int_I\int_I\int_X f(t,x)\cdot\partial_s\left[\bar{\rho}_\tau(s-t)H_{s-t}\right]\varphi^s(x)\,dmdtds.
	\end{eqnarray*}
	By rewriting $\partial_s[\bar{\rho}_\tau(s-t)H_{s-t}]$\ as $-\partial_t[\bar{\rho}_\tau(s-t)H_{s-t}]$\ and use integration by parts, we get
	\begin{eqnarray*}
		\lefteqn{||\partial_\tau \partial_su_\tau||_{\scaleto{L^2(I\times X)}{6pt}}}\\
		&=&\sup_{\substack{\left|\left|\varphi\right|\right|_{L^2(I\times X)}\leq 1\\\varphi\in C^\infty_c(I\rightarrow L^2(X))}}\left|\int_I\int_I\int_X \partial_tf(t,x)\cdot\bar{\rho}_\tau(s-t)H_{s-t}\varphi^s(x)\,dmdtds\right|.
	\end{eqnarray*}
	Here we did not consider the boundary term, but that is not a problem once we add in the cut-off function in time in the rigorous proof in the next section. By using the $W^{1,2}(I\rightarrow L^2(X))$\ norm of $f$\ and note that $\sup_{s\in I}\int_I\bar{\rho}_\tau(s-t)\,dt=1$, we conclude that $||\partial_\tau \partial_su_\tau||_{\scaleto{L^2(I\times X)}{6pt}}$\ is bounded above independent of $\tau$.
	
	To estimate $||\partial_\tau \partial_su_\tau||_{\scaleto{L^2(I\rightarrow \mathcal{F})}{6pt}}$, note that for any $\tau>0$, $s\in I$, $\partial_\tau \partial_su_\tau$\ belongs to $\mathcal{D}(P)$. Thus
	\begin{eqnarray*}
		\left(\int_I\mathcal{E}(\partial_\tau \partial_su_\tau,\,\partial_\tau \partial_su_\tau)\,ds\right)^{1/2}
		\leq \left|\left|\partial_\tau \partial_su_\tau\right|\right|_{\scaleto{L^2(I\times X)}{6pt}}^{1/2}\left|\left|P(\partial_\tau \partial_su_\tau)\right|\right|_{\scaleto{L^2(I\times X)}{6pt}}^{1/2}.
	\end{eqnarray*}
	Since $\sup\limits_{0<\tau<1}||\partial_\tau \partial_su_\tau||_{\scaleto{L^2(I\times X)}{6pt}}<\infty$\ from above, it suffices to show that
	\begin{eqnarray*}
		\left|\left|P(\partial_\tau \partial_su_\tau)\right|\right|_{\scaleto{L^2(I\times X)}{6pt}}\lesssim \frac{1}{\tau}.
	\end{eqnarray*}
	Running the estimate for $||\partial_\tau \partial_su_\tau||_{\scaleto{L^2(I\times X)}{6pt}}$\ again with $H_{s-t}$\ replaced by $PH_{s-t}$, we can get the desired estimate.
\end{proof}
\section{Proof of the main results} \setcounter{equation}{0}
\subsection{Proof of Theorem \ref{L2thm} - general strategy}
In this section we prove Theorem \ref{L2thm}. To show that $u\in \mathcal{F}^n_{\scaleto{\mbox{loc}}{5pt}}(I\times U)$, for any $J\times V\Subset I\times U$, we show there exists some function in $\mathcal{F}^n(I\times X)$\ that equals to $\overline{\psi} u$\ a.e. over $J\times V$. Here $\overline{\psi}(s,x):=\psi(x)w(s)$\ is some nice product cut-off function such that $\overline{\psi}\equiv 1$\ on some $J_{\overline{\psi}}\times V_{\overline{\psi}}$\ where $J\times V\Subset J_{\overline{\psi}}\times V_{\overline{\psi}}$; $\mbox{supp}\{\overline{\psi}\}\subset I_{\overline{\psi}}\times U_{\overline{\psi}}$\ for some $I_{\overline{\psi}}\times U_{\overline{\psi}}\Subset I\times U$. Our notational choice is that $J,V$\ are proper subsets of $I,U$, and subscripts mark which function these sets are ``affiliated with''.

To find such a function in $\mathcal{F}^n(I\times X)=W^{k,2}(I\rightarrow \mathcal{F})$, we first define an approximate sequence (now, with proper nice cut-off functions inserted) to the local weak solution $u$\ and show that the approximate sequence is Cauchy in $\mathcal{F}^n(I\times X)$. Next, we show that the sequence converges to $u$\ in the $L^2$\ sense (this step does not make use of the fact that $u$\ is a local weak solution). The limit of the approximate sequence then serves as the function in $\mathcal{F}^n(I\times X)$\ that agrees with $\overline{\psi}u$\ on $J\times V$.

The approximate sequence is defined as follows. Let $\rho_\tau$\ be defined as in the last section, that is, $\rho(t)\in C^{\infty}_c(1,2)$\ is some positive bounded function satisfying $\int_{\mathbb{R}}\rho(t)\,dt=1$, and $\rho_\tau(t)$\ is defined as $\rho_\tau(t)=\frac{1}{\tau}\rho\left(\frac{t}{\tau}\right)$\ ($\tau>0$). Note that $\mbox{supp}\{\rho_\tau\}\subset (\tau,2\tau)$. By inspection, $\partial_\tau \rho_\tau(t)=-\partial_t\overline{\rho}_\tau(t)$, where $\overline{\rho}_\tau(t)=\frac{t}{\tau^2}\rho\left(\frac{t}{\tau}\right)$. Let $\overline{\eta}(y,t)=\eta(y)l(t)$\ be another nice product cut-off function which is $1$\ over some neighborhood of the support of $\overline{\psi}$. More precisely, $\overline{\eta}\equiv 1$\ on some $J_{\overline{\eta}}\times V_{\overline{\eta}}$\ where $J\times V\Subset I_{\overline{\psi}}\times U_{\overline{\psi}}\Subset J_{\overline{\eta}}\times V_{\overline{\eta}}$; $\mbox{supp}\{\overline{\eta}\}\subset I_{\overline{\eta}}\times U_{\overline{\eta}}$\ for some $I_{\overline{\eta}}\times U_{\overline{\eta}}\Subset I\times U$. Consider the sequence defined by
\begin{eqnarray}
\widetilde{u}_\tau(s,x):=\int_I\rho_\tau(s-t)H_{s-t}\left(\overline{\eta}^t u^t\right)(x)\,dt.
\end{eqnarray}
 
As mentioned above, we claim that (1) the family $\left\{\overline{\psi}\widetilde{u}_\tau\right\}$\ is Cauchy in $\mathcal{F}^n(I\times X)$\ and hence has a limit in the same function space; (2) $\overline{\psi}\widetilde{u}_\tau\rightarrow \overline{\psi}\overline{\eta} u=\overline{\psi} u$\ in $L^2(I\times X)$\ as $\tau\rightarrow 0$. So the two limit functions must equal a.e., in other words, the ``$L^2$\ limit'' $\overline{\psi} u$\ in fact belongs to $\mathcal{F}^n(I\times X)$. Since $\overline{\psi} u=u$\ a.e. on $J\times V$, $J\times V$\ is arbitrary, the statement in Theorem \ref{L2thm} follows.

To show $\left\{\overline{\psi}\widetilde{u}_\tau\right\}$\ is Cauchy in $\mathcal{F}^n(I\times X)=W^{n,2}(I\rightarrow \mathcal{F})$, we first show that for each $\tau>0$, $\overline{\psi}\widetilde{u}_\tau\in C^\infty(I\rightarrow \mathcal{F})$. It then suffices to prove the following two propositions.
\begin{proposition}
	\label{L2prop1}
	Under the hypotheses in Theorem \ref{L2thm}, for any nice product cut-off function $\overline{\psi}$\ supported in $I\times U$,
	\begin{eqnarray*}
		\max_{0\leq k\leq n}\sup_{0<\tau<1}\left|\left|\partial_\tau\partial_s^k\left(\overline{\psi}\widetilde{u}_\tau\right)\right|\right|_{\scaleto{L^2(I\times X)}{6pt}}<+\infty.
	\end{eqnarray*}
\end{proposition}
\begin{proposition}
	\label{L2prop2}
	Under the hypotheses in Theorem \ref{L2thm}, for any nice product cut-off function $\overline{\psi}$\ supported in $I\times U$,
	\begin{eqnarray*}
		\max_{0\leq k\leq n}\left(\int_I\mathcal{E}(\partial_\tau\partial_s^k(\overline{\psi}\widetilde{u}_\tau),\, \partial_\tau\partial_s^k(\overline{\psi}\widetilde{u}_\tau))\,ds\right)^{1/2}\lesssim \frac{1}{\sqrt{\tau}}.
	\end{eqnarray*}
\end{proposition}
These two propositions together imply that
\begin{eqnarray*}
	\int_{0}^{\gamma}\left|\left|\partial_\tau(\overline{\psi}\widetilde{u}_\tau)\right|\right|_{\scaleto{W^{n,2}(I\rightarrow \mathcal{F})}{6pt}}\,d\tau\lesssim \int_{0}^{\gamma} \frac{1}{\sqrt{\tau}}\,d\tau\rightarrow 0\ \ \mbox{as }\gamma\rightarrow 0,
\end{eqnarray*}
hence the family $\left\{\overline{\psi}\widetilde{u}_\tau\right\}$\ is Cauchy in $W^{n,2}(I\rightarrow \mathcal{F})$.

To verify that $\overline{\psi}\widetilde{u}_\tau\in C^\infty(I\rightarrow \mathcal{F})$, note that for any fixed $\tau>0$ and $m\in \mathbb{N}$,
\begin{eqnarray*}
\label{utauindomainP}
\int_I\rho_\tau(s-t)\left|\left|P^mH_{s-t}(\overline{\eta}^tu^t)\right|\right|_{\scaleto{L^2(X)}{6pt}}\,dt\lesssim \frac{1}{\tau^m}||\rho_\tau||_{\scaleto{L^\infty(\mathbb{R})}{6pt}}||\overline{\eta}u||_{\scaleto{L^2(I\times X)}{6pt}}<\infty.
\end{eqnarray*}
It follows that all $\partial_s^m(\widetilde{u}_\tau)$\ are well-defined as Bochner integrals and are in $L^\infty(I\rightarrow \mathcal{F})$. Hence $\widetilde{u}_\tau\in C^\infty(I\rightarrow \mathcal{F})$. The conclusion that $\overline{\psi}\widetilde{u}_\tau\in C^\infty(I\rightarrow \mathcal{F})$\ then follows from applying the gradient inequality (\ref{gradient111}).

In the next two subsections we prove Proposition \ref{L2prop1}. We present the proof in two steps. In the first step we express and split $\left|\left|\partial_\tau\partial_s^k(\overline{\psi}\widetilde{u}_\tau)\right|\right|_{\scaleto{L^2(I\times X)}{6pt}}$\ into three parts; in the second step we estimate each part and show that they are all bounded above independent of $0<\tau<1$\ and $0\leq k\leq n$.

\subsection{Proof of Proposition \ref{L2prop1} - Step 1}
We first compute $\partial_\tau \widetilde{u}_\tau(s,x)$.
\begin{eqnarray*}
	\partial_\tau \widetilde{u}_\tau(s,x)
	=\int_I\partial_\tau\rho_\tau(s-t)H_{s-t}(\overline{\eta}^tu^t)(x)\,dt
	=\int_I\partial_t\overline{\rho}_\tau(s-t) H_{s-t}(\overline{\eta}^tu^t)(x)\,dt.
\end{eqnarray*}
Recall (from the last subsection) that here  $\overline{\rho}_\tau(s-t)=\frac{s-t}{\tau}\rho_\tau(s-t)=\frac{s-t}{\tau^2}\rho\left(\frac{s-t}{\tau}\right)$. Let $\mathcal{T}:=\{\varphi\,\big|\,\left|\left|\varphi\right|\right|_{\scaleto{L^2(I\times X)}{6pt}}\leq 1,\,\varphi\in C^\infty_c(I\rightarrow L^2(X))\}$. Recall that $\overline{\psi}(s,x)=\psi(x)w(s)$. We have
\begin{eqnarray*}
	\left|\left|\partial_\tau\partial_s^k(\overline{\psi}\widetilde{u}_\tau)\right|\right|_{\scaleto{L^2(I\times X)}{6pt}}
		= \sup_{\varphi\in \mathcal{T}}\left<\psi\partial_\tau\partial_s^k(w\cdot\widetilde{u}_\tau),\  \varphi\right>_{\scaleto{L^2(I\times X)}{6pt}},
\end{eqnarray*}
where
\begin{eqnarray*}
\lefteqn{\left<\psi\partial_\tau\partial_s^k(w\cdot\widetilde{u}_\tau),\  \varphi\right>_{\scaleto{L^2(I\times X)}{6pt}}}\\
&=&\int_I\int_X\, \left\{\int_I\partial_s^k[w(s)\left(\partial_t\overline{\rho}_\tau(s-t)\right) H_{s-t}](\overline{\eta}^tu^t)(x)\,dt\right\} \psi(x)\varphi(s,x)\,dmds\\
&=&\int_I \int_I\int_X(\overline{\eta}^tu^t)(x)\cdot \partial_s^k[w(s)\left(\partial_t\overline{\rho}_\tau(s-t)\right) H_{s-t}](\psi\varphi^s)(x)\,dmdtds.
\end{eqnarray*}
The last line is by the Fubini Theorem (changing the integration order from $\int_I\int_X\int_I\,dtdmds$\ to $\int_I\int_I\int_X\,dmdtds$) and by the self-adjointness of $H_{s-t}$. Using the product rule for $\partial_t$\ to rewrite $w(s)\left(\partial_t\overline{\rho}_\tau(s-t)\right)H_{s-t}$\ in the square bracket as $\partial_t(w(s)\overline{\rho}_\tau(s-t)H_{s-t})-w(s)\overline{\rho}_\tau(s-t)\partial_tH_{s-t}$, altogether we get that
\begin{eqnarray*}
	\lefteqn{\left|\left|\partial_\tau\partial_s^k\left(\overline{\psi}\widetilde{u}_\tau\right)\right|\right|_{L^2\left(I\times X\right)}}\\
	&=&\sup_{\varphi\in \mathcal{T}}
	\left\{\int_I \int_I\int_X(\overline{\eta}^tu^t)(x)\cdot \partial_t\{\partial_s^k[w(s)\overline{\rho}_\tau(s-t)H_{s-t}](\psi\varphi^s)(x)\}\,dmdtds\right.\\
	&& \left.\hfill-\int_I \int_I\int_X(\overline{\eta}^tu^t)(x)\cdot \partial_s^k[w(s)\overline{\rho}_\tau(s-t)\partial_tH_{s-t}](\psi\varphi^s)(x)\,dmdtds\right\}.
\end{eqnarray*}	
In the last line, since $\partial_tH_{s-t}=P H_{s-t}$, the second term equals
\begin{eqnarray*}
	\lefteqn{\mbox{second term}}\\
	&=&\int_I\int_I\int_X(\overline{\eta}^tu^t)(x)\cdot P[\partial_s^k(w(s)\overline{\rho}_\tau(s-t)H_{s-t})(\psi\varphi^s)(x)]\,dm\,dtds\\
	&=&\int_I\int_I\mathcal{E}(\overline{\eta}^tu^t,\  \partial_s^k(w(s)\overline{\rho}_\tau(s-t)H_{s-t})(\psi\varphi^s))\,dtds.
\end{eqnarray*}
To simplify notation, let
\begin{eqnarray}
\label{abbreviation}
v_{\scaleto{k,\tau}{5pt}}(s,t,x):=\partial_s^k(w(s)\overline{\rho}_\tau(s-t)H_{s-t})(\psi\varphi^s)(x).
\end{eqnarray}
It is clear that for any fixed $\tau>0$, any $s,t\in I$, $v_{\scaleto{k,\tau}{5pt}}^{s,t}\in \mathcal{D}(P)$. Moreover, $v_{\scaleto{k,\tau}{5pt}}\in L^2(I^2\rightarrow\mathcal{D}(P))$. The result of the computations above can be written as
\begin{eqnarray}
\lefteqn{\left|\left|\partial_\tau\partial_s^k(\overline{\psi}\widetilde{u}_\tau)\right|\right|_{\scaleto{L^2(I\times X)}{6pt}}}\notag\\
&=&\sup_{\varphi\in \mathcal{T}}
\left\{\int_I \int_I\int_X\overline{\eta}(t,x)u(t,x)\cdot \partial_t[v_{\scaleto{k,\tau}{5pt}}(s,t,x)]\,dmdtds\right.\notag\\
\hfill&&\left.-\int_I\int_I\mathcal{E}(\overline{\eta}(t,\cdot)u(t,\cdot),\,  v_{\scaleto{k,\tau}{5pt}}(s,t,\cdot))\,dtds\right\}.\label{L2i}
\end{eqnarray}
Recall that $u$\ is a local weak solution on $I\times U$. If in (\ref{L2i}), $\overline{\eta}$\ is not grouped with $u$\ but appears on the same side with $v_{\scaleto{k,\tau}{5pt}}$, then (\ref{L2i}) is exactly $\int_I<f,\, \overline{\eta}v_{\scaleto{k,\tau}{5pt}}^s>\,ds$\ (the pairing is the $L^2(I\times X)$\ pairing). This observation inspires us to write (\ref{L2i}) as this term plus the difference, and then estimate them each separately. More precisely, we have
\begin{eqnarray*}
	\lefteqn{\left|\left|\partial_\tau\partial_s^k(\overline{\psi}\widetilde{u}_\tau)\right|\right|_{\scaleto{L^2(I\times X)}{6pt}}
		=(\ref{L2i})}\notag\\
	&\leq& \sup_{\varphi\in \mathcal{T}}|A_k(\tau, \varphi)|+\sup_{\varphi\in \mathcal{T}}|B_k(\tau, \varphi)|+\sup_{\varphi\in \mathcal{T}}|C_k(\tau, \varphi)|,
\end{eqnarray*}
where
\begin{eqnarray*}
\lefteqn{A_k(\tau,\varphi)=
\int_I\int_I\int_X(\overline{\eta}^tu^t)\cdot \partial_tv_{\scaleto{k,\tau}{5pt}}^{s,t}-u^t\cdot \partial_t(\overline{\eta}^tv_{\scaleto{k,\tau}{5pt}}^{s,t})\,dmdtds}\notag\\
&=&-\int_I\int_I\int_Xu(t,x)\cdot \partial_t\overline{\eta}(t,x)\cdot v_{\scaleto{k,\tau}{5pt}}(s,t,x)\,dmdtds;
\end{eqnarray*}
\begin{eqnarray*}
\lefteqn{B_k(\tau,\varphi)=-\int_I\int_I\mathcal{E}(\overline{\eta}^tu^t,\,v_{\scaleto{k,\tau}{5pt}}^{s,t})\,dtds+\int_I\int_I\mathcal{E}(u^t,\,\overline{\eta}^tv_{\scaleto{k,\tau}{5pt}}^{s,t})\,dtds}\notag\\
&=&-\int_I\int_I\int_Xd\Gamma(\overline{\eta}^tu^t,\,v_{\scaleto{k,\tau}{5pt}}^{s,t})\,dtds+\int_I\int_I\int_Xd\Gamma(u^t,\,\overline{\eta}^tv_{\scaleto{k,\tau}{5pt}}^{s,t})\,dtds;
\end{eqnarray*}
\begin{eqnarray*}
C_k(\tau,\varphi)=\int_I\left<f,\,\overline{\eta}v_{\scaleto{k,\tau}{5pt}}^{s}\right>_{\scaleto{L^2(I\times X)}{6pt}}\,ds= \int_I\int_I\int_X f(t,x)\overline{\eta}(t,x)v_{\scaleto{k,\tau}{5pt}}(s,t,x)\,dmdtds.
\end{eqnarray*}
\subsection{Proof of Proposition \ref{L2prop1} - Step 2}
Next we estimate $|A_k(\tau,\varphi)|$, $|B_k(\tau,\varphi)|$, and $|C_k(\tau,\varphi)|$\ individually. We will see that the upper bounds we find for $|A_k|,|B_k|,|C_k|$\ involve some $L^2$\ or $\mathcal{E}_1$\ norms of the local weak solution $u$\ on some precompact subsets of $I\times X$\ (hence the norms are well-defined). To conveniently express these norms of $u$, we introduce a nice (product) cut-off function that lives in (i.e., has compact support in) $I\times U$\ and is flat $1$\ on some open set that covers the supports of all other cut-off functions in the whole proof. We denote this cut-off function by $\overline{\Psi}(t,x)=n(t)\Psi(x)$. It can be determined after all other nice (product) cut-off functions in the proof of Theorem \ref{L2thm} are being introduced.

For $A_k(\tau,\varphi)$, note that $\partial_t\overline{\eta}(t,x)$\ is only nonzero for $t\in \left(J_{\scaleto{\overline{\eta}}{6pt}}\right)^c$\ (i.e., away from where $\overline{\eta}\equiv 1$) and $s\in I_{\scaleto{\overline{\psi}}{6pt}}\Subset J_{\scaleto{\overline{\eta}}{6pt}}$\ (because of $w(s)$). Hence for small $\tau$, more precisely, for $\tau<\min\left\{d\left(I_{\scaleto{\overline{\psi}}{6pt}},\,\left(J_{\scaleto{\overline{\eta}}{6pt}}\right)^c\right)/2,\,1/2\right\}=:c_0$,
\begin{eqnarray*}
	\partial_t\overline{\eta}(t,x)v_{\scaleto{k,\tau}{5pt}}(s,t,x)\equiv 0,
\end{eqnarray*}
and $A_k(\tau,\varphi)=0$\ for $0<\tau<c_0$. For $\tau\geq c_0$, first note that
\begin{eqnarray*}
\lefteqn{\left|\left|\partial_s^k(w(s)\overline{\rho}_\tau(s-t)H_{s-t})(\psi\varphi^s)\right|\right|_{\scaleto{L^2(X)}{6pt}}}\\
&\leq& 2^k\frac{2||w||_{\scaleto{C^k}{6pt}}||\rho||_{\scaleto{C^k}{6pt}}}{\tau^{k+1}}\max_{0\leq b\leq k}\left|\left|\partial_s^bH_{s-t}\right|\right|_{\scaleto{L^2(X)\rightarrow L^2(X)}{6pt}}\left|\left|\psi\varphi^s\right|\right|_{\scaleto{L^2(X)}{6pt}},
\end{eqnarray*}
where $||\partial_s^bH_{s-t}||_{2\rightarrow 2}=||P^bH_{s-t}||_{2\rightarrow 2}\leq\left(b/e(s-t)\right)^b\lesssim (1/\tau)^b$\ since $\tau<s-t<2\tau$\ for $\bar{\rho}_\tau$\ to be nonzero. So
\begin{eqnarray*}
\left|\left|\partial_s^k(w(s)\overline{\rho}_\tau(s-t)H_{s-t})(\psi\varphi^s)\right|\right|_{\scaleto{L^2(X)}{6pt}}\leq \frac{C(k,\overline{\psi},\rho)}{\tau^{2k+1}}\left|\left|\varphi^s\right|\right|_{\scaleto{L^2(X)}{6pt}}.
\end{eqnarray*}
It follows that
\begin{eqnarray*}
	\lefteqn{\hspace{-.5in}\left|A_k(\tau,\varphi)\right|
	=\left|\int_I\int_I\int_Xu^t\cdot \partial_t\overline{\eta}^t\cdot\partial_s^k(w(s)\overline{\rho}_\tau(s-t)H_{s-t})(\psi\varphi^s)\,dmdtds\right|}\\
	&\hspace{-.7in}\leq& \hspace{-.3in}\frac{C(k,\overline{\eta},\overline{\psi},\rho)}{\tau^{2k+1}}\int_{I_{\overline{\psi}}}\left|\left|\varphi^s\right|\right|_{\scaleto{L^2(X)}{6pt}}\,ds\int_{I_{\scaleto{\overline{\eta}}{5pt}}}\left|\left|u^t\right|\right|_{\scaleto{L^2(U_{\overline{\eta}})}{6pt}}\,dt\\
	&\hspace{-.7in}\leq& \hspace{-.3in}\widetilde{C}(k,\overline{\eta},\overline{\psi},\rho)\left|\left|\varphi\right|\right|_{\scaleto{L^2(I\times X)}{6pt}}\left|\left|\overline{\Psi} u\right|\right|_{\scaleto{L^2(I\times X)}{6pt}}.
\end{eqnarray*}
Here the constant $\widetilde{C}(k,\overline{\eta},\overline{\psi},\rho):=|I|^2C(k,\overline{\eta},\overline{\psi},\rho)c_0^{-(2k+1)}$\  depends only on the two cut-off functions $\overline{\eta}$, $\overline{\psi}$, the function $\rho$, and the sum of the binomial coefficients that is bounded by $2^k$. Note that $c_0=\min\left\{d\left(I_{\scaleto{\overline{\psi}}{6pt}},\,\left(J_{\scaleto{\overline{\eta}}{6pt}}\right)^c\right)/2,\,1/2\right\}$\ is determined by the cut-off functions. The function $\overline{\Psi}$\ is equal to $1$\ on the support of $\overline{\eta}$\ as introduced at the beginning of this subsection. So
\begin{eqnarray*}
	\max_{0\leq k\leq n}\widetilde{C}(k,\overline{\eta},\overline{\psi},\rho)<\infty.
\end{eqnarray*}
Denote some fixed upper bound of it by $C_A$, and recall that we take supremum over the functions $\varphi$\ with $\left|\left|\varphi\right|\right|_{\scaleto{L^2(I\times X)}{6pt}}\leq 1$. Hence 
\begin{eqnarray}
\label{estimateA}
\max_{0\leq k\leq n}\sup_{0<\tau<1}\sup_{\varphi\in \mathcal{T}}\left|A_k(\tau, \varphi)\right|\leq C_A(n,\overline{\eta},\overline{\psi},\rho)\cdot \left|\left|\overline{\Psi}u\right|\right|_{\scaleto{L^2(I\times X)}{6pt}}.
\end{eqnarray}
For $B_k(\tau,\varphi)$, since $\overline{\eta}(t,y)=l(t)\eta(y)$\ and $\eta\equiv 1$\ on $V_{\overline{\eta}}$, by the strong locality of the energy measure $d\Gamma$, the two terms in $B_k(\tau,\varphi)$\ satisfy that
\begin{eqnarray*}
	1_ {\scaleto{V_{\overline{\eta}}}{6pt}}\,d\Gamma(\overline{\eta}^tu^t,\,v_{\scaleto{k,\tau}{5pt}}^{s,t})=1_ {\scaleto{V_{\overline{\eta}}}{6pt}}\,d\Gamma(u^t,\,\overline{\eta}^tv_{\scaleto{k,\tau}{5pt}}^{s,t}).
\end{eqnarray*}
In other words,
\begin{eqnarray}
\label{L2ii}
d\Gamma(\overline{\eta}^tu^t,\,v_{\scaleto{k,\tau}{5pt}}^{s,t})-d\Gamma(u^t,\,\overline{\eta}^tv_{\scaleto{k,\tau}{5pt}}^{s,t})=d\Gamma(\overline{\eta}^tu^t,\,\Phi v_{\scaleto{k,\tau}{5pt}}^{s,t})-d\Gamma(u^t,\,\Phi\overline{\eta}^tv_{\scaleto{k,\tau}{5pt}}^{s,t})
\end{eqnarray}
for any ``bowl-shaped'' $\Phi$\ that equals $0$\ inside $V_{\overline{\eta}}$\ and becomes $1$\ before it reaches the boundary of $V_{\overline{\eta}}$, provided that the products of the functions are still in the domain $\mathcal{F}$. To later utilize the $L^2$\ Gaussian type upper bound to estimate, we take $\Phi$\ to be a nice cut-off function ``disjointly supported'' from $\psi$. More precisely, recall that $V_{\overline{\psi}}\Subset U_{\overline{\psi}}\Subset V_{\overline{\eta}}\Subset U_{\overline{\eta}}$. Let $V',U'$\ be two open sets that sit in the middle of this chain, and let $V'',U''$\ be two open sets at the right end of the chain, i.e.,
\begin{eqnarray*}
	V_{\overline{\psi}}\Subset U_{\overline{\psi}}\Subset V'\Subset U'\Subset V_{\overline{\eta}}\Subset U_{\overline{\eta}}\Subset V''\Subset U''\Subset U.
\end{eqnarray*}
Let $V_\Phi:=V''\setminus U'$\ and $U_\Phi:=U''\setminus V'$. Then $V_\Phi\Subset U_\Phi$, and there exists a nice cut-off function that is $1$\ on $V_\Phi$\ and supported in $U_\Phi$. We fix such a function and denote it by $\Phi$. The existence of $\Phi$\ is guaranteed by Lemma \ref{cut-offannuli}, or we can take the difference of two nice cut-off functions (for pairs $V''\subset U''$\ and $V'\subset U'$) and show that the difference still satisfies (\ref{cut-offbddenergy}). The nice cut-off function $\Phi$\ has the desired ``bowl-shape'', satisfies equation (\ref{L2ii}), and has disjoint support from $\psi$. In summary,
\begin{eqnarray*}
	\lefteqn{\left|B_k(\tau,\varphi)\right|}\\
	&=&\left|-\int_I\int_I\int_Xd\Gamma(\overline{\eta}^tu^t,\,\Phi v_{\scaleto{k,\tau}{5pt}}^{s,t})\,dtds+\int_I\int_I\int_Xd\Gamma(u^t,\,\Phi\overline{\eta}^tv_{\scaleto{k,\tau}{5pt}}^{s,t})\,dtds\right|,
\end{eqnarray*}
where by the Cauchy-Schwartz inequality for the strongly local part of $\mathcal{E}$\ and the H\"{o}lder's inequality applied to the integrals on $I$, we have
\begin{eqnarray*} 
	\lefteqn{\int_I\int_I\int_Xd\Gamma(\overline{\eta}^tu^t,\,\Phi v_{\scaleto{k,\tau}{5pt}}^{s,t})\,dtds}\\
	&\leq&\left(\int_I\int_I\int_Xd\Gamma(\overline{\eta}^tu^t,\,\overline{\eta}^tu^t)\,dtds\right)^{1/2} \left(\int_I\int_I\int_Xd\Gamma(\Phi v_{\scaleto{k,\tau}{5pt}}^{s,t},\,\Phi v_{\scaleto{k,\tau}{5pt}}^{s,t})\,dtds\right)^{1/2}\\
	&\leq& \left(|I|\cdot \int_I\mathcal{E}(\overline{\eta}^tu^t,\,\overline{\eta}^tu^t)\,dt\right)^{1/2}\left(\int_I\int_I\mathcal{E}(\Phi v_{\scaleto{k,\tau}{5pt}}^{s,t},\,\Phi v_{\scaleto{k,\tau}{5pt}}^{s,t})\,dtds\right)^{1/2},
\end{eqnarray*}	
and similarly (recall that $\overline{\Psi}$\ equals to $1$\ on the supports of all other nice cut-off functions),
\begin{eqnarray*}
	\lefteqn{\int_I\int_I\int_Xd\Gamma(u^t,\,\Phi\overline{\eta}^tv_{\scaleto{k,\tau}{5pt}}^{s,t})\,dtds
		=\int_I\int_I\int_Xd\Gamma(\overline{\Psi}u^t,\,\Phi\overline{\eta}^tv_{\scaleto{k,\tau}{5pt}}^{s,t})\,dtds}\\
	&\leq&\left(|I|\cdot \int_I\mathcal{E}(\overline{\Psi}u^t,\,\overline{\Psi}u^t)\,dtds\right)^{1/2} \left(\int_I\int_I\mathcal{E}(\Phi \overline{\eta}^tv_{\scaleto{k,\tau}{5pt}}^{s,t},\,\Phi \overline{\eta}^tv_{\scaleto{k,\tau}{5pt}}^{s,t})\,dtds\right)^{1/2}.
\end{eqnarray*}
Hence
\begin{eqnarray*}
	\lefteqn{\left|B_k(\tau,\varphi)\right|\leq C\left(\left|\left|\overline{\eta}u\right|\right|_{\scaleto{L^2(I\rightarrow \mathcal{F})}{6pt}}+ \left|\left|\overline{\Psi}u\right|\right|_{\scaleto{L^2(I\rightarrow \mathcal{F})}{6pt}}\right)\times}\\
	&& \left[\left(\int_{I^2}\mathcal{E}(\Phi v_{\scaleto{k,\tau}{5pt}}^{s,t},\,\Phi v_{\scaleto{k,\tau}{5pt}}^{s,t})\,dtds\right)^{1/2}+\left(\int_{I^2}\mathcal{E}(\Phi \overline{\eta}^tv_{\scaleto{k,\tau}{5pt}}^{s,t},\,\Phi \overline{\eta}^tv_{\scaleto{k,\tau}{5pt}}^{s,t})\,dtds\right)^{1/2}\right],
\end{eqnarray*}
it remains to estimate the two integrals in the square bracket. The estimate for the two terms are almost identical, so here we only do it for the second term, $\left(\int_{I^2}\mathcal{E}(\Phi \overline{\eta}^tv_{\scaleto{k,\tau}{5pt}}^{s,t},\,\Phi \overline{\eta}^tv_{\scaleto{k,\tau}{5pt}}^{s,t})\,dtds\right)^{1/2}$. Recall that $v_{\scaleto{k,\tau}{5pt}}^{s,t}\in L^2(I^2\rightarrow\mathcal{D}(P))$, we first want to move all $\Phi\overline{\eta}$\ to one side in order to rewrite the $\mathcal{E}$\ integral as an $L^2$\ integral with $Pv_{\scaleto{k,\tau}{5pt}}^{s,t}$. To this end we apply the gradient inequality. Using (\ref{gradient11}) applied to the nice cut-off function $\Phi\eta$, we get that
\begin{eqnarray*}
	\lefteqn{\int_{I^2}\mathcal{E}(\Phi \overline{\eta}^tv_{\scaleto{k,\tau}{5pt}}^{s,t},\,\Phi \overline{\eta}^tv_{\scaleto{k,\tau}{5pt}}^{s,t})\,dtds}\\
	&\leq& 2\int_{I^2}\left|\int_X(\Phi \overline{\eta}^t)^2v_{\scaleto{k,\tau}{5pt}}^{s,t}\cdot Pv_{\scaleto{k,\tau}{5pt}}^{s,t}\,dm\right|\,dtds+2C_2\int_{I^2}\int_{\scaleto{\mbox{supp}\{\Phi\overline{\eta}\}}{6pt}}(v_{\scaleto{k,\tau}{5pt}}^{s,t})^2\,dmdtds.
\end{eqnarray*}
Here $C_2$\ is associated with $\Phi\eta$. The second integral equals $2C_2\int_{I^2}\int_X1_{\scaleto{\Phi\eta}{5pt}} v_{\scaleto{k,\tau}{5pt}}^{s,t}\cdot v_{\scaleto{k,\tau}{5pt}}^{s,k}\,dmdtds$. Recall that by (\ref{abbreviation}),
\begin{eqnarray*}
	v_{\scaleto{k,\tau}{5pt}}(s,t,x)=\partial_s^k(w(s)\overline{\rho}_\tau(s-t)H_{s-t})(\psi\varphi^s)(x),
\end{eqnarray*}
which is essentially $P^mH_{s-t}\left(\psi\varphi^s\right)$\ for $0\leq m\leq k$\ (up to the derivatives of $w(s)\overline{\rho}_\tau(s-t)$\ which are bounded by some multiple of $1/(s-t)^{k+1}$). Moreover, by construction, $\Phi$\ and $\psi$\ have disjoint supports. Hence the two pairs of functions, $(\Phi \overline{\eta}^t)^2v_{\scaleto{k,\tau}{5pt}}^{s,t}$\ and $\psi\varphi^s$, $1_{\scaleto{\Phi\eta}{5pt}}v^{s,t}_{k,\tau}$\ and $\psi\varphi^s$, have disjoint supports, respectively. Thus we can apply the $L^2$\ Gaussian type upper bound and get
\begin{eqnarray*}
	\lefteqn{\int_I\int_I\mathcal{E}(\Phi \overline{\eta}^tv_{\scaleto{k,\tau}{5pt}}^{s,t},\,\Phi \overline{\eta}^tv_{\scaleto{k,\tau}{5pt}}^{s,t})\,dtds
	\leq  C(k,w,\rho,C_2)\sup_{\tau<r<2\tau}G(k+1,k,r) \times}&& \\
	&&\left\{\left|\left|(\Phi \overline{\eta})^2v_{\scaleto{k,\tau}{5pt}}\right|\right|_{\scaleto{L^2(I\times I\times X)}{6pt}}+\left|\left|1_{\scaleto{\Phi\eta}{5pt}} v_{\scaleto{k,\tau}{5pt}}\right|\right|_{\scaleto{L^2(I\times I\times X)}{6pt}}\right\} \left|\left|\psi\varphi\right|\right|_{\scaleto{L^2(I\times I\times X)}{6pt}}.
	\end{eqnarray*}
Here $G(k+1,k,r):=G_{\scaleto{U_\Phi,\,U_\psi}{6pt}}(k+1,k,r)$\ as defined under Assumption \ref{L2gaussian}. 	
Then to estimate $\left|\left|1_{\scaleto{\Phi\eta}{5pt}}v_{\scaleto{k,\tau}{5pt}}\right|\right|_{\scaleto{L^2(I\times I\times X)}{6pt}}$\ (and  $\left|\left|(\Phi \overline{\eta})^2v_{\scaleto{k,\tau}{5pt}}\right|\right|_{\scaleto{L^2(I\times I\times X)}{6pt}}$), note that
\begin{eqnarray*}
	\lefteqn{\left|\left|1_{\scaleto{\Phi\eta}{5pt}}v_{\scaleto{k,\tau}{5pt}}\right|\right|_{\scaleto{L^2(I\times I\times X)}{6pt}}^2}\\
	&=&\int_I\int_I 1_{\scaleto{\Phi\eta}{5pt}}v_{\scaleto{k,\tau}{5pt}}(s,t,x)\cdot \partial_s^k(w(s)\overline{\rho}_\tau(s-t)H_{s-t})(\psi\varphi^s)(x)\,dmdtds\\
	&\leq& C(k,w,\rho) \sup_{\tau<r<2\tau}G(k+1,k,r)\cdot\left|\left|1_{\scaleto{\Phi\eta}{5pt}}v_{\scaleto{k,\tau}{5pt}}\right|\right|_{\scaleto{L^2(I\times I\times X)}{6pt}}\left|\left|\psi\varphi\right|\right|_{\scaleto{L^2(I\times X)}{6pt}}|I|^{1/2},
\end{eqnarray*}
where the left-hand side and the right-hand side have a common factor $\left|\left|1_{\scaleto{\Phi\eta}{5pt}}v_{\scaleto{k,\tau}{5pt}}\right|\right|$. Combining the two estimates above gives
\begin{eqnarray*}
	\lefteqn{\int_I\int_I\mathcal{E}(\Phi \overline{\eta}^tv_{\scaleto{k,\tau}{5pt}}^{s,t},\,\Phi \overline{\eta}^tv_{\scaleto{k,\tau}{5pt}}^{s,t})\,dtds}\\
&\leq& C(k,\overline{\eta},\overline{\psi},\rho,\Phi)\sup_{\tau<r<2\tau}G(k+1,k,r)^2\left|\left|\varphi\right|\right|^2_{L^2\left(I\times X\right)}.
\end{eqnarray*}
Since $\sup_{0<r<1}G(k+1,k,r)$\ is finite by Assumption \ref{L2gaussian}, we obtain the estimate for $B_k(\tau,\varphi)$
\begin{eqnarray}
\label{estimateB}
\lefteqn{\max_{0\leq k\leq n}\sup_{0<\tau<1}\sup_{\varphi\in \mathcal{T}}\left|B_k(\tau, \varphi)\right|}\notag\\
&\leq& C_B(n,\overline{\eta},\overline{\psi},\rho,\Phi)\cdot \left(\left|\left|\overline{\eta}u\right|\right|_{\scaleto{L^2(I\rightarrow \mathcal{F})}{6pt}}+ \left|\left|\overline{\Psi}u\right|\right|_{\scaleto{L^2(I\rightarrow \mathcal{F})}{6pt}}\right),
\end{eqnarray}
where $C_B(n,\overline{\eta},\overline{\psi},\rho,\Phi)>0$\ is some constant.

Last, we estimate the term $C_k(\tau,\varphi)$. The idea is to first use the product rule for differentiation in time ($\partial_s$) to expand and rewrite
\begin{eqnarray*}
	\lefteqn{v_{\scaleto{k,\tau}{5pt}}^{s,t}=\partial_s^k(w(s)\overline{\rho}_\tau(s-t)H_{s-t})(\psi\varphi^s)}\\
	&=& \sum_{a=0}^{k}\binom{k}{a}\partial_s^{k-a}w(s)\cdot \partial_s^a(\overline{\rho}_\tau(s-t)H_{s-t})(\psi\varphi^s)\\
	&=& \sum_{a=0}^{k}\binom{k}{a}\partial_s^{k-a}w(s)\cdot (-1)^a\partial_t^a(\overline{\rho}_\tau(s-t)H_{s-t})(\psi\varphi^s),
\end{eqnarray*}
then move all the $\partial_t^a$\ on $\overline{\rho}_\tau(s-t)H_{s-t}$, $0\leq a\leq k$, to $\overline{\eta}f$, using integration by parts. More precisely,
\begin{eqnarray*}
	\lefteqn{\left|C_k(\tau,\varphi)\right|}\\
	&=&\left|\int_I\int_I\int_X f(t,x) \overline{\eta}(t,x)\,\partial_s^k(w(s)\overline{\rho}_\tau(s-t)H_{s-t})(\psi\varphi^s)(x)\,dmdtds\right|\\
	&=&\left|\sum_{a=0}^{k}\binom{k}{a}\int_I\partial_s^{k-a}w(s)\cdot\left<\partial_t^a(\overline{\eta}^tf^t),\ \overline{\rho}_\tau(s-t)H_{s-t}(\psi\varphi^s)\right>_{\scaleto{L^2(I\times X)}{6pt}}\,ds\right|\\
	&\leq& 2^k\left|\left|w\right|\right|_{\scaleto{C^k}{6pt}}\max_{0\leq a\leq k}\left|\int_I\int_I\int_X\partial_t^a(\overline{\eta}^tf^t)\cdot \overline{\rho}_\tau(s-t)H_{s-t}(\psi\varphi^s)\,dmdtds\right|.
\end{eqnarray*}
For any $0\leq a\leq k$, note that
\begin{eqnarray*}
	\lefteqn{\left|\int_I\int_X\partial_t^a(\overline{\eta}^tf^t)\cdot \int_I\overline{\rho}_\tau(s-t)H_{s-t}(\psi\varphi^s)\,ds\,dmdt\right|}\\
	&\leq& \int_I\left|\left|\partial_t^a(\overline{\eta}^tf^t)\right|\right|_{\scaleto{L^2(X)}{6pt}}\cdot \left|\left|\int_I\overline{\rho}_\tau(s-t)H_{s-t}(\psi\varphi^s)\,ds\right|\right|_{\scaleto{L^2(X)}{6pt}}\,dt\\
	&\leq& \int_I\left|\left|\partial_t^a(\overline{\eta}^tf^t)\right|\right|_{\scaleto{L^2(X)}{6pt}}\cdot \int_I\overline{\rho}_\tau(s-t)\left|\left|H_{s-t}(\psi\varphi^s)\right|\right|_{\scaleto{L^2(X)}{6pt}}\,ds\,dt.
	\end{eqnarray*}
To further estimate, we apply the H\"{o}lder's inequality to the integral in $t$. Since
\begin{eqnarray*}
\max_{0\leq a\leq k}\left(\int_I\left|\left|\partial_t^a(\overline{\eta}^tf^t)\right|\right|_{\scaleto{L^2(X)}{6pt}}\,dt\right)^{1/2}\leq \left|\left|\overline{\eta}f\right|\right|_{\scaleto{W^{k,2}(I\rightarrow L^2(X))}{6pt}};
\end{eqnarray*}
by the Jensen's inequality and that $H_t$\ is a contraction on $L^2(X)$,
\begin{eqnarray*}
\lefteqn{\left(\int_I\left(\int_I\overline{\rho}_\tau(s-t)\left|\left|H_{s-t}(\psi\varphi^s)\right|\right|_{\scaleto{L^2(X)}{6pt}}\,ds\right)^2 dt\right)^{1/2}}\\
&\leq& \left(\int_IA(t)\int_I\bar{\rho}_\tau(s-t)\left|\left|H_{s-t}(\psi\varphi^s)\right|\right|_{\scaleto{L^2(X)}{6pt}}^2\,ds\,dt\right)^{1/2}\leq 2\left|\left|\psi\varphi\right|\right|_{\scaleto{L^2(I\times X)}{6pt}},
\end{eqnarray*}
where $A(t):=\int_I\bar{\rho}_\tau(s-t)\,ds\leq 2$; altogether we get that
\begin{eqnarray*}
	\lefteqn{\left|C_k(\tau,\varphi)\right|}\\
	&\leq& C(k,w)\max_{0\leq a\leq k}\left|\int_I\int_X\partial_t^a(\overline{\eta}^tf^t)\cdot \int_I\overline{\rho}_\tau(s-t)H_{s-t}(\psi\varphi^s)\,ds\ dmdt\right|\\
	&\leq& C'(k,w)\left|\left|\overline{\eta}f\right|\right|_{\scaleto{W^{k,2}(I\rightarrow L^2(X))}{6pt}}\cdot \left|\left|\psi\varphi\right|\right|_{\scaleto{L^2(I\times X)}{6pt}}.
\end{eqnarray*}
Hence
\begin{eqnarray}
\label{estimateC}
\max_{0\leq k\leq n}\sup_{0<\tau<1}\sup_{\varphi\in\mathcal{T}}\left|C_k(\tau, \varphi)\right|\leq C_C(n,\overline{\psi})\cdot \left|\left|\overline{\eta}f\right|\right|_{\scaleto{W^{k,2}(I\rightarrow L^2(X))}{6pt}}
\end{eqnarray}
for some constant $C_C(n,\overline{\psi})>0$.

In the above estimates for $A_k,B_k,C_k$, we kept terms like $\left|\left|\overline{\eta}f\right|\right|_{\scaleto{W^{k,2}(I\rightarrow L^2(X))}{6pt}}$\ and $\left|\left|\overline{\Psi}u\right|\right|_{\scaleto{L^2(I\rightarrow \mathcal{F})}{6pt}}$, since $u,f$\ are only assumed to be locally in those function spaces. If we take any proper representative $u^\sharp,f^\sharp$, we can bound those norms by the corresponding norms of $u^\sharp$\ and $f^\sharp$.

Combining the estimates (\ref{estimateA}), (\ref{estimateB}), and (\ref{estimateC}) for $A_k(\tau,\varphi)$, $B_k(\tau,\varphi)$, and $C_k(\tau,\varphi)$, completes the proof for Proposition \ref{L2prop1}. To finish with the proof for $\{\overline{\psi}\widetilde{u}_\tau\}$\ being Cauchy in $W^{n,2}(I\rightarrow \mathcal{F})$, we still need to prove Proposition \ref{L2prop2}.

\subsection{Proof of Proposition \ref{L2prop2}} Our aim is to show that for $0\leq k\leq n$,
\begin{eqnarray}
\label{L2prop2toshow}
\int_I\mathcal{E}(\partial_\tau\partial_s^k(\overline{\psi}\widetilde{u}_\tau),\,\partial_\tau\partial_s^k(\overline{\psi}\widetilde{u}_\tau))\,ds\lesssim \frac{1}{\tau}.
\end{eqnarray}
Note that, $\widetilde{u}_\tau\in L^2(I\rightarrow\mathcal{D}(P))$\ for any fixed $\tau$, we apply the gradient inequality (\ref{gradient111}) to bound this $\mathcal{E}$\ integral of product functions by $L^2$\ integrals as below.
\begin{eqnarray*}
	\lefteqn{\int_I\mathcal{E}(\partial_\tau\partial_s^k(\overline{\psi}\widetilde{u}_\tau),\, \partial_\tau\partial_s^k(\overline{\psi}\widetilde{u}_\tau))\,ds}\\
	&\leq& 2\int_I\mathcal{E}(\psi^2\partial_\tau\partial_s^k(w(s)\widetilde{u}_\tau^s),\, \partial_\tau\partial_s^k(w(s)\widetilde{u}_\tau^s))\,ds\\
	&&+2C_2\int_I\int_{\scaleto{\mbox{supp}\{\psi\}}{6pt}}(\partial_\tau\partial_s^k(w(s)\widetilde{u}_\tau^s))^2\,dmds.
\end{eqnarray*}
Here $C_2$\ is associated with $\psi$. The proof of Proposition \ref{L2prop1} implies that the second term is bounded, i.e.,
\begin{eqnarray*}
	C_2\int_I\int_{\scaleto{\mbox{supp}\{\psi\}}{6pt}}(\partial_\tau\partial_s^k(w(s)\widetilde{u}_\tau^s))^2\,dmds=C_2\left|\left|\partial_\tau(w\widetilde{u}_\tau)\right|\right|^2_{\scaleto{W^{k,2}(I\rightarrow U_{\overline{\psi}})}{6pt}}\leq M_1
\end{eqnarray*}
for some constant $M_1$\ independent of $0<\tau<1$.

To estimate the first term, note that for any $s\in I$, $\widetilde{u}_\tau^s\in \mathcal{D}(P)$, so
\begin{eqnarray*}
	\lefteqn{\int_I\mathcal{E}(\psi^2\partial_\tau\partial_s^k(w(s)\widetilde{u}_\tau^s),\, \partial_\tau\partial_s^k(w(s)\widetilde{u}_\tau^s))\,ds}\\
	&=&\int_I\int_X\psi^2\partial_\tau\partial_s^k(w(s)\widetilde{u}_\tau^s)\cdot \partial_\tau\partial_s^k(w(s)P\widetilde{u}_\tau^s)\,ds\\
	&\leq& \left|\left|\partial_\tau\partial_s^k(\overline{\psi}\widetilde{u}_\tau)\right|\right|_{\scaleto{L^2(I\times X)}{6pt}} \left|\left|\partial_\tau\partial_s^k(\overline{\psi}P\widetilde{u}_\tau)\right|\right|_{\scaleto{L^2(I\times X)}{6pt}}.
\end{eqnarray*}
The first $L^2$\ norm is exactly the quantity treated in  Proposition \ref{L2prop1}, it is bounded independent of $0<\tau<1$. (\ref{L2prop2toshow}) follows after we show that the second $L^2$\ norm,
\begin{eqnarray*}
\left|\left|\partial_\tau\partial_s^k(\overline{\psi}P\widetilde{u}_\tau)\right|\right|_{\scaleto{L^2(I\times X)}{6pt}}\lesssim \frac{1}{\tau}.
\end{eqnarray*}
Replace $\widetilde{u}_\tau$\ by $P\widetilde{u}_\tau$\ in the proof of Proposition \ref{L2prop1}. By the same arguments, $\left|\left|\partial_\tau\partial_s^k(\overline{\psi}P\widetilde{u}_\tau)\right|\right|_{\scaleto{L^2(I\times X)}{6pt}}$\ breaks into three parts $A_k'(\tau,\varphi)$, $B_k'(\tau,\varphi)$, $C_k'(\tau,\varphi)$, and the estimates for $A_k'$\ and $B_k'$\ look almost identical to that for $A_k$\ and $B_k$. We write about the estimate for $C_k'(\tau,\varphi)$\ here. The only difference is that instead of using $||H_t||_{2\rightarrow 2}\leq 1$\ as in the estimate for $C_k$, we use $||PH_t||_{2\rightarrow 2}\leq 1/e\tau$\ here.
\begin{eqnarray*}
	\lefteqn{C_k'(\tau,\varphi)}\\
	&=& \int_I\int_I\int_X f(t,x) \overline{\eta}(t,x)\,\partial_s^k(w(s)\overline{\rho}_\tau(s-t)PH_{s-t})(\psi\varphi^s)(x)\,dm(x)dtds.
\end{eqnarray*}
As in the estimate for $C_k$, the estimate for $C_k'$\ comes down to estimate
\begin{eqnarray*}
	\lefteqn{\max_{0\leq a\leq k}\left|\int_I\int_X\partial_t^a(\overline{\eta}^tf^t)\cdot \int_I\overline{\rho}_\tau(s-t)PH_{s-t}(\psi\varphi^s)ds\,dmdt\right|}\\
	&\leq& \left|\left|\overline{\eta}f\right|\right|_{\scaleto{W^{k,2}(I\rightarrow L^2(X))}{6pt}}\cdot \left[2\int_I\int_I\overline{\rho}_\tau(s-t)\left|\left|PH_{s-t}(\psi\varphi^s)\right|\right|^2_{\scaleto{L^2(X)}{6pt}}\,ds dt\right]^{1/2}\\ 
	&\leq& \left|\left|\overline{\eta}f\right|\right|_{\scaleto{W^{k,2}(I\rightarrow L^2(X))}{6pt}}\sup_{s\in I}\left\{2\int_I\overline{\rho}_\tau(s-t)\,dt\right\}^{1/2}\left(\int_I\frac{1}{(e\tau)^2}\left|\left|\psi\varphi^s\right|\right|^2_{\scaleto{L^2(X)}{6pt}}\,ds\right)^{1/2}\\ 
	&\leq&\frac{2}{e\tau}\left|\left|\overline{\eta}f\right|\right|_{\scaleto{W^{k,2}(I\rightarrow L^2(X))}{6pt}}\cdot \left|\left|\psi\varphi\right|\right|_{\scaleto{L^2(I\times X)}{6pt}}.
\end{eqnarray*}
Hence indeed
\begin{eqnarray*}
	\lefteqn{\left|\left|\partial_\tau\partial_s^k(\overline{\psi}P\widetilde{u}_\tau)\right|\right|_{\scaleto{L^2(I\times X)}{6pt}}=\sup_{\substack{\left|\left|\varphi\right|\right|_{L^2(I\times X)}\leq 1\\
				\varphi\in C^\infty_c(I\rightarrow L^2(X))}}\left<\psi\partial_\tau\partial_s^k(w(s)P\widetilde{u}_\tau),\  \varphi\right>_{\scaleto{L^2(I\times X)}{6pt}}}\\
	&\leq&	\sup_{\varphi\in\mathcal{T}}\left|A_k'(\tau,\varphi)\right|+\sup_{\varphi\in\mathcal{T}}\left|B_k'(\tau,\varphi)\right|+\sup_{\varphi\in\mathcal{T}}\left|C_k'(\tau,\varphi)\right|\lesssim \frac{1}{\tau}.
\end{eqnarray*}
\subsection{Convergence of the approximate sequence in $L^2$\ sense}
Proposition \ref{L2prop1} and \ref{L2prop2} together imply that the approximate sequence $\{\overline{\psi}\widetilde{u}_\tau\}$\ is Cauchy in $W^{n,2}(I\rightarrow \mathcal{F})$. As we explained at the beginning of this section, to finish with the proof for Theorem \ref{L2thm}, it suffices to show that the approximate sequence converges to $\overline{\psi}u$\ in some weak sense. We prove the following slightly more general result.

For any function $w$\ in $L^2(I\times X)$, any $s\in I$, $\tau>0$, define
\begin{eqnarray}
\label{Ataudef}
A_\tau w(s,x):=\int_I\rho_\tau(s-t)H_{s-t}(w^t)(x)\,dt.
\end{eqnarray}
When $\tau$\ is not small enough, $A_\tau w$\ is the zero function. Similar to checking $\widetilde{u}_\tau\in C^\infty(I\rightarrow \mathcal{F})$\ for any $\tau>0$, we can show that for any $\tau>0$, $A_\tau w\in C^\infty(I\rightarrow \mathcal{F})$.

\begin{proposition}
	\label{L2convergence}
	Let $(H_t)_{t>0}$\ be any strongly continuous semigroup. Then $A_\tau w$\ defined as in (\ref{Ataudef}) converges to $w$\ in $L^2(I\times X)$, for any $w$\ in $L^2(I\times X)$.
\end{proposition}
\begin{proof}
	In this proposition we treat the larger class of strongly continuous semigroups $H_t$\ (not necessarily satisfying the Markov property and corresponding to a Dirichlet form), as roughly the same proof works under the weaker assumption. Such $H_t$\ satisfies that there exist some $M,\omega>0$, so that
	\begin{eqnarray*}
		\left|\left|H_t\right|\right|_{\scaleto{L^2(X)\rightarrow L^2(X)}{6pt}}\leq Me^{\omega t}.
	\end{eqnarray*}
	We first show that for any $w$\ in $C_c(I\rightarrow L^2(X))$, $A_\tau w$\ converges to $w$\ in $L^2(I\times X)$. As $C_c(I\rightarrow L^2(X))$\ is dense in $L^2(I\times X)$, and
	\begin{eqnarray*}
	\sup_{0<\tau<1}\left|\left|A_\tau\right|\right|_{\scaleto{L^2(I\times X)\rightarrow L^2(I\times X)}{6pt}}<+\infty,
	\end{eqnarray*}
	the statement thus holds for all $w$\ in $L^2(I\times X)$. For $w\in C_c(I\rightarrow L^2(X))$,
	\begin{eqnarray*}
		\lefteqn{\hspace{-.6in}\left|\left|A_\tau w-w\right|\right|_{\scaleto{L^2(I\times X)}{6pt}}
		\leq \left|\left|\int_I\rho_\tau(\cdot-t)\left[H_{\cdot-t}(w^t)-w^\cdot\right]dt\right|\right|_{\scaleto{L^2(I\times X)}{6pt}}}\\
		&&\hspace{.3in}+\left|\left|\left(1-\int_I\rho_\tau(\cdot-t)\,dt\right)w^\cdot\right|\right|_{\scaleto{L^2(I\times X)}{6pt}}.
	\end{eqnarray*}
    When $s-a\geq 2\tau$\ ($s\in I=(a,b)$), $\frac{s-t}{\tau}$\ runs over the full $(1,2)$\ as $t$\ runs over $I$, so $1-\int_I\rho_\tau(s-t)\,dt=1-1=0$. So $1-\int_I\rho_\tau(s-t)\,dt$\ can only be nonzero when $a<s<a+2\tau$, which is an interval of length $2\tau$. Since $w$\ is in $C_c(I\rightarrow L^2(X))$, we conclude that the second term
	\begin{eqnarray*}
	\left|\left|\left(1-\int_I\rho_\tau(\cdot-t)\,dt\right)w^\cdot\right|\right|_{\scaleto{L^2(I\times X)}{6pt}}\rightarrow 0\ \ \mbox{as }\tau\rightarrow 0.
	\end{eqnarray*}
	For the first term, we first write
	\begin{eqnarray*}
	\lefteqn{\int_I\rho_\tau(s-t)\left[H_{s-t}(w^t)-w^s\right]dt}\\
	&=& \int_I\rho_\tau(s-t)H_{s-t}(w^t-w^s)\,dt
	+\int_I\rho_\tau(s-t)\left[H_{s-t}(w^s)-w^s\right]dt.
	\end{eqnarray*}
	The $L^2$\ norm of the first part, $\left|\left|\int_I\rho_\tau(s-t)H_{s-t}(w^t-w^s)\,dt\right|\right|_{\scaleto{L^2(I\times X)}{6pt}}$, is bounded above by 
	\begin{eqnarray*}
		\lefteqn{\int_{\tau}^{2\tau}\rho_\tau(r)\left|\left|H_r(w^{\cdot-r}-w^\cdot)\right|\right|_{\scaleto{L^2(I\times X)}{6pt}}\,dr=\int_{\tau}^{2\tau}\rho_\tau\left|\left|\left|\left|H_r(w^{\cdot-r}-w^\cdot)\right|\right|_{\scaleto{L^2(X)}{6pt}}\right|\right|_{\scaleto{L^2(I)}{6pt}}dr}\\
		&\leq&\int_{\tau}^{2\tau}\rho_\tau(r)\left|\left|Me^{\omega r}\left|\left|w^{\cdot-r}-w^\cdot\right|\right|_{\scaleto{L^2(X)}{6pt}}\right|\right|_{\scaleto{L^2(I)}{6pt}}\,dr\leq C\sup_{\substack{\scaleto{s\in I}{5pt}\\\scaleto{\tau<r<2\tau}{5pt}}}\left|\left|w^{s-r}-w^s\right|\right|_{\scaleto{L^2(X)}{6pt}},
	\end{eqnarray*}
	which tends to $0$\ as $\tau\rightarrow 0$. The $L^2$\ norm of the second part has upper bound
	\begin{eqnarray*}
		\left|\left|\int_I\rho_\tau(s-t)\left[H_{s-t}(w^s)-w^s\right]dt\right|\right|_{\scaleto{L^2(I\times X)}{6pt}}\leq C\sup_{s\in I,\tau<r<2\tau}\left|\left|H_r(w^s)-w^s\right|\right|_{\scaleto{L^2(X)}{6pt}}.
	\end{eqnarray*}
	The right-hand side tends to $0$\ as $\tau\rightarrow 0$\ essentially because $\{s\mapsto H_r(w^s)\}_{r>0}\subset C_c(I\rightarrow L^2(X))$\ is equicontinuous in $s$. The details are as follows. First note that for any fixed $r>0$, any $s,t\in I$,
	\begin{eqnarray*}
	\lefteqn{\left|\left|H_r(w^s)-w^s\right|\right|_{\scaleto{L^2(X)}{6pt}}}\\
	&\leq& \left|\left|H_r(w^s-w^t)\right|\right|_{\scaleto{L^2(X)}{6pt}}+\left|\left|H_r(w^t)-w^t\right|\right|_{\scaleto{L^2(X)}{6pt}}+\left|\left|w^t-w^s\right|\right|_{\scaleto{L^2(X)}{6pt}}\\
	&\leq& 2Me^{\omega r}\left|\left|w^t-w^s\right|\right|_{\scaleto{L^2(X)}{6pt}}+\left|\left|H_r(w^t)-w^t\right|\right|_{\scaleto{L^2(X)}{6pt}}.
	\end{eqnarray*}
	For any $\epsilon>0$, any $s\in I$, there is some $\tau_0(s)>0$\ such that
	\begin{itemize}
		\item[(1)] for any $r<\tau_0(s)$, $\left|\left|H_r(w^s)-w^s\right|\right|_{\scaleto{L^2(X)}{6pt}}<\epsilon$\ (since $w^s\in L^2(X)$);
		\item[(2)] $\left|\left|w^t-w^s\right|\right|_{\scaleto{L^2(X)}{6pt}}<\epsilon$, for any $|s-t|<\tau_0(s)$\ (since $w\in C_c(I\rightarrow L^2(X))$).
	\end{itemize}
Let $B(s;\tau_0(s)):=(s-\tau_0(s),\,s+\tau_0(s))$. Since $\overline{I}$\ is compact and $\{B(s;\tau_0(s))\}_{s\in I}$\ covers $\overline{I}$, there exists some finite subcover $\{B(s_k;\tau_0(s_k))\}_{k=1}^{N}$. Hence there exists some fixed $\tau_0$\ ($\tau_0=\min\limits_{\scaleto{1\leq k\leq N}{4pt}}\{\tau_0(s_k)\}$) such that
\begin{itemize}
	\item[(1)] for any $r<\tau_0$, any $s_k$, $1\leq k\leq N$, $\left|\left|H_r(w^{s_k})-w^{s_k}\right|\right|_{\scaleto{L^2(X)}{6pt}}<\epsilon$;
	\item[(2)] for any $s\in I$, there exists some $s_k$\ such that $s\in B(s_k;\tau_0(s_k))$, so $\left|\left|w^s-w^{s_k}\right|\right|_{\scaleto{L^2(X)}{6pt}}<\epsilon$.
\end{itemize}
	Therefore,
	\begin{eqnarray*}
	\sup_{s\in I,\tau<r<2\tau}\left|\left|H_r(w^s)-w^s\right|\right|_{\scaleto{L^2(X)}{6pt}}\rightarrow 0\ \ \mbox{as }\tau\rightarrow 0.
	\end{eqnarray*}
This completes the proof for Proposition \ref{L2convergence}.
\end{proof}
Note that for any local weak solution $u$\ in Theorem \ref{L2thm}, the function $\widetilde{u}_\tau$\ is exactly $A_\tau(\overline{\eta}u)$. So Proposition \ref{L2convergence} applies to $\widetilde{u}_\tau$, and it follows that $\overline{\psi}\widetilde{u}_\tau\rightarrow \overline{\psi}u$\ in $L^2(I\times X)$\ as $\tau\rightarrow 0$. This completes the proof for Theorem \ref{L2thm}.
\subsection{Proof of Corollary \ref{L2corollary}}
In this subsection we prove Corollary \ref{L2corollary}, which says essentially that time derivatives of local weak solutions of the heat equation are still local weak solutions. 
\begin{proof}[Proof for Corollary \ref{L2corollary}]
	By Theorem \ref{L2thm}, $u$\ belongs to $\mathcal{F}_{\scaleto{\mbox{loc}}{5pt}}^n(I\times U)$. By definition of local weak solutions on $I\times U$, for any test function $\varphi$\ (and hence $\partial_t^k\varphi$\ for any $1\leq k\leq n$) in $\mathcal{F}_c(I\times U)\bigcap C_c^\infty(I\rightarrow \mathcal{F})$,
	\begin{eqnarray}
	\label{L2coreqn1.1}
	-\int_I\int_X u\, \partial_t^{k+1}\varphi\,dmdt+\int_I\mathcal{E}(u,\partial_t^k\varphi)\,dt=\int_I\int_X f \partial_t^k\varphi\,dmdt.
	\end{eqnarray}
	To show $\partial_t^ku$\ is a local weak solution of the heat equation (\ref{L2coreqn1}), intuitively it suffices to do integration by parts $k$\ times to move $\partial_t^k$\ to the $u$\ and $f$\ sides of the integrals. We now justify this procedure.
	
	For the justification of integration by parts for the first and third integrals in (\ref{L2coreqn1.1}), we only describe the first step and the remaining is clear by induction. By Fubini-Tonelli Theorem, suppose $\mbox{supp}\{\varphi\}\subset J\times V\Subset I\times U$, since
	\begin{eqnarray*}
		\int_I\int_X\left|u\,\partial_t^{k+1}\varphi\right|dmdt\leq \left|\left|u\right|\right|_{L^2(J\times U)}\cdot \left|\left|\varphi\right|\right|_{W^{k+1,2}(I\rightarrow L^2(U))}<\infty,
	\end{eqnarray*}
	we can switch the order of integration and get
	\begin{eqnarray*}
		-\int_I\int_X u\, \partial_t^{k+1}\varphi\,dmdt=-\int_X\int_I u\, \partial_t^{k+1}\varphi\,dtdm= \int_X\int_I \partial_tu\, \partial_t^{k}\varphi\,dtdm,
	\end{eqnarray*}
	where the second equality is by integration by parts and that $\varphi$\ is compactly supported in time. The same argument works for the integral 
	\begin{eqnarray*}
		\int_I\int_X f \partial_t^k\varphi\,dmdt=-\int_X\int_I \partial_tf\, \partial_t^{k-1}\varphi\,dtdm.
	\end{eqnarray*}
	For the second term in (\ref{L2coreqn1.1}), let $\varphi_n(t,\cdot):=\frac{\varphi(t+1/n,\cdot)-\varphi(t,\cdot)}{1/n}$, then $\varphi_n\rightarrow\varphi$\ in $C^\infty(I\rightarrow \mathcal{F})$. By the Cauchy-Schwartz inequality for $\mathcal{E}$, 
	\begin{eqnarray*}
	\left|\int_I\mathcal{E}(u,\varphi_n-\partial_t\varphi)\,dt\right|\leq \int_I(\mathcal{E}(u,u))^{1/2}(\mathcal{E}(\varphi_n-\partial_t\varphi,\,\varphi_n-\partial_t\varphi))^{1/2}\,dt\rightarrow 0
	\end{eqnarray*}
	as $n\rightarrow\infty$. Here $\mathcal{E}(u,u)$\ is understood as $\mathcal{E}(u^\sharp,u^\sharp)$\ where $u^\sharp\in \mathcal{F}(I\times X)$\ and agrees with $u$\ a.e. on a neighborhood of $\mbox{supp}\{\varphi\}$. For $n$\ large enough,
	\begin{eqnarray*}
	\int_I\mathcal{E}(u,\varphi_n)\,dt=\int_I\mathcal{E}(u_n,\varphi)\,dt,
	\end{eqnarray*}
	where $u_n(t,\cdot):=\frac{u(t-1/n,\cdot)-u(t,\cdot)}{1/n}$. Passing to the limit then shows
	\begin{eqnarray*}
	\int_I\mathcal{E}(u,\partial_t\varphi)\,dt=-\int_I\mathcal{E}(\partial_tu,\varphi)\,dt.
	\end{eqnarray*}
	In summary, after $k$\ times of integration by parts, (\ref{L2coreqn1.1}) becomes
$$(-1)^{k+1}\int_I\int_X\partial_t^ku\,\partial_t\varphi\, dmdt+(-1)^k\int_I\mathcal{E}(\partial_t^ku,\varphi)\,dt 
= (-1)^k\int_I\int_X\partial_t^kf\varphi\,dmdt,$$
thus $\partial_t^ku$\ is a local weak solution of (\ref{L2coreqn1}) on $I\times U$. The statement in Corollary \ref{L2corollary} for $f=0$\ then follows.
\end{proof}
\section{Ancient solutions} \setcounter{equation}{0}
\subsection{Statement of results}
In this section we generalize the results in \cite{app1,app2analyticity} on the structure of ancient solutions of the heat equation to the setting of Dirichlet spaces. As usual we assume that $(\mathcal{E},\mathcal{F})$\ is symmetric regular local. We call a local weak solution $u$\ of $(\partial_t+P)u=0$\ on $(-\infty,b)\times X$\ for some $b>0$\ an \textit{ancient (local weak) solution}. We assume $(X,\mathcal{E},\mathcal{F})$\ satisfies the assumption on existence of nice cut-off functions (Assumption \ref{assumptionbothcases}), and the following further assumption.
\begin{assumption}
	\label{appassumption}
	For any precompact open set $V\Subset X$, any $C_1>0$, any $n\in \mathbb{N}_+$, there exists 
	\begin{itemize}
	    \item[(1)] an exhaustion of $X$, $\{W_{V,i}^n\}_{i\in \mathbb{N}_+}$, with each set covering $V$. That is, $\{W_{V,i}^n\}_{i\in \mathbb{N}_+}$\ is a sequence of increasing open sets, satisfying
	\begin{eqnarray*}
	V\subset W^n_{V,1},\ W_{V,i}^n\Subset W_{V,i+1}^n,\ \bigcup_{i=1}^{\infty}W_{V,i}^n=X.
	\end{eqnarray*}
	\item[(2)] a sequence of cut-off functions $\{\varphi_{V,i}^n\}_{i\in \mathbb{N}_+}$, satisfying that each $\varphi_{V,i}^n=:\varphi_i$\ is a cut-off function for the pair $W_{V,i}^n\subset W_{V,i+1}^n$, i.e., $\varphi_i=1$\ on $W_{V,i}^n$, $\mbox{supp}\{\varphi_i\}\subset W_{V,i+1}^n$; $\varphi_i$\ further satisfies that for any $v\in \mathcal{F}$,
	\begin{eqnarray}
	\label{cut-offappineq}
	\int_Xv^2\,d\Gamma(\varphi_i,\varphi_i)\leq C_1 \int_X \varphi_i^2\,d\Gamma(v,v)+\frac{1}{n}\int_{\scaleto{\mbox{supp}\{\varphi_i\}}{6pt}}v^2\,dm.
	\end{eqnarray}
	\end{itemize}
\end{assumption}

When $X$\ is compact,  Assumption \ref{appassumption} trivially holds because we can take all $W_{V,i}^n$\ to be the whole space $X$, and take all $\varphi_i$\ to be the constant function $1$. For noncompact spaces, in the most classical setting $(\mathbb{R}^d,dx)$\ with the standard Dirichlet form, if $V\subset B(0;R)$\ where $B(0;R)$\ stands for the ball of radius $R$\ centered at the origin, we can take $W_{V,i}^n=B(0; R+cin^{1/2})$\ for some $c\geq 1$. It is standard to construct nice cut-off functions $\varphi_i$\ for each pair $W_{V,i}^n\subset W_{V,i+1}^n$\ such that
\begin{eqnarray*}
	d\Gamma(\varphi_i,\varphi_i)\leq \frac{2}{c^2 n}\,dx,
\end{eqnarray*}
which implies (\ref{cut-offappineq}) with $C_1=0$. See also the end of Section 6.1 and Section 7.1.

In the following theorems we consider two types of ancient solutions, one with polynomial $L^2$\ growth bound, and the other with exponential $L^2$\ growth bound. We first remark that for any ancient local weak solution $u$, by Theorem \ref{L2thm}, $u$\ is locally in $W^{\infty,2}((-\infty,b)\rightarrow \mathcal{F})\subset C^\infty((-\infty,b)\rightarrow \mathcal{F})$. As generalizations of results in \cite{app1,app2analyticity}, we state the following two theorems on the structure of ancient solutions in the Dirichlet space setting.

\begin{theorem}
	\label{appthmpoly}
	Let $(X,m)$\ be a metric measure space and $(\mathcal{E},\mathcal{F})$\ be a symmetric regular local Dirichlet form on $X$. Assume that the Dirichlet space $(X,\mathcal{E},\mathcal{F})$\ satisfies Assumption \ref{assumptionbothcases}, and when $X$\ is not compact, further satisfies Assumption \ref{appassumption}. Let $(H_t)_{t>0}$\ and $-P$\ be the corresponding semigroup and generator. Let $u$\ be a local weak solution of $(\partial_t+P)u=0$\ on $(-\infty,b)\times X$\ for some $b>0$, i.e., $u$\ is an ancient solution of the heat equation. Suppose $u$\ satisfies the $L^2$\ polynomial growth condition, namely, for any open subset $V\Subset X$, any $i\in \mathbb{N}_+$, there exist positive constants $b_{u},d_{u},C_{u,V,i}>0$\ ($b_u,d_u$\ are independent of $V,i$), such that for any $T>1$, $n\in \mathbb{N}_+$,
	\begin{eqnarray}
	\label{polygrowthbound}
	\left(\int_{[-T,\,0]\times W_{V,i}^n}|u(t,x)|^2\,dmdt\right)^{1/2}\leq C_{u,V,i}\max\left\{T^{d_{u}},\,n^{b_{u}}\right\}.
	\end{eqnarray}
	Then there exists some $N>0$\ such that for any $k>N$,
	\begin{eqnarray*}
		\partial_t^ku=0.
	\end{eqnarray*}
	More precisely, $u$\ is a polynomial in time, with
	\begin{eqnarray*}
		u(t,x)=u(0,x)+\partial_tu(0,x)\,t+\partial_t^2u(0,x)\,t^2\frac{1}{2!}+\cdots+\partial_t^Nu(0,x)\frac{1}{N!}t^N.
	\end{eqnarray*}
Here $N=\lfloor d_u\rfloor$, the largest integer not exceeding $d_u$.
\end{theorem}
For ancient solutions of the exponential growth type, we only need one sequence of exhaustion to get sufficient estimates, so we fix $n=1$\ and some precompact open set $V_0$, and consider the sequence $W_{V_0,i}^1=:W_i$\ only.
\begin{theorem}
\label{appthmexp}
Let $(X,m)$\ be a metric measure space and $(\mathcal{E},\mathcal{F})$\ be a symmetric regular local Dirichlet form on $X$. Assume that the Dirichlet space $(X,\mathcal{E},\mathcal{F})$\ satisfies Assumption \ref{assumptionbothcases}, and when $X$\ is not compact, further satisfies Assumption \ref{appassumption}. Let $(H_t)_{t>0}$\ and $-P$\ be the corresponding semigroup and generator. Let $u$\ be a local weak solution of $(\partial_t+P)u=0$\ on $(-\infty,b)\times X$\ for some $b>0$, i.e., $u$\ is an ancient solution of the heat equation. Suppose $u$\ satisfies the $L^2$\ exponential growth condition, namely, there exists some $c_u>0$, such that for any $T>1$, any $i\in \mathbb{N}_+$,
	\begin{eqnarray}
	\label{expgrowthbound}
	\int_{[-T,\,0]\times W_i}|u(t,x)|^2\,dmdt\leq e^{c_u(T+i)}.
	\end{eqnarray}
	Then $u$\ is analytic in $t\in (-\infty,0]$, in the sense that for any precompact open set $V\subset X$,
	\begin{eqnarray}
	\label{appcaseii0}
	\left|\left|u(t,\cdot)-\sum_{i=1}^{k}\frac{\partial_t^iu(0,\cdot)}{i!}t^i\right|\right|_{L^2(V)}\rightarrow 0\ \mbox{as }k\rightarrow \infty,
	\end{eqnarray}
	and the convergence is uniform in $t\in [a,0]$\ for any $a<0$.
\end{theorem}
We first make some remarks about the two theorems.
\begin{remark}
	\label{polyrmk}
	For Theorem \ref{appthmpoly}, if we denote $\frac{1}{k!}\partial_t^ku(0,x)=u_k(x)$\ and let $N=\lfloor d_u\rfloor$, then $\{u_k\}_{k=0}^{N}$\ satisfies 
	\begin{eqnarray*}
		-P u_k(x)&=&(k+1)\,u_{k+1}(x),\ \ \mbox{for }0\leq k\leq N-1,\\
		-P u_{N}(x)&=&0,
	\end{eqnarray*}
	both in the sense that for any $\varphi\in \mathcal{F}_c(X)$,
	\begin{eqnarray*}
		\mathcal{E}(u_k,\,\varphi)=(k+1)\int_X u_{k+1}\varphi\,dm,
	\end{eqnarray*}
	for any $1\leq k\leq N$\ ($u_{N+1}=0$). We call $u_k$\ a local weak solutions of $-Pu_k=(k+1)u_{k+1}$\ on $X$. Moreover, all $u_k$\ satisfy the $L^2$\ growth bound that for any precompact open set $V\Subset X$, any $i,n\in \mathbb{N}_+$, there exist constants $C_{u,V,i}>0$\ and $b_u>0$\ (independent of $V,i$), such that
	\begin{eqnarray*}
		\left(\int_{W_{V,i}^n}|u_k(x)|^2\,dm\right)^{1/2}\leq C_{u,V,i}\,n^{b_u}.
	\end{eqnarray*}
\end{remark}
\begin{remark}
	In Theorem \ref{appthmexp}, if we write $u(t,x)=\sum\limits_{k=0}^{\infty} \frac{a_k(x)}{k!}t^k$\ where the two sides equal in the above $L^2$\ sense, then the $a_k(x)$\ functions are $a_k(x)=\partial_t^ku(0,x)$. A Cacciopolli type estimate for local weak solutions can be derived from the proof of Proposition \ref{appproposition1}, namely, for any ancient local weak solution $v$\ of the heat equation $(\partial_t+P)v=0$, for any $T>0$, there exists some $K>0$\ such that
	\begin{eqnarray*}
		\sup_{t\in [-T,\,0]}\int_{W_i}|v(t,x)|^2\,dm\leq K\int_{[-T-1,\,0]\times W_{i+1}}|v(t,x)|^2\,dmdt,
	\end{eqnarray*}
	where $i\in \mathbb{N}_+$\ and $W_i$\ is defined as in Theorem \ref{appthmexp}. For any $k$, by taking $T=1$\ and $v=\partial_t^ku(t,x)$\ which by Corollary \ref{L2corollary} is a local weak solution, and by using the inequality in Proposition \ref{appcor} given in the next section, we get that $a_k(x)$\ satisfies the $L^2$\ upper bound
	\begin{eqnarray*}
		\int_{W_i}|a_k(x)|^2\,dm\leq C_{u,k}\,e^{c_u(i+5k)}.
	\end{eqnarray*}
\end{remark}
\begin{remark}
	The conclusion in Theorem \ref{appthmexp} is in the $L^2$\ sense. By Proposition \ref{appproposition1}, for any ancient (local weak) solution $u$, it is also true that the partial sum $\sum\limits_{i=0}^{k}\frac{\partial_t^iu(0,x)}{i!}t^i$\ tends to $u$\ in the energy integral over any precompact set as $k$\ tends to infinity, uniformly in time on any finite interval. If the (essential) supremum of $u$\ over each time-space cylinder can be controlled by the $L^2$\ integral of $u$\ over the same cylinder, then we can make the conclusion in Theorem \ref{appthmexp} an (m-a.e.) pointwise conclusion. For example, some ultracontractivity property of the heat semigroup is sufficient for this purpose.
\end{remark}
As a corollary for Theorem \ref{appthmpoly}, we recover in the current setting the dimension result in \cite{app1} under an additional condition on the polynomial growth of the $W_{V,i}^n$\ sets, $V\Subset X$\ is an arbitrarily fixed open set. We first define the function spaces. For each $d,b\in \mathbb{N}_+$, let $\mathcal{P}_{d,b}(X)$\ denote the vector space of all ancient (local weak) solutions $u$\ of $(\partial_t+P)u=0$\ satisfying that for any $n,i\in \mathbb{N}_+$, there exists some constant $C_{u,V,i}>0$, such that
\begin{eqnarray}
\label{appinfinitybound}
\esssup_{[-T,\,0]\times W_{V,i}^n}|u(t,x)|\leq C_{u,V,i}\max\left\{T^d,n^b\right\}.
\end{eqnarray}
Let $\mathcal{H}_{b}(X)$\ denote the vector space of all local weak solutions $v$\ of $Pv=0$\ on $X$\ with polynomial growth bound
\begin{eqnarray*}
	\esssup_{x\in W_{V,i}^n}|v(x)|\,dm\leq D_{v,V,i}\,n^b.
\end{eqnarray*}
Here $D_{v,V,i}>0$\ is some constant, $i\in \mathbb{N}_+$.
\begin{corollary}
	Under the hypotheses of Theorem \ref{appthmpoly}, and further assume that for some precompact open set $V\Subset X$, for any $n,i\in \mathbb{N}_+$, the sets $W_{V,i}^n$\ satisfy some polynomial volume growth bound
	\begin{eqnarray}
	\label{polyvolgrowth}
	m(W_{V,i}^n)\leq E_{V,i}\,n^a
	\end{eqnarray}
	where $E_{V,i},a>0$\ are constants. Then
	\begin{eqnarray*}
		\mbox{dim}\,\mathcal{P}_{d,b}(X)\leq (d+1)\,\mbox{dim}\,\mathcal{H}_{b}(X).
	\end{eqnarray*}
\end{corollary}
\begin{proof}
	Take any $u\in \mathcal{P}_{d,b}(X)$. Note that (\ref{appinfinitybound}) and (\ref{polyvolgrowth}) together imply the $L^2$\ growth condition (\ref{polygrowthbound}) for some $d_u,b_u$\ (e.g. $d_u=2d+1,b_u=2b+a$). Hence by Theorem \ref{appthmpoly}, $u$\ is a polynomial in time with $\partial_t^ku=0$\ for $k$\ large enough. The growth condition (\ref{appinfinitybound}) then implies that $\partial_t^ku=0$\ for $k>d$. As in Remark \ref{polyrmk}, let $u_k=\frac{1}{k!}\partial_t^ku(0,x)$. By the discussion in \cite{app1}, for any fixed $t_0,t_1,\cdots,t_d\in [-\frac{1}{2},0]$\ that are distinct, there exist numbers $b_j^k\geq 0$\ such that for any $0\leq k\leq d$,
	\begin{eqnarray*}
		u_k(x)=\sum_{j=0}^{d}b_j^ku(t_j,x).
	\end{eqnarray*}
	Since all $|t_j|<1$, and $u\in \mathcal{P}_{d,b}(X)$, for any $i,n\in \mathbb{N}_+$,
	\begin{eqnarray*}
		\esssup_{x\in W_{V,i}^n}|u_k(x)|\leq \max_{0\leq j\leq d}\left|b_j^k\right|\cdot C_{u,V,i}\,n^b.
	\end{eqnarray*}
	This implies that $u_k\in \mathcal{H}_b(X)$. By the arguments in \cite{app1}, it follows that
	\begin{eqnarray*}
	\mbox{dim}\,\mathcal{P}_{d,b}(X)\leq (d+1)\,\mbox{dim}\,\mathcal{H}_{b}(X).
	\end{eqnarray*}
\end{proof}

We make some final remarks about the two assumptions on existence of cut-off functions, Assumptions \ref{assumptionbothcases} and \ref{appassumption}.

First, Assumption \ref{assumptionbothcases} focuses on for any fixed pair of open sets $V\Subset U$, in particular they could be very close to each other, for any small $C_1$, the existence of a cut-off function for the pair $V\subset U$\ that satisfies (\ref{cut-offbddenergy}). There $C_2$\ depends on $C_1,U,V$\ and is usually a large number when $C_1$\ is small and $U,V$\ are close. Intuitively, the cut-off function is steep. In contrast, in Assumption \ref{appassumption}, the focus is on for any fixed beginning set $V\Subset X$\ and fixed $C_1$, for small $C_2$\ ($C_2=\frac{1}{n}$\ for large $n$), the existence of an exhaustion and cut-off functions for each pair of adjacent open sets therein. Intuitively, for large $n$, the sets in the exhaustion are far apart, and the cut-off functions have flat shapes.

Regarding the validity of Assumption \ref{appassumption}, we remark that in general Dirichlet spaces which have some notion of distance that interacts well with the energy measure this assumption is satisfied. Roughly speaking, for large $n$, to find $W_{V,i}^n$'s and $\varphi_i$'s, we just require $W_{V,i}^n$ and the complement of  $W_{V,i+1}^n$  to be separated by a large enough distance. For example, consider a Dirichlet space $(X,m,\mathcal{E},\mathcal{F})$\ that admits ``nice metric cut-off functions'', namely, there exists some distance $d$\ that defines the same topology of $X$, such that for any pair of open sets $V\Subset U$, any $0<C_1<1$, there exists some nice cut-off function $\varphi$\ satisfying that for any $v\in \mathcal{F}$,
\begin{eqnarray*}
\int_Xv^2\,d\Gamma(\varphi,\varphi)\leq C_1\int_X \varphi^2\,d\Gamma(v,v)+C(C_1)\cdot d(V,U^c)^{-\beta}\int_{\scaleto{\mbox{supp}\{\varphi\}}{6pt}}v^2\,dm,
\end{eqnarray*}
where $\beta>0$; $C(C_1)$\ is some positive function of $C_1$. Assume $V\subset B(0;R)$. Then we can take $W_{V,i}^n=B(x_0; R+ain^{1/\beta})$, for any $a$\ satisfying $a^\beta\geq C(C_1)$. Concretely,
\begin{itemize}
	\item[(1)] when the Dirichlet space admits a nice intrinsic distance, it is a special case of the discussion above with $\beta=2$;
	\item[(2)] when the Dirichlet space is the standard Dirichlet form on the Sierpinski gasket and $d$\ is the Euclidean metric, the discussion above applies with $\beta=\frac{\log 5}{\log 2}$, which is the walk dimension $d_w$\ of the Sierpinski gasket. 
\end{itemize}

\subsection{Proof of Theorem \ref{appthmpoly} and Theorem \ref{appthmexp}}
\subsubsection{Overview and a key estimate}
There are two difficulties in generalizing the structure results on ancient heat equation solutions to the setting of Dirichlet spaces. The first difficulty is to introduce proper assumptions on the existence of  cut-off functions in order to adapt estimates of the form  $$|\int fg\nabla v\cdot \nabla w\,dx|\leq \left(\int (fg)^2\,dx\right)^{1/2}\left(\int |\nabla v|^2|\nabla w|^2\,dx\right)^{1/2}$$ to the setting of energy measures, especially when the energy measure is singular with respect to the measure $m$\ in the metric measure space $(X,m)$. The more essential difficulty is about whether the time derivatives of an ancient local weak solution is still an ancient weak solution, and this is addressed by the main part of this paper with an affirmative answer (Corollary \ref{L2corollary}).

In this section we state the key estimate and use it to prove the two theorems stated in the previous section. The estimate is about bounding the $L^2$\ integral of time derivatives of an ancient solution $u$\ over some time-space cylinder by the $L^2$\ integral of $u$\ over some larger time-space cylinder, where the spatial sets are ones in an exhaustion of $X$. Here we treat the exhaustion with less precision in the sense that we do not specify an initial precompact open set $V$, and instead of specifying the constant $\frac{1}{n}$\ in (\ref{cut-offappineq}) of Assumption \ref{appassumption}, we consider an exhaustion with cut-off functions to correspond to constants $C_1$\ and $C$. In other words, we consider some exhaustion denoted by $\{W_i\}$\ together with cut-off functions $\{\varphi_i\}$, where each $\varphi_i$\ is a cut-off function for the pair $W_i\subset W_{i+1}$, and satisfies for any $v\in \mathcal{F}$,
\begin{eqnarray*}
\int_Xv^2\,d\Gamma(\varphi_i,\varphi_i)\leq C_1 \int_X \varphi_i^2\,d\Gamma(v,v)+C\int_{\scaleto{\mbox{supp}\{\varphi_i\}}{6pt}}v^2\,dm.
\end{eqnarray*}
We call such $\left(\{W_i\}_{i\in \mathbb{N}_+},\{\varphi_i\}_{i\in \mathbb{N}_+}\right)$\ an exhaustion of $X$\ corresponding to $C_1$, $C$.
\begin{proposition}
	\label{appcor}
	Let $(X,m)$\ be a metric measure space and $(\mathcal{E},\mathcal{F})$\ be a symmetric regular local Dirichlet form on $X$. Assume that the Dirichlet space $(X,\mathcal{E},\mathcal{F})$\ satisfies Assumption \ref{assumptionbothcases}, and when $X$\ is not compact, further satisfies Assumption \ref{appassumption}. Let $(H_t)_{t>0}$\ and $-P$\ be the corresponding semigroup and generator. Let $u$\ be an ancient (local weak) solution of $(\partial_t+P)u=0$. Let $J=[c,0]$, $c<0$, be any finite interval with a fixed right end, let $J_{-s}:=[c-s,0]$\ for any $s>0$.Take $C_1=\frac{1}{16}$\ and fix an arbitrary $C>0$. Let $\left(\{W_i\},\{\varphi_i\}\right)$\ be an exhaustion of $X$\ corresponding to $C_1$, $C$. Then for any $i,k\in \mathbb{N}_+$,
	\begin{eqnarray*}
		\int_{J}\int_{W_i}\left(\partial_t^ku\right)^2\,dmdt\leq \left(1200\left(C+\frac{1}{r}\right)^2\right)^k\int_{J_{-2kr}}\int_{W_{i+3k}}u^2\,dmdt.
	\end{eqnarray*}
\end{proposition}
In Lemma \ref{applemma} below it can be seen that $C_1$\ can take any value less than $\frac{1}{9}$, we take $C_1=\frac{1}{16}$\ for convenience. We now use Proposition \ref{appcor} to prove Theorem \ref{appthmpoly} and Theorem \ref{appthmexp}. The proof of the proposition is given in Section 6.3.
\subsubsection{Proof of Theorem \ref{appthmpoly}}
To show $\partial_t^ku=0$\ for $k$\ large enough, we show that the $L^2$\ integral of such $\partial_t^ku$\ over any large time-space cylinder is zero. Consider an arbitrary cylinder $[-T,\,0]\times V$\ where $V\Subset X$\ and $T>0$, recall that for such a precompact open set $V$\ and for any $n\in \mathbb{N}_+$, Assumption \ref{appassumption} guarantees the existence of an exhaustion $\left(\{W_{V,i}^n\},\{\varphi_i\}\right)$\ of $X$\ corresponding to $C_1=\frac{1}{16}$, $C=\frac{1}{n}$. Apply Proposition \ref{appcor} to $\left(\{W_{V,i}^n\},\{\varphi_i\}\right)$\ with $i=1$, $J=[-T,\,0]\subset (-\infty,0]$. Taking $r=n$, we have
\begin{eqnarray*}
	\int_{J}\int_{W_{V,1}^n}\left(\partial_t^ku\right)^2\,dmdt\leq \left(\frac{5000}{n^2}\right)^k\int_{J_{-2kr}}\int_{W_{V,1+3k}^n}u^2\,dmdt.
\end{eqnarray*}

The strategy is to let $n$\ tend to infinity in the above inequality. This has the effect of taking ``balls'' with bigger and bigger distance from each other, so that the right-hand side, which is bounded by some rational function in $n$, tends to zero as $n$\ tends to infinity.

More precisely, we have
\begin{eqnarray*}
	\int_{[-T,\,0]\times W_{V,1}^n}\left(\partial_t^ku\right)^2\,dmdt\leq \frac{5000^k \left(C_{u,V,1+3k}\max\left\{(\lceil T\rceil +2kn)^{d_u},n^{b_u}\right\}\right)^2}{n^{2k}}.
\end{eqnarray*}
Since for any $2k>2(d_u+b_u)$, the right-hand side tends to $0$\ as $n$\ tends to infinity ($k$\ is fixed), by the discussion above, we conclude that for $k>d_u+b_u$, 
\begin{eqnarray*}
	\partial_t^ku=0.
\end{eqnarray*}
This concludes the proof that $u$\ is a polynomial in $t$. Applying the growth bound (\ref{polygrowthbound}) to $u$\ in the explicit polynomial form, we conclude that $\partial_t^ku=0$\ for $k>d_u$.
\subsubsection{Proof of Theorem \ref{appthmexp}} 
By Taylor expansion formula (expansion in $t$), for any $t<0$,
\begin{eqnarray*}
	u(t,x)=\sum_{i=0}^{k}\frac{\partial_t^iu(0,x)}{i!}t^i+\int_{0}^{t}\partial_s^{k+1}u(s,x)\frac{(t-s)^k}{k!}\,ds
\end{eqnarray*}
as $L^2$\ functions in $x$. So to prove the statement in Theorem \ref{appthmexp}, we prove that for any precompact open set $V$, any $a<0$,
\begin{eqnarray}
\sup_{a\leq t\leq 0}\int_V\left(\int_{0}^{t}\partial_s^{k+1}u(s,x)\frac{(t-s)^k}{k!}\,ds\right)^2dm(x)\rightarrow 0
\end{eqnarray}
as $k\rightarrow \infty$. We first bound the integral by
\begin{eqnarray}
\lefteqn{\left|\int_V\left(\int_{0}^{t}\partial_s^{k+1}u(s,x)\frac{(t-s)^k}{k!}\,ds\right)^2dm(x)\right|}\notag\\
&\leq& \frac{|t|}{(k!)^2}\int_V\int_{t}^{0}\left(\partial_s^{k+1}u(s,x)\cdot (t-s)^k\right)^2\,dsdm\notag\\
&\leq& \frac{|t|^{2k+1}}{(k!)^2}\int_V\int_{t}^{0}\left(\partial_s^{k+1}u(s,x)\right)^2\,dsdm.\label{appcaseii1}
\end{eqnarray}
Recall the notation introduced in the statement of Theorem \ref{appthmexp}, i.e., $W_j:=W_{V_0,j}^1$\ for some fixed $V_0\Subset X$. Intuitively by fixing $n=1$\ (or any fixed integer), we are looking at open sets whose sizes grow linearly. Since $V\Subset X$\ and $\{W_j\}$\ is an exhaustion of $X$, there exists some $j_0$\ such that for all $j\geq j_0$, $V\subset W_j$. By Proposition \ref{appcor}, for any $r>0$,
\begin{eqnarray*}
	\lefteqn{\int_{t}^{0}\int_{W_j}\left(\partial_s^{k+1}u(s,x)\right)^2\,dmdt}\\
	&\leq& \left(1200\left(1+\frac{1}{r}\right)^2\right)^{k+1}\int_{t-2(k+1)r}^{0}\int_{W_{j+3(k+1)}}u(s,x)^2\,dmdt.
\end{eqnarray*}
By the exponential growth assumption (\ref{expgrowthbound}) on $u$, we conclude that (take for example $r=1$) for any $t\in [a,0]$,
\begin{eqnarray*}
	\int_{t}^{0}\int_{W_j}\left(\partial_s^{k+1}u(s,x)\right)^2\,dmdt
	\leq (5000)^{k+1}e^{c_u\left(|a|+j+5(k+1)\right)}.
\end{eqnarray*}
Substituting this bound back to (\ref{appcaseii1}), note that $V\subset W_j$, we get
\begin{eqnarray*}
	\lefteqn{\left|\int_V\left(\int_{0}^{t}\partial_s^{k+1}u(s,x)\frac{(t-s)^k}{k!}\,ds\right)^2dm(x)\right|}\\
	&\leq& \frac{|a|^{2k+1}}{(k!)^2}\cdot (5000)^{k+1}e^{c_u(|a|+j+5(k+1))}
	\rightarrow 0\ \ (k\rightarrow \infty).
\end{eqnarray*}
This completes the proof of Theorem \ref{appthmexp}.
\subsection{Proof of the key estimate}
In this subsection we give the proof for Proposition \ref{appcor}. Proposition \ref{appcor} follows from the following proposition by iteration.
\begin{proposition}
	\label{appproposition1}
	Let $(X,m)$\ be a metric measure space and $(\mathcal{E},\mathcal{F})$\ be a symmetric regular local Dirichlet form on $X$. Assume that the Dirichlet space $(X,\mathcal{E},\mathcal{F})$\ satisfies Assumption \ref{assumptionbothcases}, and when $X$\ is not compact, further satisfies Assumption \ref{appassumption}. Let $(H_t)_{t>0}$\ and $-P$\ be the corresponding semigroup and generator. Let $u$\ be an ancient (local weak) solution of $(\partial_t+P)u=0$. Let $J=[c,0]$, $c<0$, be any finite interval with a fixed right end. Take $C_1=\frac{1}{16}$\ and fix an arbitrary $C>0$. Let $\left(\{W_i\},\{\varphi_i\}\right)$\ be an exhaustion of $X$\ corresponding to $C_1,C$. Then for any $r>0$, there exist constants $K_1,K_2$\ (dependent on $C$\ and $r$) such that for any $i\in \mathbb{N}_+$,
	\begin{eqnarray*}
		\int_{J}\int_{W_i}(\partial_tu)^2\,dmdt&\leq& K_1(C,r)\int_{J_{-r}}\left(\int_{W_{i+2}}d\Gamma(u,u)+\int_{W_{i+2}}u^2\,dk\right)dt\\
		&\leq& K_2(C,r)\int_{J_{-2r}}\int_{W_{i+3}}u^2\,dmdt.
	\end{eqnarray*}
	Here $J_{-s}:=[c-s,0]$\ for any $s>0$, and
	\begin{eqnarray*}
	K_1(C,r)=200\left(C+\frac{1}{r}\right),\ K_2(C,r)=1200\left(C+\frac{1}{r}\right)^2.
	\end{eqnarray*}
\end{proposition}

\subsubsection{Proof of Proposition \ref{appproposition1}}
	Let $l$\ be a bump function in time (on $(-\infty,0]$) that equals to $1$\ on $[c,0]$\ and has compact support in $(c-r,0]$, with $||l'||_\infty\leq \frac{2}{r}$. It can be easily extended into a function in $C_c^\infty(\mathbb{R})$, in the following we only use its part on $(-\infty,0]$.  By Assumption \ref{appassumption}, for each $i$\ there exists a nice cut-off function $\varphi_i$\ for the pair $W_i\subset W_{i+1}$\ such that for any $v\in \mathcal{F}$,
	\begin{eqnarray}
	\label{cut-offi}
	\int_Xv^2\,d\Gamma(\varphi_i,\varphi_i)\leq C_1 \int_X \varphi_i^2\,d\Gamma(v,v)+C\int_{\scaleto{\mbox{supp}\{\varphi_i\}}{6pt}}v^2\,dm.
	\end{eqnarray}
	The product $\varphi_i(x)l(t)$\ is a nice product cut-off function for the pair $(c,0]\times W_i\Subset (c-r,0]\times W_{i+1}$. First,
	\begin{eqnarray*}
		\lefteqn{\int_{J_{-r}} \int_{W_{i+1}} 2u\partial_tu \varphi_i^2l^2\,dmdt+\int_{J_{-r}}\int_{W_{i+1}} u^2\varphi_i^2\left(l^2\right)'\,dmdt}\notag\\
		&=&\int_{J_{-r}}\partial_t\left(\int_{W_{i+1}} u^2\varphi_i^2l^2\,dm\right)\,dt
		=\left(\int_{W_{i+1}} u^2\varphi_i^2l^2\,dm\right)\bigg|_{t=0}\geq 0.\label{apppropi}
	\end{eqnarray*}
	On the other hand, since $u$\ is an ancient local weak solution of the heat equation $(\partial_t+P)u=0$, 
	\begin{eqnarray*}
		\lefteqn{\int_{J_{-r}} \int_{W_{i+1}} 2u\partial_tu \varphi_i^2l^2\,dmdt
		= -2\int_{J_{-r}}l^2\mathcal{E}(u,\, u\varphi_i^2)\,dt}\\
		&=& -2\int_{J_{-r}}\int_{W_{i+1}} \varphi_i^2l^2\,d\Gamma(u,u)\,dt-4\int_{J_{-r}}l^2\int_{W_{i+1}} \varphi_i u\,d\Gamma(\varphi_i,u)\,dt\\
		&&-2\int_{J_{-r}}l^2\int_Xu^2\varphi_i^2\,dkdt.
	\end{eqnarray*}	
By the Cauchy-Schwartz inequality (\ref{CSineq}), 
	 \begin{eqnarray*}
	 \lefteqn{-2\int_{J_{-r}}\int_{W_{i+1}} \varphi_i^2l^2\,d\Gamma(u,u)\,dt-4\int_{J_{-r}}l^2\int_{W_{i+1}} \varphi_i u\,d\Gamma(\varphi_i,u)\,dt}\\
	 &\leq&  -\int_{J_{-r}}\int_{W_{i+1}} \varphi_i^2l^2\,d\Gamma(u,u)\,dt+4\int_{J_{-r}}\int_{W_{i+1}} u^2l^2\,d\Gamma(\varphi_i,\varphi_i)\,dt,
	\end{eqnarray*}
	 Combining the estimates so far and apply (\ref{cut-offi}), we get
	\begin{eqnarray*}
		&&\hspace{-.35in}(1-4C_1)\int_{J_{-r}}\int_{W_{i+1}} l^2\varphi_i^2\,d\Gamma(u,u)\,dt+2\int_{J_{-r}}l^2\int_Xu^2\varphi_i^2\,dkdt\\
		&\leq& (4C+\frac{4}{r})\int_{J_{-r}}\int_{W_{i+1}}u^2\,dmdt,
	\end{eqnarray*}
	which implies
	\begin{eqnarray}
	\lefteqn{\int_{J}\left(\int_{W_i} d\Gamma(u,u)+\int_{W_i}u^2\,dk\right)dt}\notag\\
	&\leq& (1-4C_1)^{-1}\left(4C+\frac{4}{r}\right)\int_{J_{-r}}\int_{W_{i+1}}u^2\,dmdt.\label{appest1}
	\end{eqnarray}
	Replace $\varphi_i$\ by $\varphi_i^2$\ in the above estimates, we obtain
	\begin{eqnarray*}
		\lefteqn{-\int_{J_{-r}}\int_{W_{i+1}} u^2\varphi_i^4\left(l^2\right)'\,dmdt}\\
		&\leq& \int_{J_{-r}} \int_{W_{i+1}} 2u\partial_tu \varphi_i^4l^2\,dmdt
		= -2\int_{J_{-r}}l^2\mathcal{E}(u,\, u\varphi_i^4)\,dt\\
		&=& -2\int_{J_{-r}}\int_{W_{i+1}} \varphi_i^4l^2\,d\Gamma(u,u)\,dt-8\int_{J_{-r}}l^2\int_{W_{i+1}} \varphi_i^3 u\,d\Gamma(\varphi_i,u)\,dt\\
		&&-2\int_{J_{-r}}l^2\int_Xu^2\varphi_i^4\,dkdt;
		\end{eqnarray*}
		where
		\begin{eqnarray*}
		\lefteqn{-2\int_{J_{-r}}\int_{W_{i+1}} \varphi_i^4l^2\,d\Gamma(u,u)\,dt-8\int_{J_{-r}}l^2\int_{W_{i+1}} \varphi_i^3u\,d\Gamma(\varphi_i,u)\,dt}\\
		&\leq&  -\int_{J_{-r}}\int_{W_{i+1}} \varphi_{i}^4l^2\,d\Gamma(u,u)\,dt+16\int_{J_{-r}}\int_{W_{i+1}} \varphi_i^2u^2l^2\,d\Gamma(\varphi_i,\varphi_i)\,dt.
		\end{eqnarray*}
		By a lemma below (Lemma \ref{applemma}), the term
		\begin{eqnarray*}
		\int_{J_{-r}}\int_{W_{i+1}} \varphi_i^2u^2l^2\,d\Gamma(\varphi_i,\varphi_i)\,dt\leq C_0\int_{J_{-r}}\int_{W_{i+1}}\varphi_i^2u^2l\,dmdt.
		\end{eqnarray*}
		Hence
		\begin{eqnarray}
		\label{new}
		\lefteqn{\int_{J_{-r}}\int_{W_{i+1}} \varphi_{i}^4l^2\,d\Gamma(u,u)\,dt+\int_{J_{-r}}l^2\int_Xu^2\varphi_i^4\,dkdt}\notag\\
		&\leq& 4\left(4C_0+\frac{1}{r}\right)\int_{J_{-r}}l\int_{W_{i+1}}\varphi_i^2u^2\,dmdt.
		\end{eqnarray}
		Here both sides of the inequality are integrals over the same set $J_{-r}\times W_{i+1}$.
		
	Next we estimate the $L^2$\ norm of $\partial_tu$, which by Corollary \ref{L2corollary} is also a local weak solution on $(-\infty,b)\times X$. Thus
	\begin{eqnarray}
	\lefteqn{\int_{J_{-r}}\int_{W_{i+1}}\left(\partial_tu\,\varphi_i l\right)^2\,dmdt=-\int_{J_{-r}}l^2\,\mathcal{E}(u,\,\partial_tu\,\varphi_i^2)\,dt}\notag\\
	&=&-\int_{J_{-r}}l^2\int_{W_{i+1}}\varphi_i^2\,d\Gamma(u,\partial_tu)\,dt-\int_{J_{-r}}l^2\int_{W_{i+1}}2\varphi_i \partial_tu\,d\Gamma(u,\varphi_i)\,dt\notag\\
	&&-\int_{J_{-r}}l^2\int_Xu\partial_tu\varphi_i^2\,dkdt.\label{appstep2}
    \end{eqnarray}
	We show that
	\begin{eqnarray*}
		\int_{W_{i+1}}\varphi_i^2\,d\Gamma(\partial_tu,u)=\frac{1}{2}\partial_t\left(\int_{W_{i+1}}\varphi_i^2\,d\Gamma(u,u)\right)
	\end{eqnarray*}
	to replace the first term in (\ref{appstep2}). This follows from the estimate
	\begin{eqnarray*}
		\left|\left(\int_Xf\,d\Gamma(v,v)\right)^{1/2}-\left(\int_Xf\,d\Gamma(w,w)\right)^{1/2}\right|\leq \left(\int_Xf\,d\Gamma(v-w,v-w)\right)^{1/2}
	\end{eqnarray*}
	where $f$\ is any nonnegative, bounded, Borel function, and $v,w\in \mathcal{F}$\ (cf. Chapter 3 in \cite{Fukushima}). Together with the Cauchy-Schwartz inequality (\ref{CSineq}), it follows that if $v_n\rightarrow v$\ in $\mathcal{E}_1$\ norm (or just in the $\mathcal{E}$-energy), then
	\begin{eqnarray*}
		\lim_{n\rightarrow \infty}\int_Xf\,d\Gamma(v_n,w)= \int_Xf\,d\Gamma(v,w).
	\end{eqnarray*}
	Here $w\in \mathcal{F}$, and $f$\ is as above. For any fixed $t$, by taking $v_n(x):=\frac{u(t+1/n,x)-u(t,x)}{1/n}$, $w(x):=u(t,x)$, and $f:=\varphi_i^2$, we conclude that
	\begin{eqnarray*}
		\frac{1}{2}\partial_t\left(\int_{W_{i+1}}\varphi_i^2\,d\Gamma(u,u)\right)=\int_{W_{i+1}}\varphi_i^2\,d\Gamma(\partial_tu,u).
	\end{eqnarray*}
	Thus the first term in (\ref{appstep2}) equals
	\begin{eqnarray}
	\lefteqn{-\int_{J_{-r}}l^2\int_{W_{i+1}}\varphi_i^2\,d\Gamma(u,\partial_tu)\,dt
		= -\frac{1}{2}\int_{J_{-r}}l^2\,\partial_t\left(\int_{W_{i+1}}\varphi_i^2\,d\Gamma(u,u)\right)\,dt}\notag\\
	&=& -\frac{1}{2}\int_{J_{-r}}\partial_t\left(l^2\int_{W_{i+1}}\varphi_i^2\,d\Gamma(u,u)\right)\,dt+\frac{1}{2}\int_{J_{-r}}\left(l^2\right)'\int_{W_{i+1}}\varphi_i^2\,d\Gamma(u,u)\,dt\notag\\
	&\leq&\frac{1}{2}\int_{J_{-r}}\left(l^2\right)'\int_{W_{i+1}}\varphi_i^2\,d\Gamma(u,u)\,dt.\label{appstep2.1}
	\end{eqnarray}
The second term in (\ref{appstep2}) satisfies
	\begin{eqnarray*}
		\lefteqn{\left|-\int_{J_{-r}}l^2\int_{W_{i+1}}2\varphi_i \partial_tu\,d\Gamma(u,\varphi_i)\,dt\right|}\\
&=&\left|-\int_{J_{-r}}l^2\int_{W_{i+1}}2\varphi_i \varphi_{i+1}\partial_tu\,d\Gamma(u,\varphi_i)\,dt\right|\\
		&\leq& \epsilon\int_{J_{-r}}\int_{W_{i+1}}l^4\varphi_i^2\left(\partial_tu\right)^2\,d\Gamma(\varphi_i,\varphi_i)\,dt+\frac{1}{\epsilon}\int_{J_{-r}}\int_{X}\varphi_{i+1}^2\,d\Gamma(u,u)\,dt
	\end{eqnarray*}
for any $\epsilon>0$. Here $\varphi_{i+1}$\ is the nice cut-off function for the pair $W_{i+1}\subset W_{i+2}$, in particular, $\varphi_{i+1}\equiv 1$\ on $\mbox{supp}\{\varphi_i\}$. By Lemma \ref{applemma},
	\begin{eqnarray*}
		\int_{J_{-r}}\int_{W_{i+1}}l^4\varphi_i^2\left(\partial_tu\right)^2\,d\Gamma(\varphi_i,\varphi_i)dt\,\leq C_0\int_{J_{-r}}\int_{W_{i+1}} l^2\varphi_i^2\left(\partial_tu\right)^2\,dmdt,
	\end{eqnarray*}
	where $C_0\leq 3C+\frac{1}{r}$. Thus
	\begin{eqnarray}
	\lefteqn{\left|-\int_{J_{-r}}l^2\int_{W_{i+1}}2\varphi_i \partial_tu\,d\Gamma(u,\varphi_i)\,dt\right|}\notag\\
&\hspace{-.3in}\leq&\hspace{-.15in}\epsilon C_0\int_{J_{-r}}\int_{W_{i+1}} \left(\partial_tu\varphi_il\right)^2\,dmdt+\frac{1}{\epsilon}\int_{J_{-r}}\int_{X}\varphi_{i+1}^2\,d\Gamma(u,u)\,dt.\label{appstep2.2}
	\end{eqnarray}
	The last term in (\ref{appstep2}) satisfies by Cauchy-Schwartz inequality
	\begin{eqnarray}
	\lefteqn{-\int_{J_{-r}}l^2\int_Xu\partial_tu\varphi_i^2\,dkdt=-\int_{J_{-r}}\int_X\varphi_{i+1} l^2u\partial_tu\varphi_i^2\,dkdt}\notag\\
	&\leq& \frac{1}{2c}\int_{J_{-r}}\int_{X}\varphi_{i+1}^2u^2\,dkdt+\frac{c}{2}\int_{J_{-r}}\int_Xl^4(\partial_tu)^2\varphi_{i}^4\,dkdt\label{kpart}
	\end{eqnarray}
	for any $c>0$. 
Now we plug in $C_1=\frac{1}{16}$\ and take $\epsilon=\frac{1}{2C_0}$, then by (\ref{new}), (\ref{appstep2}), (\ref{appstep2.1}), (\ref{appstep2.2}), and (\ref{kpart}),
	\begin{eqnarray*}
		\lefteqn{\int_{J_{-r}}\int_{W_{i+1}}\left(\partial_tu\,\varphi_i l\right)^2\,dmdt}\notag\\
		&\leq& \frac{1}{2}\int_{J_{-r}}\left(l^2\right)'\int_{W_{i+1}}\varphi_i^2\,d\Gamma(u,u)\,dt+\frac{1}{2}\int_{J_{-r}}\int_{W_{i+1}}\left(\partial_tu\varphi_il\right)^2\,dmdt\\
		&&+2C_0\int_{J_{-r}}\int_{X}\varphi_{i+1}^2\,d\Gamma(u,u)\,dt+\frac{1}{2c}\int_{J_{-r}}\int_{X}\varphi_{i+1}^2u^2\,dkdt\\
		&&+c\left(8C_0+\frac{2}{r}\right)\int_{J_{-r}}\int_Xl^2(\partial_tu)^2\varphi_{i}^2\,dmdt.
	\end{eqnarray*}
	Let $c=\frac{1}{4}\left(8C_0+\frac{2}{r}\right)^{-1}$, then
	\begin{eqnarray}
		\lefteqn{\int_{J_{-r}}\int_{W_{i+1}}\left(\partial_tu\,\varphi_i l\right)^2\,dmdt}\notag\\
		&\leq& 2\int_{J_{-r}}\left(l^2\right)'\int_{W_{i+1}}\varphi_i^2\,d\Gamma(u,u)\,dt+8C_0\int_{J_{-r}}l^2\int_{X}\varphi_{i+1}^2\,d\Gamma(u,u)\,dt\notag\\
		&&8\left(8C_0+\frac{2}{r}\right)\int_{J_{-r}}\int_Xu^2\varphi_{i+1}^2\,dkdt\notag\\
		&\leq&8\left(8C_0+\frac{2}{r}\right)\int_{J_{-r}}\left(\int_{W_{i+2}}d\Gamma(u,u)+\int_{W_{i+2}}u^2\,dk\right)dt.\label{appest2}
	\end{eqnarray}
	Take $J_{-r}$\ and $W_{i+2}$\ on the left in (\ref{appest1}) with $C_1=\frac{1}{16}$, combine (\ref{appest1}) and (\ref{appest2}), and recall that $C_0\leq 2C+\frac{1}{r}$, we obtain that
\begin{eqnarray*}
\lefteqn{\int_{J}\int_{W_{i}}(\partial_tu)^2\,dmdt}\\
&\leq&200\left(C+\frac{1}{r}\right)\int_{J_{-r}}\left(\int_{W_{i+2}}d\Gamma(u,u)+\int_{W_{i+2}}u^2\,dk\right)dt\\
&\leq& 1200\left(C+\frac{1}{r}\right)^2\int_{J_{-2r}}\int_{W_{i+3}}u^2\,dmdt.
\end{eqnarray*}
Let $K_1(C,r):=200\left(C+\frac{1}{r}\right)$\ and $K_2(C,r):=1200\left(C+\frac{1}{r}\right)^2$. This completes the proof for Proposition \ref{appproposition1}. Note that by taking $C$\ small and $r$\ large enough, we can make the coefficients $K_1(C,r)$\ and $K_2(C,r)$\ as small as needed.

Straightforward iterations lead to Proposition \ref{appcor}.
\subsubsection{A technical lemma} Last we state and prove the technical lemma used in the proof of Proposition \ref{appproposition1}.
\begin{lemma}
\label{applemma}
Let $(X,m)$\ be a metric measure space and $(\mathcal{E},\mathcal{F})$\ be a symmetric regular local Dirichlet form on $X$. Assume that the Dirichlet space $(X,\mathcal{E},\mathcal{F})$\ satisfies Assumption \ref{assumptionbothcases}, and when $X$\ is not compact, further satisfies Assumption \ref{appassumption}. Let $(H_t)_{t>0}$\ and $-P$\ be the corresponding semigroup and generator. Let $I,I'$\ be two intervals where $I=(a,b)$\ or $(a,b]$, $I\Subset I'$. Let $u$\ be a local weak solution of the heat equation $(\partial_t+P)u=0$\ on $I'\times X$. Let $\overline{\varphi}(t,x):=\varphi(x)l(t)$\ be a nice product cut-off function corresponding to coefficients $C_1,C_2$\ where $C_1<\frac{1}{9}$\ and $l$\ is a bump function with support in $I$. Then there exists some $C_0$\ such that
	\begin{eqnarray*}
	\int_I\int_X\overline{\varphi}^2u^2\,d\Gamma(\overline{\varphi},\overline{\varphi})\,dt\leq C_0\int_I\int_X\overline{\varphi}^2u^2\,dmdt.
	\end{eqnarray*}
\end{lemma}
The last inequality says when there is the same cut-off function with bounded energy in both the integrand and in the energy measure, the net effect is the same as having a cut-off function with bounded gradient in the energy measure. This is easy to check for the special case $\int_X\varphi^2\,d\Gamma(\varphi,\varphi)$. Here we generalize this observation to $\int_X\varphi^2u^2\,d\Gamma(\varphi,\varphi)$, for local weak solutions $u$.
\begin{proof} Since $\overline{\varphi}$\ is a nice product cut-off function associated with $C_1,C_2$,
	\begin{eqnarray*}
	\int_I\int_X\overline{\varphi}^2u^2\,d\Gamma(\overline{\varphi},\overline{\varphi})\,dt\leq C_1\int_I\int_X\overline{\varphi}^2\,d\Gamma(\overline{\varphi} u,\overline{\varphi} u)\,dt+C_2\int_I\int_X\overline{\varphi}^2u^2\,dmdt.
	\end{eqnarray*}
To estimate $\int_I\int_X\overline{\varphi}^2\,d\Gamma(\overline{\varphi} u,\overline{\varphi} u)\,dt$, we make the following two observations
\begin{itemize}
	\item[(i)] $\int_I\int_X\overline{\varphi}^2\,d\Gamma(u,\overline{\varphi}^2u)\,dt=\int_I\int_X\overline{\varphi}^2\,d\Gamma(\overline{\varphi} u,\overline{\varphi} u)\,dt-\int_I\int_X\overline{\varphi}^2u^2\,d\Gamma(\overline{\varphi},\overline{\varphi})\,dt$;
\item[(ii)]$\int_I\int_X\overline{\varphi}^2\,d\Gamma(u,\overline{\varphi}^2u)\,dt\\[0.05in]
=\int_I\int_Xd\Gamma(u,\overline{\varphi}^4u)\,dt-2\int_I\int_X\overline{\varphi}^2u\,d\Gamma(\overline{\varphi} u,\overline{\varphi})+2\int_I\int_X\overline{\varphi}^2u^2\,d\Gamma(\overline{\varphi},\overline{\varphi})\,dt$.
\end{itemize}
The middle term in (ii) can be estimated by
\begin{eqnarray*}
\left|2\int_I\int_X\overline{\varphi}^2u\,d\Gamma(\overline{\varphi} u,\overline{\varphi})\right|\leq \epsilon\int_I\int_X\overline{\varphi}^2\,d\Gamma(\overline{\varphi} u,\overline{\varphi} u)\,dt+\frac{1}{\epsilon}\int_I\int_X\overline{\varphi}^2u^2\,d\Gamma(\overline{\varphi},\overline{\varphi})\,dt.
\end{eqnarray*}
To estimate the first term in (ii), note that $u$\ being a local weak solution implies that
\begin{eqnarray*}
\lefteqn{\int_I\int_Xd\Gamma(u,\overline{\varphi}^4u)\,dt=-\int_I\int_X\partial_tu\cdot \overline{\varphi}^4u\,dmdt-\int_I\int_X\overline{\varphi}^4u^2\,dkdt}\\
&\leq&-\frac{1}{2}\int_I\int_X(\partial_t(l^4u^2\varphi^4)-\partial_t(l^4)u^2\varphi^4)\,dmdt\leq 2||l'||_\infty\int_I\int_Xl^3u^2\varphi^4\,dmdt.
\end{eqnarray*}
Combining (i)(ii) and the estimates above, we get that
\begin{eqnarray*}
\lefteqn{\int_I\int_X\overline{\varphi}^2\,d\Gamma(\overline{\varphi} u,\overline{\varphi} u)\,dt\leq 2||l'||_\infty\int_I\int_Xl^3\varphi^4u^2\,dmdt}\\
&&+\epsilon\int_I\int_X\overline{\varphi}^2\,d\Gamma(\overline{\varphi} u,\overline{\varphi} u)\,dt+\left(\frac{1}{\epsilon}+3\right)\int_I\int_X\overline{\varphi}^2u^2\,d\Gamma(\overline{\varphi},\overline{\varphi})\,dt\\
&\leq& 2||l'||_\infty\int_I\int_Xl^3\varphi^4u^2\,dmdt+\left[\epsilon+C_1\left(\frac{1}{\epsilon}+3\right)\right]\int_I\int_X\overline{\varphi}^2\,d\Gamma(\overline{\varphi} u,\overline{\varphi} u)\,dt\\
&&+\left(\frac{1}{\epsilon}+3\right)C_2\int_I\int_X\overline{\varphi}^2u^2\,dmdt.
\end{eqnarray*}
When $C_1<\frac{1}{9}$, we can pick $\epsilon$\ small so that $\epsilon+C_1\left(\frac{1}{\epsilon}+3\right)<1$. Let $\alpha=1-\left[\epsilon+C_1\left(\frac{1}{\epsilon}+3\right)\right]>0$, the above estimate is equivalent to
\begin{eqnarray*}
\lefteqn{\int_I\int_X\overline{\varphi}^2\,d\Gamma(\overline{\varphi} u,\overline{\varphi} u)\,dt}\\
&\leq& \frac{1}{\alpha}\left\{2||l'||_\infty\int_I\int_Xl^3\varphi^4u^2\,dmdt+\left(\frac{1}{\epsilon}+3\right)C_2\int_I\int_X\overline{\varphi}^2u^2\,dmdt\right\}\\
&\leq& K\int_I\int_X\overline{\varphi}^2u^2\,dmdt,
\end{eqnarray*}
where $K=\frac{1}{\alpha}\left[2||l'||_\infty+\left(\frac{1}{\epsilon}+3\right)C_2\right]$. Combining this with the very first inequality, we get
\begin{eqnarray*}
\int_I\int_X\overline{\varphi}^2u^2\,d\Gamma(\overline{\varphi},\overline{\varphi})\,dt\leq C_1K\int_I\int_X\overline{\varphi}^2u^2\,dmdt+C_2\int_I\int_X\overline{\varphi}^2u^2\,dmdt.
\end{eqnarray*}
Letting $C_0=C_1K+C_2$\ gives the inequality in the lemma. To apply this lemma to the proof of Proposition \ref{appproposition1}, let $C_1=\frac{1}{16}$\ and $C_2=C$, plug in $||l'||_\infty\leq \frac{2}{r}$, and take for example $\alpha=\frac{5}{16}$, we get that $\epsilon=\frac{1}{4}$, $C_0\leq 3C+\frac{1}{r}$.
\end{proof}

\section{Examples} \setcounter{equation}{0}
In this section we list some examples to which our theorems apply. We group them according to the types of nice cut-off functions they admit. Note that the properties we require on the nice cut-off functions involve only the energy measure associated with the Dirichlet form, so in the following we describe examples of strongly local Dirichlet forms; our theorems apply to local Dirichlet forms whose strongly local parts belong to the following examples as well.

\subsection{Dirichlet spaces with good intrinsic distance} In \cite{Sturm2}, Sturm showed that in a symmetric strongly local regular Dirichlet space, when the topology induced by the intrinsic distance (\ref{intrinsicdist}), that is,
\begin{eqnarray*}
\rho_X(x,y)=\sup\left\{\varphi(x)-\varphi(y)\,|\,\varphi\in \mathcal{F}_{\scaleto{\mbox{loc}}{5pt}}(X)\cap C(X),\ d\Gamma(\varphi,\varphi)\leq dm\right\},
\end{eqnarray*}
is equivalent to the original topology on $X$, one can use the intrinsic distance to construct nice cut-off functions with bounded gradient. More precisely, for $V\Subset U\Subset X$, define 
\begin{eqnarray*}
\eta(x):=\frac{\left(\frac{1}{\sqrt{2}}\rho_X(V,U^c)-\rho_X(x,V)\right)_+}{\frac{1}{\sqrt{2}}\rho_X(V,U^c)}.
\end{eqnarray*}
Clearly $\eta=1$\ on $V$ and $\mbox{supp}\{\eta\}\subset U$. Further, $\eta$\ is in $\mathcal{F}_{\scaleto{\mbox{loc}}{5pt}}(X)\cap C(X)$, and
\begin{eqnarray}
\label{gradientbound}
d\Gamma(\eta,\eta)\leq \frac{2}{d(V,U^c)^2}\,dm.
\end{eqnarray} 
These results are from Lemma 1.9 in \cite{Sturm2}. It clearly follows that such Dirichlet spaces satisfy Assumptions \ref{assumptionbothcases} and \ref{appassumption} (pick the exhaustion $\{W_{V,i}^n\}_{i=1}^\infty$\ to be given by balls with radii $r_i$\ that increase fast enough). By Lemma \ref{L2Gaussian1}, these Dirichlet spaces satisfy the $L^2$\ Gaussian type upper bound. Thus all results in this paper apply to this type of examples which include:

\begin{itemize}
\item[(1)] Weighted Riemannian manifold with Dirichlet form associated with any uniformly elliptic operator with bounded measurable coefficients. See,  e.g., \cite{unifelliptic}.  This includes the example we described in the Introduction, 
and we remark that all results in this paper hold when the operator is only locally uniformly elliptic. 
\item[(2)] Riemannian polyhedra under minimal local assumptions (cf. \cite{cplxbook,Riemcplx} and \cite{stripcplx}).
\item[(3)] 
Alexandrov spaces and their Dirichlet space structures as considered for instance  in \cite{Alexandrovdef,Alexandrov}.
\end{itemize}

\subsection{Fractal type Dirichlet spaces}
For fractal spaces, Assumption \ref{assumptionbothcases} is a nontrivial hypothesis to check. It is well known that in many fractal spaces the only functions in $\mathcal{F}_{\scaleto{\mbox{loc}}{5pt}}(X)\cap C(X)$\ are constant functions (cf. e.g. \cite{Kusuoka}), so fractal spaces in general do not possess cut-off functions with bounded gradient. More generally, in a recent paper \cite{Mathavsingularity}, it was shown that for a very general class of Dirichlet spaces, two-sided off-diagonal heat kernel estimate with walk-dimension strictly larger than two implies the singularity of the energy measure with respect to the symmetric measure. 

On the other hand, many fractal spaces admit cut-off functions satisfying the inequality (\ref{cut-offbddenergy}) in Assumption \ref{assumptionbothcases}. For example, the Sierpinsket gasket and its non-compact extension as in the following pictures both satisfy Assumption \ref{assumptionbothcases}. (The picture on left ($\mathcal{SG}$) is from Wikipedia, the picture on right ($\mathcal{ISG}$) is created by shifting copies of $\mathcal{SG}$.)

\begin{figure}[h]
\begin{center}
	\begin{subfigure}[h]{0.2\textwidth}
		\includegraphics[width=\textwidth]{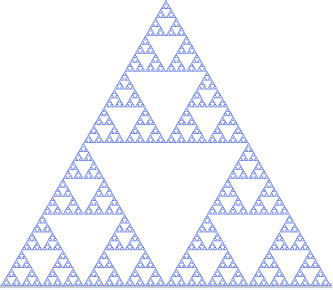}
		\caption{the Sierpinski gasket $\mathcal{SG}$}
	\end{subfigure}
	\hspace{.3in}
	\begin{subfigure}[h]{0.5\textwidth}
		\includegraphics[width=\textwidth]{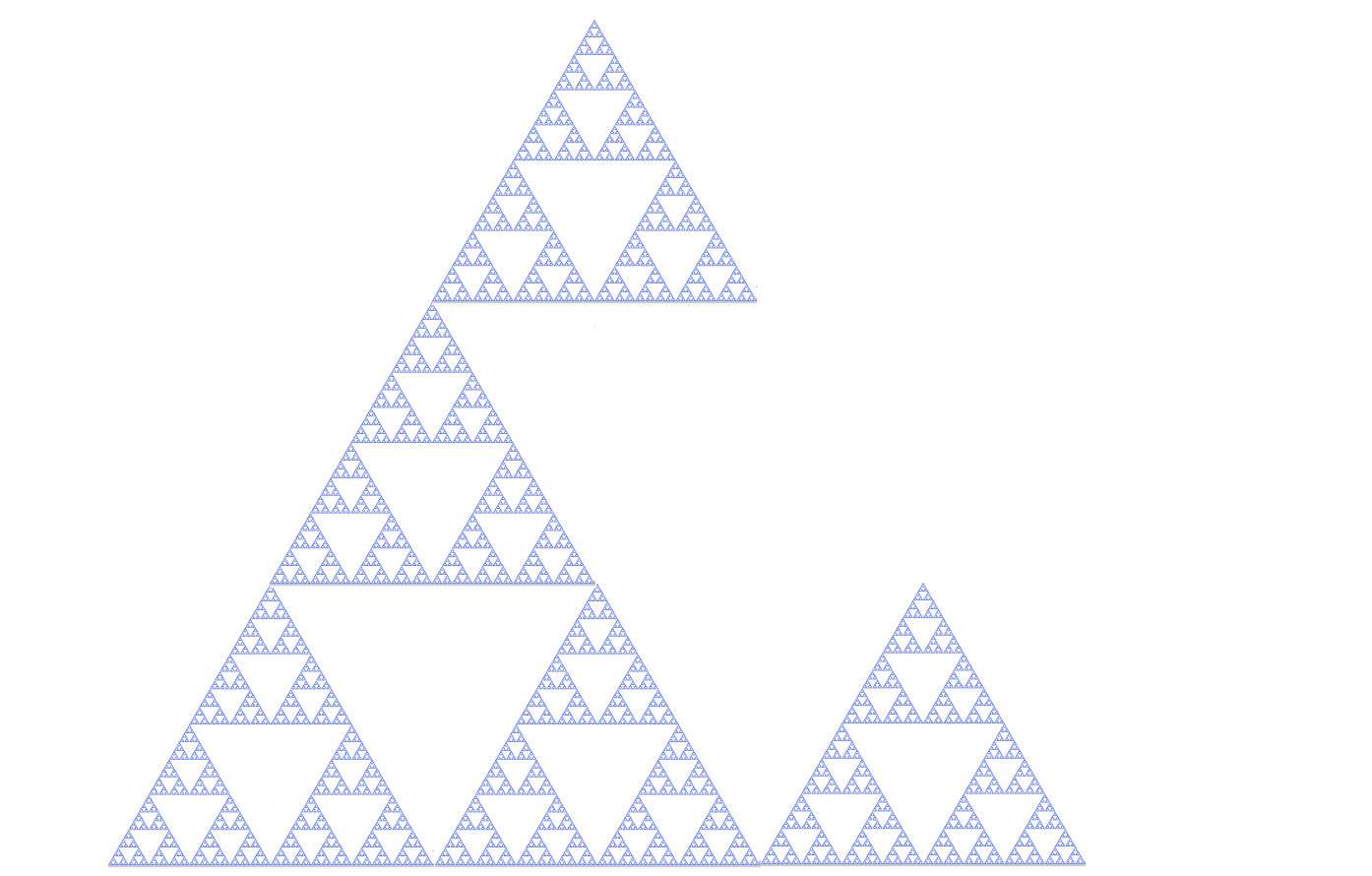}
		\caption{the infinite Sierpinski gasket $\mathcal{ISG}$}
	\end{subfigure}
\end{center}	
\end{figure}

We remark that the existence of cut-off functions satisfying (\ref{cut-offbddenergy}) on such examples is highly nontrivial, and although their existence is known (cf. \cite{AndresBarlow}), there is in general no direct geometric construction of such cut-off functions. For example, in \cite{AndresBarlow} the authors showed that fractal spaces that satisfy some version of parabolic Harnack inequality must admit cut-off functions that satisfy some more specified version of the inequality (\ref{cut-offbddenergy}). A typical more specific dependence of $C_2$\ on $C_1,U,V$\ is that $C_2\sim C_1^{-\alpha}d_X(V,U^c)^{-\beta}$\ for some distance $d_X(V,U^c)$\ between $V,U^c$\ and some constants $\alpha,\beta>0$. In \cite{EHI}, the authors proved that this dependence is in fact a fairly general case, that for a class of functions called regular scale functions, denoted by $\Psi$, there exists some distance $d_\Psi$\ that defines the same topology (being so-called quasisymmetric with the original distance of the fractal space), such that the fractal space admits cut-off functions satisfying the so-called cut-off energy inequality $CS(\Psi)$\ (a special form of (\ref{cut-offbddenergy}) where $C_2$\ is expressed in terms of the $\Psi$\ function, more general than powers), if and only if the fractal space admits cut-off functions satisfying (\ref{cut-offbddenergy}) with $C_2\sim C_1^{-\alpha}d_\Psi(V,U^c)^{-\beta}$, which are power functions with respect to the distance $d_\Psi$. For all these fractal spaces, Lemma \ref{L2Gaussian1} guarantees they satisfy the $L^2$\ Gaussian type upper bound. Because of the distance $d_X$, similar to the first type of examples, we can check that Assumption \ref{appassumption} is satisfied. Hence all results in this paper apply.
\subsection{Infinite products of Dirichlet spaces of the first two types}
The first examples we have in mind for this type of examples are the infinite dimensional torus $\mathbb{T}^\infty$\ and the infinite product of Sierpinski gaskets $\mathcal{SG}^\infty$, the first one being a special case of the class of locally compact connected metrizable (infinite dimensional) groups, cf. \cite{IwaegLaurent}, and the second one the simplest of the infinite product of compact fractal spaces. A general treatment of the elliptic diffusion on (compact) infinite product spaces like $\mathbb{T}^\infty$\ is \cite{Elliptic}, and their results apply more generally to anomalous diffusion on infinite products of (compact) fractal spaces too. To have some noncompact examples we can consider the Iwasawa's example (cf. \cite{Iwaeg,IwaegLaurent}), or replace one piece of Sierpinski gasket in the product $\mathcal{SG}^\infty$\ by the infinite Sierpinski gasket $\mathcal{ISG}$. 

On a locally compact connected metrizable group $G$\ that is unimodular, one usually starts with a heat (convolution) semigroup, or a (left-invariant) Laplacian of the form $L=-\sum a_{ij}X_iX_j$, where $(a_{ij})_{i,j=1}^\infty$\ is symmetric positive definite and $\{X_i\}_{i=1}^\infty$\ is a projective basis of the left-invariant vector fields on $G$\ (in the projective Lie algebra of $G$), and then consider the associated (left-invariant) Dirichlet form. Depending on the coefficients, the Dirichlet form may or may not have nondegenerate intrinsic distance.

For general product spaces $G$\ that have rougher differential structures, like $\mathcal{SG}^\infty$, it is easier and more convenient to consider only the ``diagonal Dirichlet form'', namely, for any diagonal matrix $(a_{ii})_{i=1}^\infty$\ with all $a_{ii}>0$, consider
\begin{eqnarray}
\label{infiniteE}
	\mathcal{E}(f,g)=\sum_{i=1}^\infty a_{ii}\int \mathcal{E}_i(f,g)\, d\left(\tens{\scaleto{j\neq i}{5pt}}m_j\right).
\end{eqnarray}
Here $\mathcal{E}_i$\ stands for the standard Dirichlet form on the $i$th factor of $G$, $m_j$\ stands for the normalized Hausdorff measure on the $j$th factor of $G$, and $f,g$\ are proper functions. 

This third type of examples does not satisfy a property often satisfied in the previous two types of examples, namely, for these infinite dimensional spaces, the volume doubling property (local or global) cannot hold. In some cases these infinite product examples do possess nondegenerate intrinsic distances that define the same topology (e.g. when the coefficient matrix for the Laplacian on $\mathbb{T}^\infty$\ is diagonal and satisfies $\sum a_{ii}^{-2}<\infty$), in which case
Assumption \ref{assumptionbothcases} and Assumption \ref{appassumption} follow.
But, more generally, one can show that the cut-off function assumptions 
(Assumption \ref{assumptionbothcases} and Assumption \ref{appassumption}) are satisfied using the fact that each factor in the infinite product possesses nice cut-off functions in the senses required.

More precisely, since the product topology is generated by cylindric sets (sets that are direct products of open sets as the first few factors, and the whole space for all remaining factors), for pairs of cylindric sets it is easy to construct a nice cut-off function being a product of nice cut-off functions for pairs of open sets on the first few factors, namely,
\begin{eqnarray}
\label{prodcut-off}
	\varphi({\bf x}):=\prod_{i=1}^{N}\varphi_i(x_i)
\end{eqnarray}
for some $N$. We verify this for the simpler case when the Dirichlet form is defined as in (\ref{infiniteE}) (for the group case this is when the coefficient matrix is diagonal, and there is no drift part in the Laplacian).

Suppose $\varphi_i(x_i)$\ is a nice cut-off function on the $i$-th factor $X_i$\ of the infinite product space $X=\prod_{i}X_i$, satisfying that for any $v\in \mathcal{D}(\mathcal{E}_i)$,
\begin{eqnarray}
\int v^2\, d\Gamma_i(\varphi_i,\varphi_i)\leq C_1\int \varphi_i^2\,d\Gamma(v,v)+C_2\int_{\scaleto{\mbox{supp}\{\varphi_i\}}{6pt}}v^2\,dm_i.
\end{eqnarray}
Here $\Gamma_i$\ represents the energy measure on $X_i$, and $C_1,C_2$\ are the same for all factors $X_i$. Then for any $f\in \mathcal{D}(\mathcal{E})$, for the function $\varphi$\ defined as in (\ref{prodcut-off}),
\begin{eqnarray*}
\lefteqn{\int_X f^2\,d\Gamma(\varphi,\varphi)}\\
&=&\sum_{i=1}^\infty a_{ii}\int_{\prod_{j\neq i}X_j}\left(\int_{X_i}f^2\,d\Gamma_i(\varphi_i,\varphi_i)\right)\left(\prod_{\substack{\scaleto{j=1}{4pt}\\\scaleto{j\neq i}{4pt}}}^{\scaleto{N}{4pt}}\varphi_j(x_j)\right)^2\,d\left(\tens{\scaleto{j\neq i}{5pt}}m_j\right)\\
&\leq& \sum_{i=1}^\infty a_{ii}\left[C_1\int_{\prod_{j\neq i}X_j}\left(\int_{X_i}(\varphi_i)^2\,d\Gamma_i(f,f)\right)\left(\prod_{\substack{\scaleto{j=1}{4pt}\\\scaleto{j\neq i}{4pt}}}^{\scaleto{N}{4pt}}\varphi_j(x_j)\right)^2\,d\left(\tens{\scaleto{j\neq i}{5pt}}m_j\right)\right.\\
&&\left.\hfill+C_2\int_{\scaleto{\mbox{supp}\{\varphi\}}{6pt}}f^2\,dm\right].
\end{eqnarray*}
In the last line we bounded the product of $\varphi_i$'s by $1$. Then note that since
\begin{eqnarray*}
\int_X\varphi^2\,d\Gamma(f,f)=\sum_{i=1}^\infty a_{ii}\int_{\prod_{j\neq i}X_j}\int_{X_i}\varphi^2\,d\Gamma(f,f)\,d\left(\tens{\scaleto{j\neq i}{5pt}}m_j\right),
\end{eqnarray*} 
we conclude that
\begin{eqnarray*}
\int_X f^2\,d\Gamma(\varphi,\varphi)\leq C_1\int_X\varphi^2\,d\Gamma(f,f)+C_2\int_{\scaleto{\mbox{supp}\{\varphi\}}{6pt}}f^2\,dm.
\end{eqnarray*}
Thus these infinite product spaces satisfy Assumption \ref{assumptionbothcases}. Using the topological basis of cylindric open sets, we can also easily check that these infinite product spaces satisfy Assumption \ref{appassumption}. By Lemma \ref{L2Gaussian1}, these spaces satisfy the $L^2$\ Gaussian type upper bound. We remark that here we do not have additional requirements on the coefficient matrix $(a_{ii})$\ except that all $a_{ii}>0$, or for the general infinite dimensional group case, that the coefficient matrix is positive definite.
\begin{remark} On infinite dimensional compact groups, when the Laplacian $L$\ is bi-invariant, one can define more function spaces associated with $L$\ that capture the smoothness of functions and define corresponding distributional solutions of the heat equation $(\partial_t+L)u=0$. These are broader classes of solutions than the local weak solutions we consider in this paper. In the new settings one can consider the time regularity and other spatial regularity properties of the distributional solutions of the heat equation, under more assumptions on the associated heat (convolution) semigroup, cf. \cite{distributionspaces,hypoellipticity}. In a sequel paper we will show that for these bi-invariant Laplacians $L$\ and some other differential operators that have comparable Dirichlet forms, the distributional solutions are smooth, with repeated time and spatial derivatives that still belong to the function spaces associated with $L$. These results provide generalizations of the results in \cite{hypoellipticity} and describe hypoellipticiy type properties of $\partial_t+L$.
\end{remark}
\section{The weak Gaussian bound and other lemmas} \setcounter{equation}{0}
\subsection{The weak Gaussian bound}
In this subsection we prove an $L^2$\ Gaussian type upper bound assuming the existence of cut-off functions satisfying (\ref{cut-offbddenergy}) with $C_2(C_1,U,V)=C_1^{-\alpha}C(U,V)$\ for some $\alpha>0$. Our proof is a modification of the classical proof for $L^2$\ Gaussian bound when there are cut-off functions with bounded gradient. For references that discuss about stronger (sub)-Gaussian estimates under stronger assumptions, we mention \cite{Davies,anomalous}. The last part in this subsection about transitioning to estimates on derivatives of the heat semigroup is a straightforward modification of the methods in \cite{Coulhon}.

The following is the main lemma for $L^2$\ Gaussian type upper bound. Its proof is close to for example the beginning part of the proof in \cite{anomalous}.
\begin{lemma}
	\label{L2Gaussian1}
	Let $(X,m,\mathcal{E},\mathcal{F})$\ be a symmetric local regular Dirichlet space. Assume the Dirichlet space satisfies Assumption \ref{assumptionbothcases} and that for any precompact open sets $U,V$\ with disjoint closures, $C_2$\ in (\ref{cut-offbddenergy1}) is of the form $C_2=C_1^{-\alpha}C(U,V)$\ for some $\alpha>0$\ and $C(U,V)>0$. Then for any such open sets $U,V$, for any $f,g\in L^2(X)$\ with $\mbox{supp}\{f\}\subset U$, $\mbox{supp}\{g\}\subset V$,
	\begin{eqnarray}
	\left|\left<H_tf,g\right>\right|\leq \exp{\left\{-\frac{1}{2}\left(\frac{1}{2C(U,V)t}\right)^{\frac{1}{1+2\alpha}}\right\}}\,\left|\left|f\right|\right|_{L^2}\left|\left|g\right|\right|_{L^2}.
	\end{eqnarray}
	Here $<,>$\ represents the $L^2$\ inner product on $X$.
\end{lemma}
When there exist enough nice cut-off functions with bounded gradient, Lemma \ref{L2Gaussian1} is a classical result obtained from the so-called Davies' Method. We adapt it to include the case when there only exist nice cut-off functions with bounded energy (as specified in the statement above). In the proof below we refer to the cut-off functions (that equal to $1$\ on $U$\ and $0$\ on $V$) corresponding to some $C_1,C_2$\ with $C_2=C_1^{-\alpha}C(U,V)$\ in short as nice cut-off functions.
\begin{proof}
	For any fixed $\lambda>0$, any nice cut-off function $\phi$, consider the following perturbed semigroup
	\begin{eqnarray*}
		H_t^{\lambda \phi}f:=e^{-\lambda\phi}H_t\left(e^{\lambda\phi}f\right).
	\end{eqnarray*}
	For any $f,g\in L^2(X)$\ with $\mbox{supp}\{f\}\subset U$, $\mbox{supp}\{g\}\subset V$, first observe that
	\begin{eqnarray}
	\left|\left<H_t^{\lambda\phi}f,\,g\right>\right|= e^\lambda\left|\left<H_tf,\,g\right>\right|.\label{L2Gaussian00}
	\end{eqnarray}
	On the other hand, 
	\begin{eqnarray*}
		\left|\left<H_t^{\lambda \phi}f,\, g\right>\right|\leq \left|\left|H_t^{\lambda\phi}f\right|\right|_{L^2}\left|\left|g\right|\right|_{L^2}.
	\end{eqnarray*}
	We estimate $\left|\left|H_t^{\lambda\phi}f\right|\right|_{L^2}$\ by looking at its (square's) time derivative first.
	\begin{eqnarray}
	\lefteqn{\frac{d}{dt}\left(\left|\left|H_t^{\lambda \phi}f\right|\right|_{L^2}^{2}\right)
		=\int_X 2\left(H_t^{\lambda \phi}f\right)\frac{d}{dt}H_t^{\lambda \phi}f\,dm}\notag\\
	&=&\int_X 2\left(H_t^{\lambda \phi}f\right)e^{-\lambda \phi}\frac{d}{dt}H_t\left(e^{\lambda \phi}f\right)dm=-2\mathcal{E}\left(e^{-\lambda \phi}H_t^{\lambda \phi}f,\, e^{\lambda\phi}H_t^{\lambda \phi}f\right)\notag\\
	&=& -2\mathcal{E}\left(H_t^{\lambda \phi}f,\,H_t^{\lambda \phi}f\right)+2\lambda^2\int_X\left(H_t^{\lambda \phi}f\right)^{2}d\Gamma(\phi,\phi).\label{L2Gaussian0}
	\end{eqnarray}
	Since $\phi$\ is a nice cut-off function associated with $C_1$, $C_2$, we have 
	\begin{eqnarray*}
		\lefteqn{\int_X\left(H_t^{\lambda \phi}f\right)^{2}d\Gamma(\phi,\phi)}\\
		&\leq& C_1\int_X\phi^2\,d\Gamma\left(H_t^{\lambda \phi}f,\,H_t^{\lambda \phi}f\right)+C_2\int_{\scaleto{\mbox{supp}\{\phi\}}{6pt}}\left(H_t^{\lambda \phi}f\right)^2dm\\
		&\leq& C_1\mathcal{E}\left(H_t^{\lambda \phi}f,\,H_t^{\lambda \phi}f\right)+C_2\int_{\scaleto{\mbox{supp}\{\phi\}}{6pt}}\left(H_t^{\lambda \phi}f\right)^2dm.
	\end{eqnarray*}
	Substituting this bound back to (\ref{L2Gaussian0}) gives
	\begin{eqnarray*}
		\lefteqn{\frac{d}{dt}\left(\left|\left|H_t^{\lambda \phi}f\right|\right|_{L^2}^{2}\right)
			=-2\mathcal{E}\left(H_t^{\lambda \phi}f,\ H_t^{\lambda \phi}f\right)+2\lambda^2\int_X\left(H_t^{\lambda \phi}f\right)^{2}d\Gamma(\phi,\phi)}\\
		&\leq& \left(-2+2\lambda^2C_1\right)\mathcal{E}\left(H_t^{\lambda \phi}f,\,H_t^{\lambda \phi}f\right)+2\lambda^2C_2\int_{\scaleto{\mbox{supp}\{\phi\}}{6pt}}\left(H_t^{\lambda \phi}f\right)^2dm.
	\end{eqnarray*}
	When $-2+2\lambda^2C_1\leq 0$\ ($C_1\leq \frac{1}{\lambda^2}$), we can drop the first term and get
	\begin{eqnarray*}
		\frac{d}{dt}\left(\left|\left|H_t^{\lambda \phi}f\right|\right|_{L^2}^{2}\right)\leq 2\lambda^2C_2\left|\left|H_t^{\lambda \phi}f\right|\right|_{L^2(X)}^2.
	\end{eqnarray*}
	Observe that at $t=0$, $\left|\left|H_t^{\lambda \phi}f\right|\right|_{L^2}^{2}\bigg|_{t=0}=\left|\left|f\right|\right|_{L^2}^2$, so Gronwall's inequality gives
	\begin{eqnarray*}
		\left|\left|H_t^{\lambda \phi}f\right|\right|_{L^2}^{2}\leq \left|\left|f\right|\right|_{L^2}^2 \exp{\left(2\lambda^2C_2t\right)}.
	\end{eqnarray*}
	Combining this with (\ref{L2Gaussian00}), we have
	\begin{eqnarray*}
		\left|\left<H_tf,g\right>\right|\leq e^{-\lambda}\left|\left|H_t^{\lambda \phi}f\right|\right|_{L^2(X)}\left|\left|g\right|\right|_{L^2(X)}
		\leq \left|\left|f\right|\right|_{L^2}\left|\left|g\right|\right|_{L^2(X)} \exp{\left(-\lambda+\lambda^2C_2t\right)}.
	\end{eqnarray*}
	Take $\phi$\ corresponding to $C_1=\frac{1}{\lambda^2}$\ and let
	\begin{eqnarray*}
		\lambda=\left(\frac{1}{2C(U,V)t}\right)^{\frac{1}{1+2\alpha}}.
	\end{eqnarray*}
	By $C_2=C_1^{-\alpha}C(U,V)$, $\lambda=2\lambda^2C_2t$, and
	\begin{eqnarray*}
		|<H_tf,g>|
		\leq \left|\left|f\right|\right|_{L^2}\left|\left|g\right|\right|_{L^2(X)}\exp{\left\{-\frac{1}{2}\left(\frac{1}{2C(U,V)t}\right)^{\frac{1}{1+2\alpha}}\right\}}.
	\end{eqnarray*}
\end{proof}
\begin{remark}
When $C_2$\ has the more explicit dependence $C_2(C_1,U,V)=C_1^{-\alpha}d_X(U,V)^{-\beta}$\ for some $\alpha,\beta>0$\ and some distance $d_X$\ on $X$\ that defines the same topology, substituting $C(U,V)=d_X(U,V)^{-\beta}$\ in the above $L^2$\ Gaussian bound gives the $L^2$\ version of the sub-Gaussian upper bound. For example, for fractals with walk dimension $d_w$, $C_2\sim C_1^{1-\frac{d_w}{2}}d_X(U,V)^{-d_w}$\ (cf. \cite{anomalous}), then in our expression, $\alpha=\frac{d_w}{2}-1$, $\beta=d_w$, and the exponential term in the upper bound for $\left|\left<H_tf,g\right>\right|$\ is $\exp{\left\{-c\left(\frac{d_X(U,V)^{d_w}}{4t}\right)^{\frac{1}{d_w-1}}\right\}}$\ for some $c>0$.
\end{remark}
Next we estimate $\left|\left<\partial_t^kH_tf,g\right>\right|$. The estimate essentially follows from a straightforward adaptation of Proposition 2.2 in \cite{Coulhon}. For another approach on obtaining estimates on time derivatives of $<H_tf,g>$, cf. \cite{DaviesnonGaussian}. We first record a lemma.
\begin{lemma}
	\label{L2Gaussian2}
	Suppose that $F$\ is an analytic function on $\mathbb{C}_+$. Assume that, for given numbers $A,B,\gamma>0$,
	\begin{eqnarray*}
		|F(z)|\leq B,\ \forall z\in \mathbb{C},
	\end{eqnarray*}
	and for some $0<a\leq 1$,
	\begin{eqnarray*}
		|F(t)|\leq Ae^{at}e^{-\left(\frac{\gamma}{t}\right)^a},\ \forall t\in \mathbb{R}_+.
	\end{eqnarray*}
	Then
	\begin{eqnarray}
	|F(z)|\leq B\exp{\left(-Re\left[\left(\frac{\gamma}{z}\right)^a\right]\right)},\ \forall z\in \mathbb{C}_+.
	\end{eqnarray}
\end{lemma}
When $a=1$, this is exactly Proposition 2.2 in \cite{Coulhon}, and the proof for Lemma \ref{L2Gaussian2} is close to that of the proposition in \cite{Coulhon}. Here we follow their use of the notation $\mathbb{C}_+$\ for the right half plane.
\begin{lemma}[$L^2$\ Gaussian upper bound]
	Under the hypotheses in Lemma \ref{L2Gaussian1}, for any $f,g\in L^2(X)$\ with $\mbox{supp}\{f\}\subset U$, $\mbox{supp}\{g\}\subset V$, where $U,V$\ are precompact open sets with disjoint closures, 
	\begin{eqnarray}
	\left|\left<\partial_t^nH_tf,\,g\right>\right|\leq n!\frac{2^n}{t^n}\left|\left|f\right|\right|_{L^2}\left|\left|g\right|\right|_{L^2} \exp{\left\{-\frac{1}{2}\left(\frac{1}{3C(U,V)t}\right)^{\frac{1}{1+2\alpha}}\right\}}.
	\end{eqnarray}
\end{lemma}
\begin{proof} 
	Let $F(t):=<H_tf,g>$. By spectral calculus, for any $z\in \mathbb{C}$\ with $Re(z)>0$,
	\begin{eqnarray*}
		H_zv=\int_{0}^{+\infty}e^{-z\lambda}dE_\lambda v
	\end{eqnarray*}
	is well-defined for all $v\in L^2$, hence $F(z)$\ can be analytically extended to $z\in \mathbb{C}_+$. Moreover,
	\begin{eqnarray*}
		\left|\left|H_zf\right|\right|_{L^2}^2=\int_{0}^{\infty}e^{-2Re(z)\lambda}\,d(E_\lambda f,f)\leq \left|\left|f\right|\right|_{L^2}^2,
	\end{eqnarray*}
	so $F(z)$\ satisfies $|F(z)|\leq \left|\left|f\right|\right|_{L^2}\left|\left|g\right|\right|_{L^2}$. By Lemma \ref{L2Gaussian1}, for $t>0$,
	\begin{eqnarray*}
		|F(t)|\leq \exp{\left\{-\frac{1}{2}\left(\frac{1}{2C(U,V)t}\right)^{\frac{1}{1+2\alpha}}\right\}}\,\left|\left|f\right|\right|_{L^2}\left|\left|g\right|\right|_{L^2}.
	\end{eqnarray*}
	So by Lemma \ref{L2Gaussian2}, 
	\begin{eqnarray}
	\label{L2GaussianF(z)bound}
	|F(z)|\leq \left|\left|f\right|\right|_{L^2}\left|\left|g\right|\right|_{L^2} \exp{\left(-Re\left[\left( \frac{\gamma}{z}\right)^{\frac{1}{1+2\alpha}}\right] \right)},
	\end{eqnarray}
	where $\gamma=\frac{1}{4^{1+\alpha}C(U,V)}$.
	
	Recall that in complex analysis we have the expression for the nth derivative of $F(z)$\ using the integral over some circle around $z$,
	\begin{eqnarray}
	\label{Cauchyderivativeformula}
	F^{(n)}(z)=\frac{n!}{2\pi i}\int_{\mathcal{C}}\frac{F(\xi)}{(\xi-z)^{n+1}}\,d\xi= \frac{n!}{2\pi}\int_{0}^{2\pi}\frac{F(z+re^{i\theta})}{r^{n}e^{in\theta}}\,d\theta.
	\end{eqnarray}
	Consider $z=t\in \mathbb{R}_+$. Take for example $r=\frac{t}{2}$. Then (\ref{L2GaussianF(z)bound}) gives the bound
	\begin{eqnarray*}
		\left|F\left(t+\frac{t}{2}e^{i\theta}\right)\right|&\leq& \left|\left|f\right|\right|_{L^2}\left|\left|g\right|\right|_{L^2} \exp{\left(-Re\left[ \left( \frac{\gamma}{t+\frac{t}{2}e^{i\theta}}\right)^{\frac{1}{1+2\alpha}} \right] \right)}\\
		&\leq&\left|\left|f\right|\right|_{L^2}\left|\left|g\right|\right|_{L^2} \exp{\left\{-\left(\frac{2\gamma}{3t}\right)^{\frac{1}{1+2\alpha}}\right\}}.
	\end{eqnarray*}
	Substituting this bound in (\ref{Cauchyderivativeformula}), we get
	\begin{eqnarray}
	\label{L2Gaussianfull}
	\left|F^{(n)}(t)\right|=\left|\left<\partial_t^nH_tf,\,g\right>\right|\leq n!\frac{2^n}{t^n}\left|\left|f\right|\right|_{L^2}\left|\left|g\right|\right|_{L^2} \exp{\left\{-\left(\frac{2\gamma}{3t}\right)^{\frac{1}{1+2\alpha}}\right\}}.
	\end{eqnarray}	
\end{proof}
In the application of the Gaussian upper bound in the proofs in previous sections, the exact form of the upper bounds is not essential, we only need the property that the upper bound, divided by any power of $t$, tends to $0$\ as $t$\ tends to $0$. Hence we take Assumption \ref{L2gaussian} in previous sections.
\subsection{Other lemmas} In this subsection we prove Lemma \ref{cut-offannuli} on existence of nice cut-off functions for general pairs of open sets. Starting with the existence of nice cut-off functions for pairs in a topological basis $\mathcal{TB}$\ in the sense of Assumption \ref{assumptionbothcases}, we now construct nice cut-off functions for any pair of open sets $V\Subset U$ (Lemma \ref{cut-offannuli}).

In the next two lemmas we first discuss the properties of the sum and product of two nice cut-off functions. By taking maximum if necessary, we assume that all cut-off functions correspond to the same $C_1,C_2$.
\begin{lemma}[sum of nice cut-off functions] For any two nice cut-off functions $\eta_1,\eta_2$\ for some $V_1\Subset U_2$, $V_2\Subset U_2$, respectively, where $V_1,U_1,V_2,U_2$\ are open subsets of $X$, their sum $\eta:=\eta_1+\eta_2$\ is still a nice cut-off function satisfying
	\begin{eqnarray}
	\label{cut-offsum}
	\lefteqn{\int_X v^2\,d\Gamma(\eta_1+\eta_2,\,\eta_1+\eta_2)}\notag\\
	&\leq& 2C_1\int_X(\eta_1+\eta_2)^2\,d\Gamma(v,v)+4C_2\int_{\scaleto{\mbox{supp}\{\eta_1+\eta_2\}}{6pt}}v^2\,dm.
	\end{eqnarray}
\end{lemma}
\begin{proof}
	The energy measure $d\Gamma(\eta_1+\eta_2,\,\eta_1+\eta_2)$\ equals
	\begin{eqnarray*}
		d\Gamma(\eta_1+\eta_2,\,\eta_1+\eta_2)=d\Gamma(\eta_1,\eta_1)+2d\Gamma(\eta_1,\eta_2)+d\Gamma(\eta_2,\eta_2).
	\end{eqnarray*}
	Apply the Cauchy-Schwartz inequality (\ref{CSineq}), we get that for any $v\in \mathcal{F}$,
	\begin{eqnarray*}
		\lefteqn{\int_X v^2\,d\Gamma(\eta_1+\eta_2,\,\eta_1+\eta_2)
		\leq 2\int_X v^2\,d\Gamma(\eta_1,\eta_1)+2\int_Xv^2\,d\Gamma(\eta_2,\eta_2)}\\
		&\leq& 2C_1\int_X(\eta_1+\eta_2)^2\,d\Gamma(v,v)+4C_2\int_{\scaleto{\mbox{supp}\{\eta_1+\eta_2\}}{6pt}}v^2\,dm.
	\end{eqnarray*}
	The last line follows from that $\eta_1,\eta_2$\ are nice cut-off functions corresponding to $C_1,C_2$; $\eta_1,\eta_2\geq 0$; and $\mbox{supp}\{\eta_1\},\,\mbox{supp}\{\eta_2\}\subset \mbox{supp}\{\eta_1+\eta_2\}$.
\end{proof}
\begin{lemma}[product of nice cut-off functions]
	If $0\leq C_1<\frac{1}{4}$, for any two nice cut-off functions $\eta_1,\eta_2$\ for some $V_1\Subset U_2$, $V_2\Subset U_2$, respectively, where $V_1,U_1,V_2,U_2$\ are open subsets of $X$, the product function $\eta:=\eta_1\eta_2$\ is still a nice cut-off function satisfying
	\begin{eqnarray}
	\label{cut-offproduct}
	\lefteqn{\int_X v^2\,d\Gamma(\eta_1\eta_2,\,\eta_1\eta_2)}\notag\\
	&\leq& 16C_1\int_X(\eta_1\eta_2)^2\,d\Gamma(v,v)+8C_2\int_{\scaleto{\mbox{supp}\{\eta_1\eta_2\}}{6pt}}v^2\,dm.
	\end{eqnarray}
\end{lemma}
\begin{proof}
	By the product rule for the energy measure,
	\begin{eqnarray*}
		d\Gamma(\eta_1\eta_2,\,\eta_1\eta_2)=\eta_1^2\,d\Gamma(\eta_2,\eta_2)+2\eta_1\eta_2\,d\Gamma(\eta_1,\eta_2)+\eta_2^2\,d\Gamma(\eta_1,\eta_1).
	\end{eqnarray*}
    Then by the Cauchy-Schwartz inequality (\ref{CSineq}), for any $v\in \mathcal{F}$,
	\begin{eqnarray}
	\label{product1}
	\hspace{-.2in} \int_X v^2\,d\Gamma(\eta_1\eta_2,\,\eta_1\eta_2)
	\leq 2\int_Xv^2\eta_1^2\,d\Gamma(\eta_2,\eta_2)+2\int_Xv^2\eta_2^2\,d\Gamma(\eta_1,\eta_1),
	\end{eqnarray}
	and for any $\beta>0$,
	\begin{eqnarray*}
		\lefteqn{\int_Xv^2\eta_1^2\,d\Gamma(\eta_2,\eta_2)+\int_Xv^2\eta_2^2\,d\Gamma(\eta_1,\eta_1)}\\
		&\leq& C_1\left[\int_X\eta_2^2\,d\Gamma(\eta_1v,\,\eta_1v)+\int_X\eta_1^2\,d\Gamma(\eta_2v,\,\eta_2v)\right]+2C_2\int_{\scaleto{\mbox{supp}\{\eta_1\eta_2\}}{6pt}}v^2\,dm\\
		&\leq& C_1\left[2(1+\beta)\int_X\eta_1^2\eta_2^2\,d\Gamma(v,v)+\left(1+\frac{1}{\beta}\right)\int_X\eta_1^2v^2\,d\Gamma(\eta_2,\eta_2)\right.\\
		&&\left.+\left(1+\frac{1}{\beta}\right)\int_X\eta_2^2v^2\,d\Gamma(\eta_1,\eta_1)\right]+2C_2\int_{\scaleto{\mbox{supp}\{\eta_1\eta_2\}}{6pt}}v^2\,dm.
	\end{eqnarray*}
	So
	\begin{eqnarray*}
		\lefteqn{\left(1-C_1\left(1+\frac{1}{\beta}\right)\right)\left[\int_Xv^2\eta_1^2\,d\Gamma(\eta_2,\eta_2)+\int_Xv^2\eta_2^2\,d\Gamma(\eta_1,\eta_1)\right]}\\
		&\leq& 2C_1\left(1+\beta\right)\int_X\eta_1^2\eta_2^2\,d\Gamma(v,v)+2C_2\int_{\scaleto{\mbox{supp}\{\eta_1\eta_2\}}{6pt}}v^2\,dm.
	\end{eqnarray*}
	For $C_1<\frac{1}{4}$, we can take $\beta=1$, then $\frac{2C_1(1+\beta)}{1-C_1\left(1+\frac{1}{\beta}\right)}=\frac{4C_1}{1-2C_1}<8C_1$, and
	\begin{eqnarray}
	\label{product2}
	\lefteqn{\int_Xv^2\eta_1^2\,d\Gamma(\eta_2,\eta_2)+\int_Xv^2\eta_2^2\,d\Gamma(\eta_1,\eta_1)}\notag\\
	&\leq& 8C_1\int_X\eta_1^2\eta_2^2\,d\Gamma(v,v)+4C_2\int_{\scaleto{\mbox{supp}\{\eta_1\eta_2\}}{6pt}}v^2\,dm.
	\end{eqnarray}
	Combining (\ref{product1})\ and (\ref{product2}), we get (\ref{cut-offproduct}).
\end{proof}
To extend Assumption \ref{assumptionbothcases}, we use a construction similar to the standard construction of partitions of unity to obtain cut-off functions for general pairs of open sets and then check that the so-obtained functions satisfy (\ref{cut-offbddenergy}). We first state the following lemma on using open sets in the basis $\mathcal{TB}$\ to cover any compact set.
\begin{lemma}
	\label{topology}	
	For any compact set $K\subset X$\ and any open neighborhood $U$\ of $K$\ ($K\subset U\Subset X$), there exist two finite open covers $\mathcal{C}_1=\left\{U_1,U_2,\cdots, U_n\right\}$\ and $\mathcal{C}_2=\left\{V_1,V_2,\cdots, V_m\right\}$, such that all $U_j$, $V_i$\ are elements in $\mathcal{TB}$, $K\subset \cup_{i=1}^{m}V_i\subset \cup_{j=1}^nU_j\subset U$, and $\mathcal{C}_2$\ is subordinate to $\mathcal{C}_1$, i.e., for any $V_i\in \mathcal{C}_2$, there exists some $U_j\in \mathcal{C}_1$\ such that $V_i\Subset U_j$. 
\end{lemma}
\begin{proof}
	For any point $p\in K$, there exists an open neighborhood $U_p\in \mathcal{TB}$\ such that $p\in U_p\Subset U$\ since $\mathcal{TB}$\ is a topology basis and $X$\ is regular (to ensure there is some $U_p$\ that is precompact in $U$). Then $\left\{U_p\,\big|\,p\in K\right\}$\ is an open cover of $K$, which has a finite sub-cover $\mathcal{C}_1=\left\{U_{p_1},U_{p_2},\cdots, U_{p_n}\right\}$. We rename $U_{p_j}$\ as $U_j$.
	
	Now we construct $\mathcal{C}_2$\ from $\mathcal{C}_1$. For any point $p\in K$, there exists some $U_j$, $j=1,2,...,n$, such that $p\in U_j$. Then there exists some smaller open neighborhood $V_p\in \mathcal{TB}$\ such that $p\in V_p\Subset U_j$. $\left\{V_p\,\big|\,p\in K\right\}$\ is an open cover of $K$. Let $\left\{V_{p_1},V_{p_2},\cdots, V_{p_m}\right\}$\ be a finite sub-cover, then this gives the $\mathcal{C}_2$\ open cover we wanted, after renaming $V_{p_i}$\ as $V_i$.
	\end{proof}
Next we proceed to prove the lemma on the automatic extension of the applicability of Assumption \ref{assumptionbothcases} from pairs of open sets in a topological basis to all open sets.
\begin{proof}[Proof of Lemma \ref{cut-offannuli}]
	For any pair of open sets $V\Subset U$, for any $0<C_1<1$, we want to construct a nice cut-off function $\psi$\ for the pair $V\subset U$\ corresponding to $C_1$\ in (\ref{cut-offbddenergy}). Pick another open set $V'$\ such that $V\Subset V'\Subset U\Subset X$. Applying Lemma \ref{topology} to the compact set $K=\overline{V'}$\ with open neighborhood $U$, we get two finite open covers $\mathcal{C}_1=\left\{O_1,\cdots,O_n\right\}$\ and  $\mathcal{C}_2=\left\{\Omega_1,\cdots \Omega_m\right\}$\ such that $\mathcal{C}_2$\ is subordinate to $\mathcal{C}_1$, and that both cover $\overline{V'}$\ and are contained in $U$. Applying Lemma \ref{topology} to the compact set $\overline{U}\setminus V'$\ with open neighborhood $X\setminus \overline{V}$, we get two more finite open covers $\mathcal{C}'_1=\left\{O'_1,\cdots,O'_{n'}\right\}$\ and $\mathcal{C}'_2=\left\{\Omega'_1,\cdots \Omega'_{m'}\right\}$, such that $\mathcal{C}'_2$\ is subordinate to $\mathcal{C}'_1$, both $\mathcal{C}'_1,\mathcal{C}'_2$\ cover $\overline{U}\setminus V'$, and are contained in $X\setminus \overline{V}$.
	
	For any $0<C<1$, apply Assumption \ref{assumptionbothcases} to each pair $O_i\Subset \Omega_j$\ and $O'_i\Subset \Omega'_j$. Since all $\mathcal{C}_1,\mathcal{C}_2,\mathcal{C}'_1,\mathcal{C}'_2$\ are finite covers, there are finitely many nice cut-off functions $\left\{\eta_1,\cdots, \eta_r\right\}$\ and $\left\{\varphi_1,\cdots,\varphi_k\right\}$\ for pairs $O_i\Subset \Omega_j$\ and $O'_i\Subset \Omega'_j$, respectively, where all cut-off functions correspond to $C_1=C$\ in (\ref{cut-offbddenergy}). Let
	\begin{eqnarray*}
		\eta:=\eta_1+\cdots+\eta_r,\ \ \varphi:=\sum_{i=1}^{k}\varphi_i+\sum_{j=1}^{r}\eta_j.
	\end{eqnarray*}
	Then $1\leq \varphi\leq k+r$\ on $U$, and $\varphi=\eta$\ on $\overline{V}$, since all $\varphi_i$'s vanish on $\overline{V}$. Hence $\eta/\varphi$\ is well-defined on $U$\ and becomes $0$\ before it reaches the boundary of $U$\ since $\eta$\ is supported in $U$. By extending the quotient by $0$\ outside $U$, we obtain the function $\psi$\ satisfying 
	\begin{eqnarray*}
		\psi(x)=\begin{cases}
			\frac{\eta}{\varphi},\ x\in U,\\
			\\
			0,\ x\in U^c
		\end{cases}=\ \begin{cases}
			1,\ x\in \overline{V},\\
			\mbox{between $0$\ and $1$},\ x\in U\setminus \overline{V},\\
			0,\ x\in U^c.
		\end{cases}
	\end{eqnarray*}
	Hence it remains to show $\psi$\ satisfies (\ref{cut-offbddenergy}). By the lemmas on the sum and product of nice cut-off functions, we only need to show $1/\varphi$\ satisfies (\ref{cut-offbddenergy}) for $u\in \mathcal{F}$\ with support in $U$\ (since $\psi$\ is supported in $U$). For any $u\in \mathcal{F}$\ with support in $U$,
	\begin{eqnarray*}
		\lefteqn{\int u^2\,d\Gamma(1/\varphi,\,1/\varphi)
		= \int u^2\cdot \left(-\frac{1}{\varphi^2}\right)^2d\Gamma(\varphi,\varphi)}\\
		&\leq& \int u^2\,d\Gamma(\varphi,\varphi) \leq C'\int \varphi^2\,d\Gamma(u,u)+C_2\int_{\scaleto{\mbox{supp}\{\varphi\}}{6pt}}u^2\,dm,
	\end{eqnarray*}
	where $C'=2(k+r)C$\ is obtained from the lemma on sums of nice product functions and our definition of $\varphi$; $C_2$\ can be computed correspondingly. Moreover, since $1\leq \varphi\leq k+r$, $1\leq \varphi^2\leq (k+r)^2$, we get that $\varphi\leq (k+r)^2/\varphi$\ on $U$, hence 
	\begin{eqnarray*}
		\int u^2\,d\Gamma(1/\varphi,\,1/\varphi)
		\leq C'\int \frac{(k+r)^4}{\varphi^2}\,d\Gamma(u,u)+C_2\int_{\scaleto{\mbox{supp}\{\varphi\}}{6pt}}u^2\,dm,
	\end{eqnarray*}
	which is indeed of the form (\ref{cut-offbddenergy}). By picking a proper $C$, $\psi$\ would correspond to the given $C_1$\ in (\ref{cut-offbddenergy}).
\end{proof}

\bibliographystyle{plain}
\bibliography{hypl2}

\end{document}